\documentclass[12pt]{amsart}
\usepackage[margin=1.0in]{geometry}
\usepackage{amsmath,amsthm,amssymb,amsfonts,color,fancyhdr,graphics,enumerate,color,comment,mathtools,enumitem,psfrag,caption,bbm,bm}
\usepackage[hyperpageref]{backref}

\allowdisplaybreaks

\newtheorem{thm}{Theorem}[section]
\newtheorem{cor}[thm]{Corollary}
\newtheorem{lem}[thm]{Lemma}
\newtheorem{prop}[thm]{Proposition}
\newtheorem{conj}[thm]{Conjecture}

\newtheorem{defn}[thm]{Definition}
\newtheorem{rem}[thm]{Remark}

\newtheorem{application}[thm]{Application}

\numberwithin{equation}{section}

\definecolor{darkpastelpurple}{rgb}{0.59, 0.44, 0.84}
\definecolor{darkscarlet}{rgb}{0.34, 0.01, 0.1}
\usepackage[colorlinks=true, pdfstartview=FitV, linkcolor=blue, citecolor=blue, urlcolor=blue]{hyperref}
\usepackage[capitalise]{cleveref}

\newcommand{\be}{\begin{equation}}
\newcommand{\ee}{\end{equation}}
\newcommand{\bea}{\begin{eqnarray}}

\newcommand{\eea}{\end{eqnarray}}
\newcommand{\Bea}{\begin{eqnarray*}}
\newcommand{\Eea}{\end{eqnarray*}}

\newcommand{\vertiii}[1]{{\left\vert\kern-0.25ex\left\vert\kern-0.25ex\left\vert #1 \right\vert\kern-0.25ex\right\vert\kern-0.25ex\right\vert}}

\def\C{{\mathbb C}}

\def\N{{\mathbb N}}
\def\R{{\mathbb R}}

\def\Z{{\mathbb Z}}

\def\1{\text{\bf {1}}}

\DeclareMathOperator\supp{supp}

\begin{document}

\title[Grushin pseudo-multipliers]
{On $L^2$-boundedness of pseudo-multipliers associated to the Grushin operator} 

\author{Sayan Bagchi}
\address{Department of Mathematics and Statistics, Indian Institute of Science Education and Research Kolkata, Mohanpur--741246, West Bengal, India.}
\email{sayan.bagchi@iiserkol.ac.in}
\author{Rahul Garg}
\address{Department of Mathematics, Indian Institute of Science Education and Research Bhopal, Bhopal--462066, Madhya Pradesh, India.}
\email{rahulgarg@iiserb.ac.in}



\begin{abstract} 
In this article we define analogues of pseudo-differential operators associated to the joint functional calculus of the Grushin operator using their spectral resolution, and study Calder\'on--Vaillancourt-type theorems for these operators. 
\end{abstract}

\maketitle

\tableofcontents

\section{Introduction} \label{sec:intro} 

\subsection{Pseudo-differential operators} \label{subsec:intro-pseudo-euclidean} 
Given a bounded measurable function $ m $ on $ \mathbb{R}^{n} \times \mathbb{R}^{n} $, one considers the pseudo-differential operator $ m(x, D) $ defined by 
$$ m(x, D)f(x) := (2\pi)^{-n/2} \int_{\mathbb{R}^{n}} m(x, \xi) \widehat{f}(\xi) e^{i x \cdot \xi} \, d\xi, $$
for suitable functions $f$ on $\mathbb{R}^{n}$. Here $\widehat{f}$ stands for the Fourier transform of $f$ defined by 
$$ \widehat{f}(\xi) = (2\pi)^{-n/2} \int_{\mathbb{R}^{n}} f(x) e^{- i x \cdot \xi} \, dx. $$

Such a function $m$ is said to be the symbol of the pseudo-differential operator $m(x, D)$. When $m$ is independent of $x$-variable, we get the Fourier multiplier operator $m(D)$ and by the Plancherel theorem it follows that $m(D)$ is bounded on $L^2(\mathbb{R}^{n})$ if and only if $m \in L^\infty(\mathbb{R}^{n})$. However, it is well known that the boundedness of the function $m(x, \xi)$ alone may not be sufficient to guarantee the boundedness of the (densely defined) operator $m(x, D)$ on $L^2(\mathbb{R}^{n})$. A sufficient condition for the same is given by the celebrated Calder\'on--Vaillancourt theorem \cite{CalderonVaillancourt71} (see also Section $2.5$ of \cite{FollandHarmonicPhaseSpaceBook}). There is a vast literature on pseudo-differential operators, and we refer to \cite{SteinHarmonicBook93, TaylorPseudodifferentialBook81} which are among the classical references on this subject. In order to state this theorem we need to first recall symbol classes $S^{\sigma}_{\rho,\delta}$. For any $\sigma \in \mathbb{R}$ and $\rho, \delta \geq 0$ we define $S^{\sigma}_{\rho,\delta}$ to be the set consisting of all functions $m \in \C^\infty(\mathbb{R}^{n} \times \mathbb{R}^{n})$ which satisfy the following estimate for all $\alpha_1, \alpha_2 \in \mathbb{N}^{n}$: 
$$ \left| \partial_{x}^{\alpha_2} \partial_\xi^{\alpha_1} m(x, \xi)\right| \lesssim_{\alpha_1, \alpha_2} (1+|\xi|)^{\sigma - \rho |\alpha_1| + \delta |\alpha_2|}.$$ 

\begin{thm}[Calder\'on--Vaillancourt] \label{CV1}
Let $m \in S^0_{\rho, \delta}$, with $0 \leq \delta \leq \rho \leq 1, \, \delta \neq 1$. Then the operator $m(x, D)$ extends to a bounded operator from $L^2(\mathbb{R}^{n})$ to itself.
\end{thm} 

Since its inception, Calder\'on--Vaillancourt theorem has found several important applications in the study of partial differential equations. One such application is in the work of Beals--Fefferman \cite{Beals-Fefferman-Annals-1973}, where they studied a conjecture of Nirenberg--Trèves \cite{Nirenberg-Treves-First-local-Solvability-CPAM-1970,Nirenberg-Treves-Second-local-Solvability-CPAM-1970} concerning local solvability of linear partial differential equations. In \cite{Beals-Fefferman-Annals-1973}, Beals--Fefferman made a very crucial use of $L^2$-boundedness of pseudo-differential operators corresponding to symbols coming from the class $S^0_{\frac{1}{2}, \frac{1}{2}}$. 

Analogues of Pseudo-differential operators in the context of various Lie groups have been studied extensively in the literature where the  inversion formula for the group Fourier transform is used to define the pseudo-differential operators. In particular, for compact lie groups we refer to \cite{RuzhanskyTurunenPseudoCompactGroupBook} and the references therein, whereas for Heisenberg groups and more general graded Lie groups one can see \cite{BahouriFermanianGallagherPseudodifferentialHeisenberg, FischerRuzhanskyQuantizationBook} and the references therein. 

On the other hand, for self-adjoint operators, one can also make use of the spectral decomposition to define pseudo-differential operators. We call them pseudo-multipliers. Pseudo-multipliers have been studied in some of the recent works, see for example \cite{BagchiThangaveluHermitePseudo, AnhBuiHermitePseudoBesovTriebelLizorkin, BernicotFreyPseudodifferentialSemigroupOperators, CardonaRuzhanskyHermitePseudoArxiv, EppersonHermitePseudo, GeorgiadisNielsenPseudodifferentialSelfAdjointOperators, LyNaiboHermitePseudoBesovTriebelLizorkin}. But, none of these works address the analogue of the Calder\'on--Vaillancourt theorem in cases other than $S^0_{1, \delta}$, with $0 \leq \delta < 1$. 

In the present article, we restrict our attention to pseudo-multipliers associated to the Hermite operator $H = - \Delta + |x|^2$ on $\mathbb{R}^{n}$, and the Grushin operator $G = - \Delta_{x^{\prime}} - |x^{\prime}|^2 \Delta_{x^{\prime \prime}}$ and its joint functional calculus on $\mathbb{R}^{n_1 + n_2} $. We shall see that these pseudo-multipliers can be defined using their well-known spectral resolutions.  

We begin with Hermite pseudo-multipliers and state our main result for these operators in the following subsection. Results on pseudo-multipliers for the Grushin operator and its joint functional calculus will be discussed in Subsections \ref{subsec:intro-Grushin-pseudo} and \ref{subsec:intro-joint-functional-Grushin-pseudo} respectively. 


\subsection{Hermite pseudo-multipliers} 
\label{subsec:intro-Hermite-pseudo} 
Spectral resolution of the Hermite operator $H$ is known to be 
$$ H = \sum_{k=0}^\infty (2k + n) P_k, $$ 
where $P_k$ stand for the orthogonal projection of $L^2(\mathbb{R}^{n})$ onto the eigenspace (for the Hermite operator $H = - \Delta + |x|^2$ on $\mathbb{R}^{n}$) corresponding to the eigenvalue $2k + n$. 

Given $m \in L^\infty \left(\mathbb{R}^{n} \times \mathbb{R}_+ \right)$, it is natural to (densely) define the Hermite pseudo-multiplier on $L^2(\mathbb{R}^{n})$ by 
\begin{equation} \label{def:Hermite-pseudo} 
m(x, \sqrt{H}) f(x) := \sum_{k=0}^\infty m (x, \sqrt{2k + n}) P_k f(x). 
\end{equation}
It follows from Plancherel's theorem for the Hermite expansion that the set of all those functions $f$ in $L^2(\mathbb{R}^{n})$ for which $P_k f = 0$ except possibly for finitely many $k \in \mathbb{N}$ is dense in $L^2(\mathbb{R}^{n})$. It is straightforward to see that \eqref{def:Hermite-pseudo} is pointwise well-defined on this dense set.  

Hermite pseudo-multipliers were first introduced by Epperson \cite{EppersonHermitePseudo}. Assuming that such an operator is $L^2$-bounded, Epperson \cite{EppersonHermitePseudo} gave a sufficient condition for the $L^p$-boundedness in dimension $n = 1$. In \cite{BagchiThangaveluHermitePseudo}, the first author and Thangavelu studied the same problem in higher dimensions and obtained a result assuming lesser number of derivatives. We refer to \cite{AnhBuiHermitePseudoBesovTriebelLizorkin, CardonaRuzhanskyHermitePseudoArxiv, LyNaiboHermitePseudoBesovTriebelLizorkin} for recent works concerning analysis of Hermite pseudo-multipliers. Let us define symbol classes for the Hermite operator as follows. 
\begin{defn} \label{def:symbol-class-hermite}
For any $\sigma \in \mathbb{R}$ and $\rho, \delta \geq 0$, we say that 
$m \in C^\infty \left(\mathbb{R}^{n} \times \mathbb{R}_+ \right) $ belongs to the symbol class $S^{\sigma}_{\rho, \delta}(\sqrt{H})$ if it satisfies the following estimate: 
\begin{align}
\label{def-Hermite-symb}
\left| \frac{\partial^\alpha}{\partial {x}^\alpha} \frac{\partial^l}{\partial {\eta}^l} m(x, \eta) \right| \leq_{l, \alpha} (1 + \eta)^{\sigma- \rho\, l + \delta \, |\alpha|} 
\end{align} 
for all $l \in \mathbb{N}$ and $\alpha \in \mathbb{N}^{n}$. 
\end{defn} 

With the above definition of the symbol classes, we have the following analogue of the Calder\'on--Vaillancourt theorem. 
\begin{thm} \label{CVH}
Let $m \in S^0_{\rho,\delta} (\sqrt{H})$, with $0 \leq \delta \leq \rho \leq 1, \, \delta \neq 1$. Then the operator $m(x, \sqrt{H})$ extends to a bounded operator from $L^2(\mathbb{R}^{n})$ to itself.
\end{thm}

When $\rho = 1$, Theorem \ref{CVH} was also established in \cite{LyNaiboHermitePseudoBesovTriebelLizorkin}. Analogous results were studied in \cite{CardonaRuzhanskyHermitePseudoArxiv}, but with conditions that are different than the ones considered by us. We are not aware of any work in the literature concerning Theorem \ref{CVH} with $0 \leq \rho < 1$. 


\subsection{Grushin pseudo-multipliers} 
\label{subsec:intro-Grushin-pseudo} 
Denoting points of $\mathbb{R}^{n_1 + n_2}$ by $x = (x^{\prime}, x^{\prime \prime}) \in \mathbb{R}^{n_1} \times \mathbb{R}^{n_2}$, we now describe our results concerning the Calder\'on--Vaillancourt type theorem for pseudo-multipliers associated to the Grushin operator $G = - \Delta_{x^{\prime}} - |x^{\prime}|^2 \Delta_{x^{\prime \prime}}$. Given a Schwartz class function $f$ on $\mathbb{R}^{n_1 + n_2}$, let $f^\lambda(x^{\prime}) = \int_{\mathbb{R}} f (x^{\prime}, x^{\prime \prime}) e^{i \lambda \cdot x^{\prime \prime}} \, d x^{\prime \prime}$ denote (upto a dimensional constant multiple) the inverse Fourier transform of $f$ in $x^{\prime \prime}$-variable. Then, one can verify that 
$$ G f (x) = (2\pi)^{-n_2} \int_{\mathbb{R}^{n_2}} e^{-i \lambda \cdot x^{\prime \prime}} H(\lambda)f^{\lambda}(x^{\prime}) \, d\lambda, $$
where $H(\lambda) = - \Delta_{x^{\prime}} + |\lambda|^2 |x^{\prime}|^2$, for $\lambda \neq 0$, is the scaled Hermite operator on $\mathbb{R}^{n_1}$. 

Using the spectral decomposition of $H (\lambda)$ (see for example \cite{JotsaroopSanjayThangaveluRieszTransformsGrushin,ThangaveluHermiteLaguerreExpansionsBook}), one can write the spectral decomposition of the Grushin operator $G$ as follows: 
$$Gf (x) = (2\pi)^{-n_2} \int_{\mathbb{R}^{n_2}} e^{-i \lambda \cdot x^{\prime \prime}} \sum_{k \in \mathbb{N}} (2 k + n)|\lambda| \, P_k (\lambda) f^{\lambda}(x^{\prime}) \, d\lambda.$$ 

Given a bounded measurable function $m$ on $\mathbb{R}^{n_1 + n_2} \times \mathbb{R}_+$, one can then naturally define the Grushin pseudo-multiplier $m(x, \sqrt{G})$ by 
\begin{equation} \label{def:Gru-pseudo}
m (x, \sqrt{G}) f(x) := (2\pi)^{-n_2} \int_{\mathbb{R}^{n_2}} e^{-i \lambda \cdot x^{\prime \prime}} \sum_{k \in \mathbb{N}} m (x, \sqrt{(2 k + n)|\lambda|}) \, P_k (\lambda) f^{\lambda}(x^{\prime}) \, d\lambda, 
\end{equation}
for a suitable class of functions on $\mathbb{R}^{n_1 + n_2}$. In fact, it is easy to see that $m(x, \sqrt{G})$ is pointwise well-defined on the set $\displaystyle \mathcal{E} : = \cup_{l \in \mathbb{N}} \mathcal{E}_l$, where $\mathcal{E}_l$ are the following class of functions: 
\begin{align} \label{def-dense-set-Grushin-pseudo-multiplier}
\mathcal{E}_l := \left\{ \right. f \in L^2(\R^{n_1 + n_2}) : f^\lambda (x') = & \sum_{|\mu| \leq l} C(\lambda,\mu) \, \Phi_{\mu}^{\lambda}(x'), \text{ for a bounded function } C(\lambda, \mu)  \\ 
\nonumber & \text{which is compactly supported in } \lambda \text{-variable} \left. \right\}.
\end{align}
Here, $\Phi_\mu^\lambda$, for $\lambda \in \mathbb{R}^{n_2} \setminus \{0\}$, stand for scaled Hermite functions which are given by $\Phi_\mu^\lambda (x^{\prime}) = |\lambda|^{n_1/4} \Phi_\mu (|\lambda|^{1/2} x^{\prime})$. It is well-known that Hermite functions $\Phi_\mu$ are eigen functions of the Hermite operator $H$ with eigenvalue $(2|\mu| + n_1)$, from which it follows that $\Phi_\mu^\lambda$ are eigen functions of the scaled Hermite operator $H(\lambda) = - \Delta_{\mathbb{R}^{n_1}} + |\lambda|^2 |x^{\prime}|^2$ with eigenvalue $(2|\mu| + n_1) |\lambda|$. For various properties and identities related to Hermite functions, we refer to \cite{ThangaveluHermiteLaguerreExpansionsBook}. In view of the Plancherel theorem for the Grushin operator, the set $\mathcal{E}$ is dense in $L^2(\R^{n_1+n_2})$, which implies that Grushin pseudo-multipliers (as in \eqref{def:Gru-pseudo}) are densely defined on $L^2(\R^{n_1+n_2})$.  

When $m$ is independent of $x$-variable, the associated operator is referred as a Grushin multiplier and is denoted by $m  (\sqrt{G})$. Introduced in \cite{Grushin70}, the analysis of Grushin operators has attracted a lot of attention of harmonic analysts. In particular, we refer to \cite{CasarinoCiattiMartiniGrushinSphere19, JotsaroopSanjayThangaveluRieszTransformsGrushin, MartiniMullerGrushinRevistaMath, MartiniSikoraGrushinMRL, DuongOuhabazSikoraWeightedPlancherel2002JFA} and references therein for some of the important developments on the Grushin multipliers. 

Consider the following first order gradient vector fields: 
\begin{align} \label{first-order-grad}
X_{j} = \frac{\partial}{\partial x_j^{\prime}} \quad \textup{and} \quad X_{j,k} = x_j^{\prime} \frac{\partial}{\partial x_k^{\prime \prime}}, \quad \textup{for } 1 \leq j \leq n_1 \textup{ and }  1 \leq k \leq n_2.
\end{align}
Clearly, 
$$G = - \sum_{j=1}^{n_1} X^2_{j} -\sum_{j=1}^{n_1} \sum_{k=1}^{n_2} X^2_{j,k}.$$

Denoting the gradient vector field $(X_j, X_{l,k})_{1 \leq j, l \leq n_1, 1 \leq k \leq n_2}$ by $X$, we define the symbol classes $\mathcal{S}^{\sigma}_{\rho, \delta} (\sqrt{G})$ as follows.
\begin{defn} \label{grushin-symb-old-def1-sqrt}
For any $\sigma \in \mathbb{R}$ and $\rho, \delta \geq 0$, define the symbol class $\mathcal{S}^{\sigma}_{\rho, \delta} (\sqrt{G})$ to be the set consisting of all functions $m \in C^\infty \left(\mathbb{R}^{n_1 + n_2} \times \mathbb{R} \right)$ which satisfy the following estimate: 
\begin{align} \label{def-grushin-symb-old-sqrt-root}
\left|X^\Gamma \partial_{\eta}^l m(x, \eta) \right| \leq_{\Gamma, l} (1 + |\eta|)^{\sigma - \rho \, l + \delta \, |\Gamma|} 
\end{align} 
for all $\Gamma \in \mathbb{N}^{n_1 + n_1 n_2}$ and $l \in \mathbb{N}$. 
\end{defn}

With the above definition of the symbol class, following is our main result on Grushin pseudo-multipliers. 
\begin{thm} \label{old-main-thm-sqrt}
Let $m \in \mathcal{S}^0_{\rho,\delta} (\sqrt{G})$ for some $0 \leq \delta < \rho \leq 1$. Then the operator $m (x, \sqrt{G})$ extends to a bounded operator from $L^2(\mathbb{R}^{n_1 + n_2})$ to itself. 
\end{thm}

For technical convenience, it is better to work in the following equivalent context. Given $\widetilde{m} \in L^\infty \left( \mathbb{R}^{n} \times \mathbb{R}_+ \right)$, let us define $m  (x, \eta) := \widetilde{m} (x, \sqrt{\eta})$, so that 
\begin{align*} 
m \left(x, G \right) f(x) 
&:= (2\pi)^{-n_2} \int_{\mathbb{R}^{n_2}} e^{-i \lambda \cdot x^{\prime \prime}} \sum_{k \in \mathbb{N}} m (x, (2 k + n)|\lambda|) P_k (\lambda) f^{\lambda}(x^{\prime}) \, d\lambda, \\ 
&= (2\pi)^{-n_2} \int_{\mathbb{R}^{n_2}} e^{-i \lambda \cdot x^{\prime \prime}} \sum_{k \in \mathbb{N}} \widetilde{m} (x, \sqrt{(2 k + n)|\lambda|}) P_k (\lambda) f^{\lambda}(x^{\prime}) \, d\lambda \\ 
& = \widetilde{m} (x, \sqrt{G}) f(x). 
\end{align*} 
For $\eta$ away from $0$, a direct calculation gives that whenever $\widetilde{m} \in \mathcal{S}^{\sigma}_{\rho, \delta} (\sqrt{G})$, then 
$$ \left|X^\Gamma \partial_{\eta}^l m(x, \eta) \right| \leq_{\Gamma, l} (1 + \eta)^{\frac{\sigma}{2}- (1 + \rho) \frac{l}{2} + \delta \frac{|\Gamma|}{2}}, $$
for all $\Gamma \in \mathbb{N}^{n_1 + n_1 n_2}$ and $l \in \mathbb{N}$. 

We are thus led to define an analogous symbol class $\mathcal{S}^{\sigma}_{\rho, \delta} \left( G \right)$ as follows: 
\begin{defn} \label{grushin-symb-old-def1}
For any $\sigma \in \mathbb{R}$ and $\rho, \delta \geq 0$, define the symbol classes $\mathcal{S}^{\sigma}_{\rho, \delta} \left( G \right)$ to be the set consisting of all functions $m \in C^\infty \left(\mathbb{R}^{n_1 + n_2} \times \mathbb{R} \right)$ which satisfy the following estimate: 
\begin{align} \label{def-grushin-symb-old-without-sqrt-root}
\left|X^\Gamma \partial_{\eta}^l m(x, \eta) \right| \leq_{\Gamma, l} (1 + \eta)^{\frac{\sigma}{2}- (1 + \rho) \frac{l}{2} + \delta \frac{|\Gamma|}{2}},
\end{align} 
for all $\Gamma \in \mathbb{N}^{n_1 + n_1 n_2}$ and $l \in \mathbb{N}$. 
\end{defn}

And, with the above definition of the symbol classes $\mathcal{S}^{\sigma}_{\rho, \delta} \left( G \right)$, following is our main result on Grushin pseudo-multipliers $m(x, G)$. 
\begin{thm} \label{old-main-thm}
Let $m \in \mathcal{S}^0_{\rho,\delta} \left( G \right)$ for some $0 \leq \delta < \rho \leq 1$. Then the operator $m \left(x,  G \right)$ extends to a bounded operator from $L^2(\mathbb{R}^{n_1 + n_2})$ to itself. 
\end{thm}

Given any $N \in \mathbb{N}$, one can define the seminorm on symbol classes $\mathcal{S}^{\sigma}_{\rho, \delta} ( G )$ as follows:
\begin{align} \label{def-grushin-symb-seminorm}
\| m \|_{\mathcal{S}^{\sigma, N}_{\rho,\delta}} := \sup_{|\Gamma| + l \leq N} \sup_{x, \eta} \, (1 + \eta)^{- \frac{\sigma}{2} + (1 + \rho) \frac{l}{2} - \delta \frac{|\Gamma|}{2}} |X^\Gamma \partial_{\eta}^l m(x, \eta) |. 
\end{align} 
In fact, we shall prove that there exists an $N$, depending only on $n_1, n_2, \rho$, and $\delta$ such that for all symbols $m$ and for all $f \in L^2(\mathbb{R}^{n_1 + n_2})$, 
$$ \| m (x,  G ) f \|_{L^2(\mathbb{R}^{n_1 + n_2})} \lesssim_{n_1, n_2, \rho, \delta} \| m \|_{\mathcal{S}^{0, N}_{\rho,\delta}} \| f \|_{L^2(\mathbb{R}^{n_1 + n_2})}. $$
Our method does not seem to be robust enough to capture sharp $N$, so we do not pursue it here, and for the sake of brevity, we shall mostly use shortened notations $\| m \|_{\rho,\delta}$ or $\| m \|_{\mathcal{S}^0_{\rho,\delta}}$ where the number of derivatives $N$ should be clear from the context where it is used. 

Let us show the connection between symbol classes $\mathcal{S}^{\sigma}_{\rho, \delta} ( \sqrt{G} )$ and $\mathcal{S}^{\sigma}_{\rho, \delta} (G)$ and the associated Theorems \ref{old-main-thm-sqrt} and \ref{old-main-thm}. Note that if we take $m \in \mathcal{S}^{\sigma}_{\rho, \delta} (G)$, then $\widetilde{m} \in \mathcal{S}^{\sigma}_{\rho, \delta} (\sqrt{G})$, with obvious control on the seminorms. As a consequence, Theorem \ref{old-main-thm-sqrt} implies Theorem \ref{old-main-thm}. 

On the other hand, for a given $\widetilde{m} \in \mathcal{S}^{\sigma}_{\rho, \delta} (\sqrt{G})$, we can only ensure 
$$ |X^\Gamma \partial_{\eta}^l m(x, \eta) | \leq_{\Gamma, l} \eta^{\frac{\sigma}{2}- (1 + \rho) \frac{l}{2} + \delta \frac{|\Gamma|}{2}}.$$  
The singularity at the origin in the above estimate is notable. To overcome this obstacle, we can decompose the symbol function $m$ into two pieces in $\eta$-variable, say $m_{loc}$ and $m_{glob}$, with $m_{loc}$ supported near the origin and $m_{glob}$ supported away from the origin. It is clear that $\widetilde{m}_{glob} \in \mathcal{S}^{\sigma}_{\rho, \delta} (\sqrt{G})$ and $m_{glob} \in \mathcal{S}^{\sigma}_{\rho, \delta} (G)$ with equivalence of respective symbol seminorms. Thus, for $m_{glob}$, Theorem \ref{old-main-thm} implies Theorem \ref{old-main-thm-sqrt}. 

Finally, $L^2$-boundedness of operators $m_{loc} (x,  G)$ and $\widetilde{m}_{loc} (x, \sqrt{G})$ follows from Lemma \ref{lem:general-pseudo-multiplier-compact-support-L2-bounded} and Remark \ref{discussion-two-ways-definition-pseudo-multiplier} from where it will also become clear that operator norms of $m_{loc} (x,  G)$ and $\widetilde{m}_{loc} (x, \sqrt{G})$ are controlled by suitable seminorm of the symbol function $\widetilde{m}$. 

Hence, Theorems \ref{old-main-thm-sqrt} and \ref{old-main-thm} are equivalent.

\begin{rem}
\label{rem:discussion-pho-neq-delta-issue} 
Note that Theorems \ref{old-main-thm-sqrt} and \ref{old-main-thm} do not talk about the symbol class $\mathcal{S}^0_{\rho,\rho} (G)$ for any $0 \leq \rho < 1.$ In fact, our proof is built on analysing the action of the distance function on the kernel of the operator. For the same, we estimate $\lambda$-derivatives of the kernels of certain scaled Hermite operators. But, these derivatives may very well have singularities at $\lambda = 0$. This obstacle does not allow us to treat the case when $0 \leq \delta = \rho < 1.$ Nonetheless, we are able to handle the singularities (of $\lambda$ near $0$) in all other cases, that is, when $0 \leq \delta < \rho \leq 1$. 

We shall show that if one studies the analogous question for symbols associated to the joint functional calculus of the Grushin operator, then assuming some  cancellation condition for the symbol function in the $\lambda$-variable, one can establish the $L^2$-boundedness in the full range (that is, for $0 \leq \delta = \rho < 1$ as well). 
\end{rem} 

We explain the set-up and assumptions pertaining to the above mentioned joint functional calculus in the following subsection. 


\subsection{Pseudo-multipliers for the joint functional calculus of the Grushin operator} 
\label{subsec:intro-joint-functional-Grushin-pseudo}
Let us consider the following operators associated with the Grushin operator 
\begin{align} \label{eq:operatorsLandU}
L_j= (-i \partial_{x_j'})^2 + (x_j')^2 \sum_{k=1}^{n_2} (-i\partial_{x_k''})^2 \quad \textup{and} \quad U_k = - i \partial_{x_k''}, 
\end{align} 
for $j= 1, 2, \ldots, n_1$ and $k = 1, 2, \ldots, n_2$. 

As explained in Section 3 of \cite{MartiniSikoraGrushinMRL}, the operators $L_1, L_2, \ldots, L_{n_1}, U_1, U_2, \ldots, U_{n_2}$ (and therefore all the polynomials in $L_1, L_2, \ldots, L_{n_1}$, $U_1, U_2, \ldots, U_{n_2}$) are essentially self-adjoint on $C_c^\infty(\mathbb{R}^{n_1 + n_2})$ and commute strongly (that is, their spectral resolutions commute), and therefore they admit a joint functional calculus on $L^2(\mathbb{R}^{n_1 + n_2})$ in the sense of the spectral theorem. Let us write $\boldsymbol{L} = (L_1, L_2,\ldots, L_{n_1})$ and $\boldsymbol{U} = (U_1, U_2, \ldots, U_{n_2})$. 

Given $m \in L^\infty ( \mathbb{R}^{n_1 + n_2} \times (\mathbb{R}_+)^{n_1} \times (\mathbb{R}^{n_2} \setminus \{0\}) )$, we consider pseudo-multiplier operator $m(x, \boldsymbol{L}, \boldsymbol{U})$, defined on a suitable class of functions on $\mathbb{R}^{n_1 + n_2}$ by 
\begin{align} \label{def:joint-Gru-pseudo}
m(x, \boldsymbol{L}, \boldsymbol{U}) f(x) := (2\pi)^{-n_2} \int_{\mathbb{R}^{n_2}} e^{-i \lambda \cdot x^{\prime \prime}} \sum_{\mu \in \mathbb{N}^{n_1}} m ( x, (2 \mu + \tilde{1})|\lambda|, - \lambda ) ( f^{\lambda}, \Phi_\mu^\lambda ) \Phi_\mu^\lambda (x^{\prime}) \, d\lambda, 
\end{align} 
where $\tilde{1} = (1, 1, \ldots, 1) \in \mathbb{R}^{n_1}$. 

We are naturally led to the following symbol classes. 
\begin{defn} \label{def:joint-symb-def1}
For any $\sigma \in \mathbb{R}$ and $\rho, \delta \geq 0$, define the symbol class $\mathcal{S}^{\sigma}_{\rho_1, \rho_2, \delta}(\boldsymbol{L}, \boldsymbol{U})$ to be the set of all $m \in C^\infty (\mathbb{R}^{n_1 + n_2} \times (\mathbb{R}_+)^{n_1} \times (\mathbb{R}^{n_2} \setminus \{0\}) )$ satisfying the following estimate: 
\begin{align} \label{def:grushin-symb}
|X^\Gamma \partial_{\tau}^{\theta} \partial_{\kappa}^{\beta} m(x, \tau, \kappa) | \leq_{\Gamma, \theta, \beta} (1 + |\tau| + |\kappa|)^{\frac{\sigma}{2} - (1 + \rho) \frac{(|\theta| + |\beta|)}{2} + \delta \frac{|\Gamma|}{2}} 
\end{align} 
for all $\Gamma \in \mathbb{N}^{n_1 + n_1 n_2}$, $\theta \in \mathbb{N}^{n_1}$ and $\beta \in \mathbb{N}^{n_2}$. 
\end{defn}

One can define seminorms for above symbol classes as in \eqref{def-grushin-symb-seminorm}. Though as mentioned earlier, we shall use the notation $\| m \|_{\rho, \delta}$ or $\| m \|_{\mathcal{S}^0_{\rho,\delta}}$ for convenience. 

\begin{thm} \label{thm:old-joint-calc}
Let $m \in \mathcal{S}^0_{\rho, \delta} (\boldsymbol{L}, \boldsymbol{U})$ for some $\delta, \rho \in [0,1]$ with $1 \neq \delta \leq \rho$. 
\begin{enumerate}
\item If $\delta < \rho$, then $m(x, \boldsymbol{L}, \boldsymbol{U})$ extends to a bounded operator from $L^2(\mathbb{R}^{n_1 + n_2})$ to itself. 

\item For $m \in \mathcal{S}^0_{\rho, \rho} (\boldsymbol{L}, \boldsymbol{U})$, with $0 \leq \rho < 1$, if we further assume that 
\begin{align} \label{def:grushin-symb-vanishing-0-condition}
\lim_{\kappa \to 0} \partial_{\kappa}^{\beta} m(x, \tau, \kappa) = 0, 
\tag{CancelCond} 
\end{align} 
for all $\beta \in \mathbb{N}^{n_2}$ with $|\beta| \leq 4 N_0 = 4 ( \lfloor\frac{Q}{4}\rfloor + 1 )$, then $m(x, \boldsymbol{L}, \boldsymbol{U})$ extends to a bounded operator from $L^2(\mathbb{R}^{n_1 + n_2})$ to itself. Here $Q = n_1 + 2 n_2$ is the homogeneous dimension associated to the Grushin operator. 
\end{enumerate}
\end{thm}

\begin{rem} \label{rem:joint-symb-lesser-assumption}
Note that in the definition of the operator $m(x, \boldsymbol{L}, \boldsymbol{U})$, the symbol function $m$ is evaluated at points of the form $(x, \tau, \kappa) = (x, (2 \mu + \tilde{1})|\lambda|, -\lambda)$, and therefore, in our context, we have that $|(\tau, \kappa)| \sim |\tau|$. 

Now, even though our proof of part $(2)$ of Theorem \ref{thm:old-joint-calc}  requires a cancellation condition in the form of \eqref{def:grushin-symb-vanishing-0-condition}, but it also gives us the required result with the ``genuinely" lesser assumption. Namely, we do not need any decay via $\kappa$-derivatives, and only need to assume the following condition on the symbol function $m$: 
\begin{align} \label{ineq:joint-symb-lesser-assumption-2}
|X^\Gamma \partial_{\tau}^{\theta} \partial_{\kappa}^{\beta} m(x, \tau, \kappa) | \leq_{\Gamma, \theta, \beta} (1 + |\tau|)^{- (1 + \rho) \frac{|\theta|}{2} + \rho \frac{|\Gamma|}{2}}
\end{align} 
for sufficiently many $\Gamma \in \mathbb{N}^{n_1 + n_1 n_2}$, $\theta \in \mathbb{N}^{n_1}$ and $\beta \in \mathbb{N}^{n_2}$. 
\end{rem}

\begin{application}
We discuss here the following interesting applications of Theorem \ref{thm:old-joint-calc} (and its proof) and Remark \ref{rem:joint-symb-lesser-assumption} . 

\begin{enumerate}[label=(\roman*)]
\item One can prove a Calder\'on-Vaillancourt type theorem even in the presence of Hermite-type shifts. More precisely, for $m \in \mathcal{S}^0_{\rho, \delta}(G)$ with $\delta < \rho$, and $\vec{c} \in \mathbb{R}^{n_1}$, let us define symbol functions $F : \mathbb{R}^{n_1 + n_2} \times (\mathbb{R}_+)^{n_1} \to \mathbb{C}$ and $M_{\vec{c}} : \mathbb{R}^{n_1 + n_2} \times (\mathbb{R}_+)^{n_1} \times (\mathbb{R}^{n_2} \setminus \{0\}) \to \mathbb{C}$ by 
$F (x, \tau) = m (x, |\tau|_1)$, and $M_{\vec{c}} (x, \tau, \kappa) = F ( x, \tau + |\kappa| \vec{c} )$. We show in Lemma 2.7 of \cite{BBGG-2} that given a compact set $\mathcal{K} \subset \mathbb{R}^{n_1}$, there exists a constant $C = C_{\mathcal{K}}$ such that for all $\vec{c} \in \mathcal{K}$, 
\begin{align*} 
\| M_{\vec{c}} (x, \boldsymbol{L}, \boldsymbol{U}) \|_{op} \leq C_{\mathcal{K}} \| m \|_{\mathcal{S}^0_{\rho, \delta}}. 
\end{align*}
 
\item Consider any symbol of the form  
$$ m(x, \tau, \kappa) = (\sin \kappa_j)^{4 ( \lfloor\frac{Q}{4}\rfloor + 1 )} \, F (x, |\tau|_1) $$ 
for some $1 \leq j \leq n_2$ and $F \in S^0_{\rho, \rho}(G)$ with $0 \leq \rho < 1$, then the operator $m(x, \boldsymbol{L}, \boldsymbol{U})$ is $L^2$-bounded. 

\item Given a measurable subset $\mathcal{U} \subseteq \mathbb{R}^{n_2}$, let us denote by $L^2_{\mathcal{U}} (\mathbb{R}^{n_1 + n_2})$ the collection of all functions $f \in L^2(\mathbb{R}^{n_1 + n_2})$ such that $f^\lambda(x')$ is supported in $\mathbb{R}^{n_1} \times \mathcal{U}$ as a function of $(x', \lambda)$. We have the following result. 
\begin{cor} \label{cor:Grushin-CV-compact-support} 
Let $\mathcal{U} \subseteq \mathbb{R}^{n_2}$ be a closed set such that $0 \notin \mathcal{U}$. Then, for any $m_1 \in \mathcal{S}^0_{\rho, \rho} (G)$ and $m_2 \in \mathcal{S}^0_{\rho, \rho} (\boldsymbol{L}, \boldsymbol{U})$ with $0 \leq \rho < 1$, we have 
\begin{align*}
\| m_1(x, G) f \|_2 & \lesssim_{\mathcal{U}} \| m_1 \|_{\rho, \rho} \|f\|_2, \\ 
\| m_2(x, \boldsymbol{L}, \boldsymbol{U}) f \|_2 & \lesssim_{\mathcal{U}} \| m_2 \|_{\rho, \rho} \|f\|_2 
\end{align*}
for all $f \in L^2_{\mathcal{U}} (\mathbb{R}^{n_1 + n_2})$. 
\end{cor}
\end{enumerate}
\end{application}

Although we require cancellation condition of the type \eqref{def:grushin-symb-vanishing-0-condition} in our proof of Theorem \ref{thm:old-joint-calc} (2), we believe that the conclusion of Theorem \ref{thm:old-joint-calc} should in fact hold true without assuming such a condition. 
\begin{conj} \label{conjecture:S-0-0-rho}
Let $m \in \mathcal{S}^0_{\rho, \rho} (\boldsymbol{L}, \boldsymbol{U})$ for some $0 \leq \rho < 1$. Then the operator $m(x, \boldsymbol{L}, \boldsymbol{U})$ extends to a bounded operator from $L^2(\mathbb{R}^{n_1 + n_2})$ to itself. 
\end{conj}

As an application of our techniques, we also establish the following Weighted Plancherel estimates for the joint functional calculus of $(\boldsymbol{L}, \boldsymbol{U})$. For a compactly supported, bounded Borel measurable function $m : \mathbb{R}^{n_1} \times \mathbb{R}^{n_2} \to \mathbb{C}$, let $k_{m(\boldsymbol{L}, \boldsymbol{U})}$ denote the kernel of the operator $m(\boldsymbol{L}, \boldsymbol{U})$. 

\begin{thm} \label{thm:weighted-Plancherel-L2} 
Let $2 \leq p \leq \infty$ and $r_0 \in \mathbb{N}$. For all $R > 0$ and $0 \leq r \leq 4 \lfloor{ \frac{r_0}{4} \rfloor}$ we have 
\begin{align} \label{ineq:weighted-Plancherel-L2} 
\left\| \left| B(\cdot, R^{-1}) \right|^{\frac{1}{2} - \frac{1}{p}} \left( 1 + R d(x, \cdot) \right)^{r} k_{m( \boldsymbol{L}, \boldsymbol{U} )} (x, \cdot) \right\|_p \lesssim_{p,r} \left| B(x, R^{-1}) \right|^{-1/2} \left\| m \left(R^2 \cdot \right) \right\|_{W^r_{\infty}}. 
\end{align} 
for any bounded $C^r$ function $m : \mathbb{R}^{n_1} \times \mathbb{R}^{n_2} \to \mathbb{C}$ such that $\supp m \subseteq [-R^2,  R^2]^{n_1 + n_2},$ and $\lim_{\kappa \to 0} \partial_\kappa^{\beta} m(\tau, \kappa) = 0$ for all $|\beta| \leq r_0.$ 
\end{thm}

The conclusion of Theorem \ref{thm:weighted-Plancherel-L2} is known to be true without assuming the cancellation condition of the type $\lim_{\kappa \to 0} \partial_\kappa^{\beta} m(\tau, \kappa) = 0$ (see \cite{AnhBuiDuongSpectralMultipliersBesovTriebelLizorkin, DuongOuhabazSikoraWeightedPlancherel2002JFA} for such estimates for a single self-adjoint operator, and \cite{MartiniJointFunctionalCalculiMathZ} for the joint functional calculus) with the Sobolev norm $W^{r+ \epsilon}_\infty$ in the right hand side of the estimate with $\epsilon > 0$. The main novelty of our result is the removal of this extra $\epsilon > 0$ in the specific case of the joint functional calculus of the Grushin operator on $\mathbb{R}^{n_1 + n_2}$. In fact, building on the above mentioned kernel estimates along with Theorem \ref{thm:weighted-Plancherel-L2}, we are able to establish some sharp $L^p$ weighted estimates for a class of Grushin pseudo-multipliers. These results will appear in a forthcoming work. 

In future, we are interested in studying applications of Hermite and Grushin pseudo-multipliers to the associated partial differential equations.


\subsection{Methodology of the proof} \label{subsec:intro-methodology-proofs} 
We got an inspiration from the recent work of Fisher-Ruzhansky \cite{FischerRuzhanskyQuantizationBook}, where they have established Calder\'on--Vaillancourt-type theorems on graded Lie groups for operators corresponding to symbols defined in terms of the Rockland operators (see Chapter 5 of \cite{FischerRuzhanskyQuantizationBook}). However, in our context, the absence of the group structure and the behaviour of the optimal control distance function corresponding to the Grushin operator make the analysis considerably difficult. 

When $\delta = 0$ and $\rho > 0$, we prove part $(1)$ of Theorem \ref{thm:old-joint-calc} by making use of the known kernel estimates of Grushin multipliers together with a minor variant of the classical technique of proving a similar result for the pseudodifferential operators on the Euclidean space. 

The case of $\rho = 0$ is much harder. Note that there is not enough decay on the symbols in this case and therefore we can not directly apply the known kernel estimates from the literature. This leads us to develop some new kernel estimates. We do so in Proposition \ref{prop:lambda-diff-theorem}. We have to also establish a smooth partition of unity compatible with the Grushin space, and the same is done in Lemma \ref{lem-grushin-partition-part-fix}. Thereafter, we prove our kernel estimates in Proposition \ref{prop:lambda-diff-theorem} and Lemma \ref{first-layer-lem}. While summing the involved operator norms with the help of our kernel estimates, we encounter a serious obstacle of singularities at $\lambda=0$. This forces us to further assume a cancelling condition of the type \eqref{def:grushin-symb-vanishing-0-condition}. 

For $\delta>0$ case in Theorem \ref{thm:old-joint-calc}, we again go through a delicate analysis of the kernels, making full use of Lemma \ref{first-layer-lem} together with some of the weighted Plancherel estimates that are already known in the literature.  


\subsection{Organisation of the paper} \label{subsec:intro-paper-organisation}
\begin{itemize}
\item We discuss the aspect of defining pseudo-multipliers via the spectral theory, and its relation with the definition of pseudo-multipliers via spectral resolution (as in \eqref{def:Hermite-pseudo} and \eqref{def:Gru-pseudo}) in Subsection \ref{subsec:pseudo-mult-via-spectral-theory}. Afterwards, we recall all the relevant preliminaries related to the Grushin operator and also establish some basic results that we require in later subsections. This includes the Sobolev embedding theorem and construction of a partition of unity associated to the Grushin operator, and some weighted Plancherel estimates for the joint functional calculus of $(\boldsymbol{L}, \boldsymbol{U})$. 

\item We study the $L^2$-boundedness in the case of $S^0_{\rho, 0}(\boldsymbol{L}, \boldsymbol{U}), \, \rho > 0$ in Section \ref{sec:rho-0-proof}. 

\item We prove our main kernel estimates in Section \ref{sec-kernel-estimates}. 

\item Building on the kernel estimates of Section \ref{sec-kernel-estimates}, we prove Theorem \ref{thm:weighted-Plancherel-L2} in Section \ref{Sec:weighted-Planch}. 

\item The case of  $L^2$-boundedness for $\mathcal{S}^0_{0, 0}(\boldsymbol{L}, \boldsymbol{U})$ class of symbols is developed in Sections \ref{Sec-S00-proof}, \ref{Sec-S00-proof-1}, \ref{Sec-S00-proof-2}. We sketch the key part of the proof leading to Corollary \ref{cor:Grushin-CV-compact-support} in Remark \ref{rem:proof-rem-Grushin-CV-compact-support}. 

\item Finally, the case of $S^0_{\rho, \rho}(\boldsymbol{L}, \boldsymbol{U}), \, 0 \leq \delta \leq \rho \leq 1, \delta \neq 1$, is studied in Section \ref{sec:rho-delta}, establishing Theorem \ref{thm:rho-delta}. Theorem \ref{thm:old-joint-calc} (1) follows from Theorems \ref{thm:S0-rho-0-L2} and \ref{thm:rho-delta}, whereas Theorem \ref{thm:old-joint-calc} (2) follows from the $\mathcal{S}^0_{0, 0}(\boldsymbol{L}, \boldsymbol{U})$ case and Theorem \ref{thm:rho-delta}. 

\item It is not much difficult to verify that one can prove Theorem \ref{old-main-thm-sqrt} by modifying the proof of Theorem \ref{thm:old-joint-calc} (2), so we omit it's proof. Finally, it will become apparent from these proofs that a proof of Theorem \ref{CVH} (Calder\'{o}n-Vaillancourt theorem for Hermite pseudo-multipliers) is much more simpler because of the absence of $x^{\prime \prime}$-variable and it can be established by performing some obvious modifications in the proof of Theorem \ref{thm:old-joint-calc}, so we leave the details of the proof of Theorem \ref{CVH} as well. 
\end{itemize}


\subsection{Notations and parameters} \label{app:notations} 
For ready reference, we list here most of the notations and parameters that we are going to frequently use. 
\begin{itemize} 
\item $\mathbb{N} = \{0, 1, 2, 3, \ldots\}$ and $\mathbb{R}_+ = [0, \infty)$. 

\item For $A, B >0$, the expression $A \lesssim B$ stands for $A \leq C B$ for some $C > 0$. We write $A \lesssim_{\epsilon} B$ when the implicit constant $C$ may depend on $\epsilon$. We use the notation $A \sim B$ for $A \lesssim B$ and $B \lesssim A$.

\item For a Hilbert space $H$, we denote by $\mathcal{B} ( H ) $ the Banach space of all bounded linear operators on $H$. We denote the operator norm of $T \in \mathcal{B} ( H )$ by $\| T \|_{op}$. 

\item We write $x = (x^{\prime}, x^{\prime \prime}) \in \mathbb{R}^{n_1} \times \mathbb{R}^{n_2} \, = \mathbb{R}^{n_1 + n_2}$. 

\item The homogeneous dimension of the space $\mathbb{R}^{n_1 + n_2}$ with respect to the Grushin operator is $Q = n_1 + 2 n_2$. 

\item Whenever it is obvious that a vector $h = (h_1, \ldots, h_{n_1})$ belongs to $\mathbb{N}^{n_1}$ (for any $n_1 \geq 2$), we denote by $|h|$ the $l^1$-sum of it's indices, that is, $|h| = \sum_{j=1}^{n_1} h_j$. On the other hand, if we are in a situation where $h$ could be a general element from $\mathbb{R}^{n_1}$, then we use the notations $|h| = ( \sum_{j=1}^{n_1} |h_j|^2 )^{1/2}$ and $|h|_1 = \sum_{j=1}^{n_1} |h_j|$. 

\item $J, J^{\prime}, l, l^{\prime}, j, j^{\prime}, \nu, \nu_j, N, N_j, L, L_j, q$ etc always represent elements of $\mathbb{N}$. 

\item $\mu, \tilde{\mu}, \mu^{\prime}, \tilde{\mu}^{\prime}, \alpha, \alpha_j, \gamma_j, \theta, \theta_j, \theta^{\prime}, \theta_j^{\prime}, \tilde{\theta}, \tilde{\tilde{\theta}}$ etc always represent elements of $\mathbb{N}^{n_1}$. 

\item $\beta, \beta_j, \beta^{\prime}$ etc always represent elements of $\mathbb{N}^{n_2}$. 

\item $\Gamma, \Gamma_j, \Gamma^{\prime}, \tilde{\Gamma} $ etc always represent elements of $\mathbb{N}^{n_1 + n_1 n_2}$. 

\item $\tilde{a}, \tilde{b}, \vec{a}, \vec{b}, \vec{c}$ etc denote  elements of $\mathbb{R}^{n_1}$. 

\item $\tau, \tau^{\prime}, \tilde{\tau}$ etc represent elements of $ ( \mathbb{R}_+ )^{n_1}$. Similarly, $\lambda, \lambda^{\prime}, \kappa, \kappa^{\prime}$ etc represent elements of $\mathbb{R}^{n_2}$. Moreover, 
$\partial_{\tau}^{\alpha}$, $\partial_\lambda^{\beta}$ and $\partial_\kappa^{\beta}$ stand for partial differential operators. 

\item $\vec{c}(s)$ denotes $(c_1(s), \ldots, c_{n_1}(s))$ with each $c_j(s)$ being a linear function on $[0,1]^{L}$, for an appropriate positive integer $L$. 

\item $\Omega$ denotes a compact subset of $\mathbb{R}^2$ which may be different at different occurrences. 

\item $\vec{a}(s, s^{\prime})$ stands for $(a_1(s, s^{\prime}), \ldots, a_{n_1}(s, s^{\prime}))$ with each $a_j(s, s^{\prime})$ being a real-valued linear function on $[0,1]^{\nu_1} \times \Omega^{\nu_2}$. In short we use $\vec{a}(w)$ and denote by $dw$ the Lebesgue measure on the corresponding space.

\item $g(w)$ always denote a continuous function on the compact set of integration.

\item $\Lambda$ is used for compact sets of the form $[0,1]^{\nu_1} \times \Omega^{\nu_2} \times [0,1]^{4 N}$. 
\item $\mathbbm{1}$ denotes the indicator function. 

\item For a multi-index $\alpha = (\alpha^{(j)})_{j=1}^{n_1}$, we write $\alpha! = \prod_{j = 1}^{n_1} \alpha^{(j)}!$ and $\tau^{\frac{1}{2} \alpha} = \prod_{j = 1}^{n_1} \tau_j^{\frac{1}{2} \alpha^{(j)}}$.

\item For multi-indices $\alpha_1 = (\alpha_1^{(j)})_{j=1}^{n_1}$ and $\alpha_2 = (\alpha_2^{(j)})_{j=1}^{n_1}$, the notation $\alpha_1 \leq \alpha_2$ stands for $\alpha_1^{(j)} \leq \alpha_2^{(j)}$ for all $1 \leq j \leq n_1$. 

\item $E_{\mu, \lambda} (x, y) = \Phi_\mu^{\lambda}(x^{\prime}) \, \Phi_\mu^{\lambda}(y^{\prime}) \, e^{- i \lambda \cdot ( x^{\prime \prime} - y^{\prime \prime} )}.$

\item While discussing the case of the Euclidean space or for the Hermite operator, we write $\mathbb{R}^n$ and use the indices as tabulated above with the convention of $n = n_1$. 
\end{itemize}


\section{Preliminaries and basic results} \label{sec:prelim}

\subsection{Pseudo-multipliers and spectral theory} \label{subsec:pseudo-mult-via-spectral-theory}

Let $\mathcal{M}$ be a $\sigma$-finite measure space and $\mathcal{L}$ a non-negative self-adjoint operator on $L^2(\mathcal{M})$. We know by spectral theory that $\mathcal{L}$ admits a spectral resolution 
$$ \mathcal{L} = \int_0^\infty \eta \, dE_{\mathcal{L}} (\eta),$$ 
and that given a bounded Borel function $m: [0, \infty) \to \mathbb{C}$, one can define the operator 
$$m ( \sqrt{\mathcal{L}} ) := \int_0^\infty m ( \sqrt{\eta}) \, dE_{\mathcal{L}} (\eta)$$ 
which is bounded on $L^2(\mathcal{M})$, and its operator norm equals to $ \| m \|_{L^\infty ( [0, \infty) )}$. 

In particular, for any $\widetilde{\eta} \neq 0$, the operator $\exp (i \widetilde{\eta} \sqrt{\mathcal{L}} )$ is unitary on $L^2(\mathcal{M})$. Under some additional assumptions on $m$, for example when $m, \widehat{m} \in L^1(\mathbb{R})$, one can make use of the Euclidean Fourier inversion formula, to verify that $m ( \sqrt{\mathcal{L}} )$ can alternatively be expressed as 
\begin{align} \label{def:general-spectral-multiplier-Fourier-inversion} 
m (\sqrt{\mathcal{L}}) = (2\pi)^{-1/2} \int_{\mathbb{R}} \widehat{m} (\widetilde{\eta}) \, \exp (i \widetilde{\eta} \sqrt{\mathcal{L}}) \, d \widetilde{\eta}. 
\end{align} 

Motivated by \eqref{def:general-spectral-multiplier-Fourier-inversion}, one can also define pseudo-multipliers $m (x, \sqrt{\mathcal{L}})$ by 
\begin{align} \label{def:general-pseudo-multiplier-Fourier-inversion}
m (x, \sqrt{\mathcal{L}}) f(x) := (2\pi)^{-1/2} \int_{\mathbb{R}} \widehat{m} (x, \widetilde{\eta}) \, \exp (i \widetilde{\eta} \sqrt{\mathcal{L}} ) f(x) \, d \widetilde{\eta}. 
\end{align}
for a suitable class of symbol functions $m \in L^\infty ( \mathcal{M} \times \mathbb{R} )$, where $\widehat{m} (x, \widetilde{\eta})$ stand for the Euclidean Fourier transform of $m (x, \eta)$ as a function of $\eta$-variable (that is, with $x$-fixed). 

For the sake of brevity, we have used the same notation $m (x, \sqrt{\mathcal{L}})$ in \eqref{def:general-pseudo-multiplier-Fourier-inversion} to denote pseudo-multipliers which was already used in the introduction to denote Hermite and Grushin pseudo-multipliers. Shortly, we shall see that the two definitions are equivalent, at least for some special class of operators and symbols, thus justifying the use of the same notation. 

In particular, if we assume that $m(x, \eta)$ is supported on $\mathcal{M} \times [-R, R]$ for some $R > 0$, and that for almost every $x \in \mathcal{M}$, the function $m(x, \cdot)$ is twice continuously differentiable on $\mathbb{R}$, then we shall show in Lemma \ref{lem:general-pseudo-multiplier-compact-support-L2-bounded} that for all $f \in L^2 ( \mathcal{M} )$,  
\begin{align} \label{def:general-pseudo-multiplier-compact-support-L2-bounded} 
\| m (x, \sqrt{\mathcal{L}}) f \|_{L^2 (\mathcal{M}) } 
& \lesssim R \sum_{j=0}^2 \| \partial_{\eta}^j m \|_{L^\infty ( \mathcal{M} \times \mathbb{R}_+)} \|f\|_{L^2 (\mathcal{M})}. 
\end{align}

However, in the absence of the compact support in $\eta$-variable of the symbol function, unless we assume appropriate decay conditions on the symbol function and its derivatives, it seems difficult to even properly define pseudo-multipliers in a very general set-up. The only consideration in a general frame work, that we are aware of, is mainly in the works \cite{BernicotFreyPseudodifferentialSemigroupOperators,GeorgiadisNielsenPseudodifferentialSelfAdjointOperators}, where the authors worked with appropriate symbol classes $S^0_{1, \delta}$. 

\begin{lem} \label{lem:general-pseudo-multiplier-compact-support-L2-bounded} 
Given a non-negative self-adjoint operator $\mathcal{L}$ on a $\sigma$-finite measure space $\mathcal{M}$, let $m \in L^\infty ( \mathcal{M} \times \mathbb{R} )$ be supported on $\mathcal{M} \times [-R, R]$ for some $R > 0$, and that function $m(x, \cdot)$ is twice continuously differentiable on $\mathbb{R}$, then the operator $m (x, \sqrt{\mathcal{L}})$, defined in \eqref{def:general-pseudo-multiplier-Fourier-inversion}, satisfies inequality \eqref{def:general-pseudo-multiplier-compact-support-L2-bounded}. 
\end{lem}
\begin{proof}
We can assume that $\sum_{j=0}^2  \| \partial_{\eta}^j m \|_{L^\infty ( \mathcal{M} \times \mathbb{R} )} < \infty$, otherwise there is nothing to prove. Using the standard integration by parts twice on $\mathbb{R}$, one gets that $ \widehat{ \partial_{\eta}^2 m } (x, \widetilde{\eta}) = - \widetilde{\eta}^2 \, \widehat{m}(x, \widetilde{\eta})$. Remember that $\widehat{m} (x, \widetilde{\eta})$ denotes the Euclidean Fourier transform of $m (x, \eta)$ as a function of $\eta$-variable on $\mathbb{R}$. Next, using the fact that $m$ is supported in $\eta$-variable on $[-R, R]$, one has 
$\sum_{j=0}^2 \| \widehat{\partial_{\eta}^j m} \|_{L^\infty ( \mathcal{M} \times \mathbb{R} )} \lesssim R \sum_{j=0}^2 \| \partial_{\eta}^j m \|_{L^\infty ( \mathcal{M} \times \mathbb{R} )} < \infty.$ Consequently, 
\begin{align*} 
\| m (x, \sqrt{\mathcal{L}}) f \|_{L^2 ( \mathcal{M})} &\lesssim \int_{\mathbb{R}} \| \widehat{m}(\cdot, \widetilde{\eta}) \, \exp (i \widetilde{\eta} \sqrt{\mathcal{L}} ) f \|_{L^2 ( \mathcal{M} )} d \widetilde{\eta} \\ 
&= \int_{\mathbb{R}} \| (1 + \widetilde{\eta}^2) \widehat{m}(\cdot, \widetilde{\eta}) \, \exp (i \widetilde{\eta} \sqrt{\mathcal{L}} ) f \|_{L^2 ( \mathcal{M} )} (1 + \widetilde{\eta}^2)^{-1} \, d \widetilde{\eta} \\ 
& \leq \sum_{j=0}^2  \| \widehat{\partial_{\eta}^j m} \|_{L^\infty ( \mathcal{M} \times \mathbb{R} )} \int_{\mathbb{R}} \| \exp (i \widetilde{\eta} \sqrt{\mathcal{L}} ) f \|_{L^2 ( \mathcal{M} )} \, (1 + \widetilde{\eta}^2)^{-1} d \widetilde{\eta} \\ 
& = \sum_{j=0}^2  \| \widehat{\partial_{\eta}^j m} \|_{L^\infty ( \mathcal{M} \times \mathbb{R} )} \| f \|_{L^2 ( \mathcal{M} )} \int_{\mathbb{R}} (1 + \widetilde{\eta}^2)^{-1} d \widetilde{\eta} \\ 
& \lesssim R \sum_{j=0}^2  \| \partial_{\eta}^j m \|_{L^\infty ( \mathcal{M} \times \mathbb{R} )} \| f \|_{L^2 ( \mathcal{M} )}, 
\end{align*}
for any $f \in L^2 ( \mathcal{M} )$, and this completes the proof of the claimed estimate. 
\end{proof}

As mentioned earlier, we have considered two different ways of defining pseudo-multipliers: one by direct spectral resolutions, at least for specific operators, such as the (densely defined) ones for Hermite and Grushin operators as in \eqref{def:Hermite-pseudo} and \eqref{def:Gru-pseudo} respectively, and the other one given by \eqref{def:general-pseudo-multiplier-Fourier-inversion} in the general set-up. In the following remark, we discuss that these two definitions agree for a suitable class of symbol functions. 

\begin{rem} \label{discussion-two-ways-definition-pseudo-multiplier} 
Suppose that we are given a function $m \in L^\infty ( \mathcal{M} \times \mathbb{R}_+ )$ with the property that it is continuously differentiable on $\mathbb{R}_+$ (for almost every $x \in \mathcal{M}$), and that $ \| \partial_{\eta} m \|_{L^\infty ( \mathcal{M} \times \mathbb{R}_+ )} < \infty$. It then follows from basic calculus that $\lim_{\eta \to 0^+} m(x, \eta)$ exists, that is, we can continuously extend the function $m(x, \cdot)$ to $[0, \infty)$ with $m(x, 0) := \lim_{\eta \to 0^+} m(x, \eta)$. Repeating the same argument, if we assume that $m(x, \cdot)$ is thrice continuously differentiable on $\mathbb{R}_+$ with $\sum_{j=0}^3  \| \partial_{\eta}^j m \|_{L^\infty ( \mathcal{M} \times \mathbb{R}_+ )} < \infty$, then we can conclude that $m(x, \cdot)$ extends to $[0, \infty)$ as a twice continuously differentiable function. 

Using standard analysis, one can then extend $m$ to a function, say $m_1$, on all of $\mathcal{M} \times \mathbb{R}$ as a twice continuously differentiable function in $\eta$-variable, with the property that $m_1(x, \eta) = 0$ for all $\eta \in (-\infty, -1]$ and for almost every $x \in \mathcal{M}$, and satisfying 
$$\sum_{j=0}^2 \| \partial_{\eta}^j m_1 (x, \eta) \|_{L^\infty ( \mathcal{M} \times \mathbb{R} )} \lesssim \sum_{j=0}^3  \| \partial_{\eta}^j m \|_{L^\infty ( \mathcal{M} \times \mathbb{R}_+ )}. $$
We can now decompose $m_1$ on $\mathbb{R}$ into two pieces, one near $0$, and the other piece living on the positive real line away from the origin, as follows. Fix a function $\phi \in C^\infty(\mathbb{R})$ satisfying $\phi = 1$ on $(-\infty, 1]$, and $\phi = 0$ on $[2, \infty)$, and define $m_2(x, \eta) = \phi(\eta) m_1(x, \eta)$ and $m_3(x, \eta) = ( 1 - \phi(\eta) ) m_1(x, \eta) = ( 1 - \phi(\eta) ) m(x, \eta)$, so that $m_1 = m_2 + m_3$. It follows that 
\begin{align*}
& \sum_{j=0}^2 \| \partial_{\eta}^j m_2 (x, \eta) \|_{L^\infty ( \mathcal{M} \times \mathbb{R} )} \lesssim_\phi \sum_{j=0}^2 \| \partial_{\eta}^j m_1 (x, \eta) \|_{L^\infty ( \mathcal{M} \times \mathbb{R} )} \lesssim \sum_{j=0}^3  \| \partial_{\eta}^j m \|_{L^\infty ( \mathcal{M} \times \mathbb{R}_+ )}, \\ 
\textup{and} & \quad \sum_{j=0}^2 \| \partial_{\eta}^j m_3 (x, \eta) \|_{L^\infty ( \mathcal{M} \times \mathbb{R} )} \lesssim \sum_{j=0}^2  \| \partial_{\eta}^j m \|_{L^\infty ( \mathcal{M} \times \mathbb{R}_+ )}. 
\end{align*} 
Now, one can invoke Lemma \ref{lem:general-pseudo-multiplier-compact-support-L2-bounded} for symbol function $m_2$, to conclude that the operator $m_2 (x, \sqrt{\mathcal{L}})$, as defined by \eqref{def:general-pseudo-multiplier-Fourier-inversion}, is bounded on $L^2(\mathcal{M})$, with the operator norm controlled by $ \sum_{j=0}^3  \| \partial_{\eta}^j m \|_{L^\infty ( \mathcal{M} \times \mathbb{R}_+ )}$. So, it remains to analyse $m_3 (x, \sqrt{\mathcal{L}})$, whenever it makes sense. 

Our aim now is to discuss how we can define $m (x, \sqrt{\mathcal{L}})$ in a natural way (at least in some particular cases of $\mathcal{L}$ and $\mathcal{M}$) so that the two definitions match in the case of compact support in $\eta$-variable, that is, for symbol function such as $m_1$. Let us analyse it for the Grushin operator, that is, we are specialising to $\mathcal{L} = G$ and $\mathcal{M} = \R^{n_1+n_2}$. Take a function $f \in \mathcal{E}_l$, where sets $\mathcal{E}_l$ were defined in \eqref{def-dense-set-Grushin-pseudo-multiplier}, and recall that the set $\cup_{k \in \mathbb{N}} \mathcal{E}_l$ is dense in $L^2(\R^{n_1+n_2})$. By definition of $\mathcal{E}_l$, we have that 
$f^\lambda (x') = \sum_{|\mu| \leq l} C(\lambda,\mu) \, \Phi_{\mu}^{\lambda}(x'), $
with $C(\lambda, \mu)$ a bounded function which is compactly supported in $\lambda$-variable. Using the spectral resolution of $G$, we have 
$$\exp (i \widetilde{\eta} \sqrt{G} ) f(x) = (2 \pi)^{- n_2} \int_{\mathbb{R}^{n_2}} e^{-i \lambda \cdot x^{\prime \prime}} \sum_{|\mu| \leq l} \exp ( i \widetilde{\eta} \sqrt{(2|\mu| + n_1)|\lambda|} ) \, C(\lambda,\mu) \, \Phi_{\mu}^{\lambda}(x') \, d\lambda. $$
Next, multiplying both sides by $\widehat{m_2} (x, \widetilde{\eta})$ and integrating with respect to $\widetilde{\eta}$-variable, one can make use of the assumptions on $C(\lambda, \mu)$ to validate the change of order of integration in double integrals (in $\widetilde{\eta}$ and $\lambda$-variables), that is, 
\begin{align*}
& (2 \pi)^{- 1/2} \int_{\mathbb{R}} \widehat{m_2} (x, \widetilde{\eta}) \, \exp (i \widetilde{\eta} \sqrt{G} ) f(x) \, d \widetilde{\eta} \\ 
& = (2 \pi)^{- (n_2 + \frac{1}{2} )} \int_{\mathbb{R}^{n_2}} e^{-i \lambda \cdot x^{\prime \prime}} \sum_{|\mu| \leq l} \left( \int_{\mathbb{R}} \widehat{m_2} (x, \widetilde{\eta}) \, \exp ( i \widetilde{\eta} \sqrt{(2|\mu| + n_1)|\lambda|} ) \, d \widetilde{\eta} \right) C(\lambda,\mu) \, \Phi_{\mu}^{\lambda}(x') \, d\lambda \\ 
& = (2 \pi)^{-n_2} \int_{\mathbb{R}^{n_2}} e^{-i \lambda \cdot x^{\prime \prime}} \sum_{|\mu| \leq l} m_2 (x, \sqrt{(2|\mu| + n_1)|\lambda|} ) \, C(\lambda,\mu) \, \Phi_{\mu}^{\lambda}(x') \, d\lambda \\ 
& = (2 \pi)^{-n_2} \int_{\mathbb{R}^{n_2}} e^{-i \lambda \cdot x^{\prime \prime}} \sum_{k \in \mathbb{N}} m_2 ( x, \sqrt{(2 k + n)|\lambda|} ) P_k (\lambda) f^{\lambda}(x^{\prime}) \, d\lambda, 
\end{align*} 
where the final expression is exactly what we had considered in \eqref{def:Gru-pseudo}. 

Now, $(2 k + n)|\lambda| \in \mathbb{R}_+$ for every $\lambda \neq 0$ and $k \in \mathbb{N}$, and for any $\eta \in \mathbb{R}_+$, we have $ m ( x, \eta ) = m_1 ( x, \eta ) = m_2 ( x, \eta ) + m_3 ( x, \eta ).$ Summarising, we have that definitions \eqref{def:general-pseudo-multiplier-Fourier-inversion} and \eqref{def:Gru-pseudo} coincide for symbol function $m_2$ on the dense set $\mathcal{E}$. While expression \eqref{def:general-pseudo-multiplier-Fourier-inversion} help us conclude $L^2$-boundedness of the pseudo-multiplier $m_2 (x, \sqrt{G})$, the formula given by \eqref{def:Gru-pseudo} is pointwise well-defined on the same dense set $\mathcal{E}$ for any of the symbol functions $m$, $m_1$ or $m_2$.  
\end{rem}


\subsection{Control distance for the Grushin operator} 
It is known (see Proposition $5.1$ of \cite{RobinsonSikoraDegenerateEllipticOperatorsGrushinTypeMathZ2008}) that the control distance $\tilde{d}(x,y)$ for the Grushin operator has the following behaviour: 
\begin{align*}
\tilde{d}(x,y) &\sim \left|x^{\prime} - y^\prime \right| + 
\begin{cases}
\frac{\left|x^{\prime \prime} - y^{\prime \prime}\right|}{\left|x^{\prime} \right| + \left|y^\prime \right|} &\textup{ if } \left|x^{\prime \prime} - y^{\prime \prime}\right|^{1/2} \leq \left|x^{\prime} \right| + \left|y^\prime \right| \\
\left|x^{\prime \prime} - y^{\prime \prime}\right|^{1/2} &\textup{ if } \left|x^{\prime \prime} - y^{\prime \prime}\right|^{1/2} \geq \left|x^{\prime} \right| + \left|y^\prime \right| 
\end{cases}. 
\end{align*}

Instead of $\tilde{d}$, we shall work with $d$, which is defined as 
\begin{align} \label{def:distance-1} 
d(x,y) := \left(\left|x^{\prime} - y^\prime \right|^4 + \frac{\left|x^{\prime \prime} - y^{\prime \prime}\right|^4}{\left| x^{\prime \prime} - y^{\prime \prime}\right|^2 + \left(\left|x^{\prime} \right|^2 + \left|y^\prime \right|^2\right)^2}\right)^{1/4}.
\end{align}
Clearly, $d \sim \tilde{d}$. One may notice that $d$ is smooth off-diagonal, a fact that we crucially use in our analysis. Since $d$ is equivalent to the control distance $\tilde{d}$, we have that $d$ is a quasi metric, that is, there exists a constant $C_0 \geq 1$ such that 
\begin{align} \label{quasi-metric-constant}
d(x,y) \leq C_0 (d(x,z) + d(z,y)),
\end{align}
for all $x, y, z \in \mathbb{R}^{n_1 + n_2}$. From now on wards, by abuse of notation, we refer to $d$ as a metric or distance. Let us denote by $B(x, r)$ the ball centered at $x$, of radius $r$ with respect to the metric $d$, and write $\left|B(x,r)\right|$ for it's volume. It is known (see, for example, Proposition 5.1 of \cite{RobinsonSikoraDegenerateEllipticOperatorsGrushinTypeMathZ2008}) that there exist constants $c_1, C_1 > 0$ such that 
\begin{align} \label{grushin-ball-growth} 
c_1 r^{n_1 + n_2} \max\{r, |x^{\prime}|\}^{n_2} \leq \left|B(x,r)\right| \leq C_1 r^{n_1 + n_2} \max\{r, |x^{\prime}|\}^{n_2},
\end{align}
from which it follows immediately that the metric $d$ has the doubling property. Also, we have $|B(x,r)| \gtrsim r^Q,$ where $Q = n_1 + 2 n_2$ is the homogeneous dimension of $\mathbb{R}^{n_1 + n_2}$. 


In the later sections, we shall also make use of the following result concerning the integrability property of the metric $d$. 

\begin{lem} \label{lem-integrability-ball-growth}
For any real number $s > Q$ we have 
\begin{align} \label{integrability-ball-growth}
\sup_y |B(y,1)|^{-1/2} \int_{\mathbb{R}^{n_1 + n_2}} |B(x,1)|^{-1/2} (1 + d(x,y))^{- s} \, dx  & < \infty. 
\end{align} 
\end{lem} 
\begin{proof} 
From the definition it is clear that $d(x,y) \sim \min \{\varrho_1(x,y), \varrho_2(x,y) \}$, where 
\begin{align} \label{def:distance-2} 
\varrho_1(x,y) = \left|x^\prime - y^\prime \right| + \left|x^{\prime \prime} - y^{\prime \prime}\right|^{1/2}, \quad \quad \varrho_2(x,y) = \left|x^\prime - y^\prime \right| + 
\frac{\left|x^{\prime \prime} - y^{\prime \prime}\right|}{\left|x^\prime \right| + \left|y^\prime \right|}.
\end{align} 
Therefore, it suffices to establish  \eqref{integrability-ball-growth} for both $\varrho_1$ and $\varrho_2$ instead of $d$. 

The estimate for $\varrho_1$ is quite simple and straightforward. In fact, 
\begin{align*} 
|B(y,1)|^{-1/2} \int_{\mathbb{R}^{n_1 + n_2}} |B(x,1)|^{-1/2} (1 + \varrho_1(x,y))^{- s} \, dx & \leq \int_{\mathbb{R}^{n_1 + n_2}} (1 + \varrho_1(x,y))^{-s} \, dx \\ 
& = \int_{\mathbb{R}^{n_1 + n_2}} (1 + \left| x^\prime \right| + \left| x^{\prime \prime} \right|^{1/2})^{-s} \, dx, 
\end{align*} 
where the last integral is independent of $y$ and is finite since $s > Q$. 

For the estimation with respect to $\varrho_2$, let us choose and fix $s_1, s_2\in \mathbb{R}_+$ such that $s_1 > n_1 + n_2$, $s_2 > n_2$ and $s_1 + s_2= s$. Then,  
\begin{align*} 
& |B(y,1)|^{-1/2} \int_{\mathbb{R}^{n_1 + n_2}} |B(x,1)|^{-1/2} (1 + \varrho_2(x,y))^{-s} \, dx \\ 
\lesssim & \left(1 + \left| y^\prime \right| \right)^{-n_2/2} \int_{\mathbb{R}^{n_1 + n_2}} \left(1 + \left| x^\prime \right| \right)^{-n_2/2} \left(1 + \left|x^\prime - y^\prime \right| \right)^{-s_1} \left(1 + \frac{\left|x^{\prime \prime} - y^{\prime \prime}\right|}{\left|x^\prime \right| + \left|y^\prime \right|}\right)^{-s_2} \, dx^\prime \, dx^{\prime \prime} \\ 
= & C_{s_2} \left(1 + \left| y^\prime \right| \right)^{-n_2/2} \int_{\mathbb{R}^{n_1}} \left(1 + \left| x^\prime \right| \right)^{-n_2/2} \left(1 + \left|x^\prime - y^\prime \right| \right)^{-s_1} \left( \left|x^\prime \right| + \left|y^\prime \right| \right)^{n_2} \, dx^{\prime}. 
\end{align*} 

With fixed $y$, we decompose the last integral into following three regions: $W_1 = \{ x^{\prime} \in \mathbb{R}^{n_1} : \frac{|y^{\prime}|}{2} < |x^{\prime}| < 2 |y^{\prime}| \}$, $W_2 = \{ x^{\prime} \in \mathbb{R}^{n_1} : |x^{\prime}| \leq \frac{|y^{\prime}|}{2} \}$ and $W_3 = \{ x^{\prime} \in \mathbb{R}^{n_1} : |x^{\prime}| \geq 2 |y^{\prime}| \}$. 

On $W_1$, we have 
\begin{align*} 
& \left(1 + \left| y^\prime \right| \right)^{-n_2/2} \int_{W_1} \left(1 + \left| x^\prime \right| \right)^{-n_2/2} \left(1 + \left|x^\prime - y^\prime \right| \right)^{-s_1} \left( \left|x^\prime \right| + \left|y^\prime \right| \right)^{n_2} \, dx^{\prime} \\ 
& \leq \int_{W_1} \left(1 + \left|x^\prime - y^\prime \right| \right)^{-s_1} \, dx^{\prime} \leq \int_{\mathbb{R}^{n_1}} \left(1 + \left|x^\prime \right| \right)^{-s_1} \, dx^{\prime} < \infty. 
\end{align*} 

On $W_2$, we have $\left|x^\prime - y^\prime \right| \geq \frac{\left|y^\prime \right|}{2} \geq \left|x^\prime \right|$, and thus 
\begin{align*} 
& \left(1 + \left| y^\prime \right| \right)^{-n_2/2} \int_{W_2} \left(1 + \left| x^\prime \right| \right)^{-n_2/2} \left(1 + \left|x^\prime - y^\prime \right| \right)^{-s_1} \left( \left|x^\prime \right| + \left|y^\prime \right| \right)^{n_2} \, dx^{\prime} \\ 
& \lesssim \int_{W_2} \left(1 + \left| x^\prime \right| \right)^{-n_2} \left(1 + \left|x^\prime - y^\prime \right| \right)^{-s_1}  \left|y^\prime \right|^{n_2} \, dx^{\prime} \\ 
& \lesssim \int_{W_2} \left(1 + \left| x^\prime \right| \right)^{-n_2} \{ \left(1 + \left|x^\prime \right| \right)^{-(s_1- n_2)} \left(1 + \left|y^\prime \right| \right)^{-n_2} \} \left|y^\prime \right|^{n_2} \, dx^{\prime} \lesssim \int_{\mathbb{R}^{n_1}} \left(1 + \left| x^\prime \right| \right)^{-s_1} \, dx^{\prime} < \infty. 
\end{align*} 

Finally, on $W_3$, we have $\left|x^\prime - y^\prime \right| \geq \frac{\left|x^\prime \right|}{2} \geq \left|y^\prime \right|$, and thus 
\begin{align*} 
& \left(1 + \left| y^\prime \right| \right)^{-n_2/2} \int_{W_3} \left(1 + \left| x^\prime \right| \right)^{-n_2/2} \left(1 + \left|x^\prime - y^\prime \right| \right)^{-s_1} \left( \left|x^\prime \right| + \left|y^\prime \right| \right)^{n_2} \, dx^{\prime} \\ 
& \lesssim \int_{W_3} \left(1 + \left|x^\prime - y^\prime \right| \right)^{-s_1}  \left|x^\prime \right|^{n_2} \, dx^{\prime} \\ 
& \lesssim \int_{W_3} \left(1 + \left|x^\prime \right| \right)^{-s_1} \left|x^\prime \right|^{n_2} \, dx^{\prime} \leq \int_{\mathbb{R}^{n_1}} \left(1 + \left| x^\prime \right| \right)^{-(s_1- n_2)} \, dx^{\prime} < \infty. 
\end{align*} 

This completes the proof of Lemma \ref{lem-integrability-ball-growth}. 
\end{proof} 


\subsection{Heat kernel for the Grushin operator}
It is well known that the heat kernel $p_t$ for the Grushin operator (that is, the integral kernel of the operator $e^{-t G}$, $t>0$) satisfies pointwise Gaussian estimates (see, Corollary 6.6 of \cite{RobinsonSikoraDegenerateEllipticOperatorsGrushinTypeMathZ2008}). More precisely, there exist constants $b, C>0$ such that 
\begin{align} \label{heat-kernel-bounds} 
\left|p_t(x,y)\right| \leq C \left|B(x,t^{1/2})\right|^{-1} \exp \left( \frac{-b}{t} d(x,y)^2 \right)  
\end{align}
for all $x,y \in \mathbb{R}^{n_1 + n_2}$ and $t>0$. 

From the heat semigroup property we have 
$$\int_{\mathbb{R}^{n_1 + n_2}} p_{t_1} (x,z) p_{t_2} (z,y) \, dz = p_{t_1 + t_2}(x,y) \quad \textup{for all } x,y \in \mathbb{R}^{n_1 + n_2} \textup{ and } t_1, t_2 > 0.$$

Let us consider the Bessel potential $\mathcal{B}_s$, for $s>0$, which is the integral kernel of the operator $(I+G)^{-s/2}.$ We can express $\mathcal{B}_s$ in terms of the heat kernel $p_t$ as follows: 
$$\mathcal{B}_s(x,y) = B_s(y,x) = \frac{1}{\Gamma \left( s/2 \right)} \int_0^\infty t^{\frac{s}{2} -1} e^{-t} p_t(x,y) \, dt.$$

We have the following Sobolev embedding theorem for the Grushin operator. 
\begin{lem}[Sobolev embedding] \label{grushin-Sobolev-embed}
Given $s > Q/2$, there exists a constant $C_s > 0$ such that 
$$\left\||B(\cdot, 1)|^{1/2} f\right\|_{\infty} \leq C_s \left\|(I+G)^{s/2}f\right\|_2.$$
\end{lem}
\begin{proof}
Let us write 
$\displaystyle f(x) = (I+G)^{-s/2} ((I+G)^{s/2} f) (x) = \int_{\mathbb{R}^{n_1 + n_2}} \mathcal{B}_s(x,y) (I+G)^{s/2} f(y) \, dy.$ 
Then, by Cauchy-Schwarz inequality  
\begin{align*}
|f(x)|^2 &\leq \left(\int_{\mathbb{R}^{n_1 + n_2}} \left|\mathcal{B}_s(x,y)\right|^2 \, dy\right) \left(\int_{\mathbb{R}^{n_1 + n_2}} |(I+G)^{s/2} f(y)|^2 \, dy \right) \\ 
&= \left(\int_{\mathbb{R}^{n_1 + n_2}} \left|\mathcal{B}_s(x,y)\right|^2 \, dy\right) \left\|(I+G)^{s/2}f\right\|_2^2.
\end{align*} 

We can now analyse the $L^2$-norm of $B_s(x, \cdot)$ as follows. 
\begin{align*}
& \int_{\mathbb{R}^{n_1 + n_2}} \left|\mathcal{B}_s(x,y)\right|^2 \, dy \\ 
& = \int_{\mathbb{R}^{n_1 + n_2}} \mathcal{B}_s(x,y) \mathcal{B}_s(y,x) \, dy \\
&= \frac{1}{\left( \Gamma \left( s/2 \right)\right)^2} \int_{\mathbb{R}^{n_1 + n_2}} \int_0^\infty \int_0^\infty t_1^{\frac{s}{2} -1} e^{-t_1} p_{t_1} (x,y) t_2^{\frac{s}{2} -1} e^{-t_2} p_{t_2} (y,x) \, dt_1\, dt_2 \, dy \\
&= \frac{1}{\left( \Gamma \left( s/2 \right)\right)^2} \int_0^\infty \int_0^\infty (t_1 t_2)^{\frac{s}{2} -1} e^{-(t_1 + t_2)} \left(\int_{\mathbb{R}^{n_1 + n_2}} p_{t_1} (x,y) p_{t_2} (y,x) \, dy\right) \, dt_1\, dt_2 \\
&= \frac{1}{\left( \Gamma \left( s/2 \right)\right)^2} \int_0^\infty \int_0^\infty (t_1 t_2)^{\frac{s}{2} - 1} e^{-(t_1 + t_2)} p_{t_1 + t_2}(x,x) \, dt_1\, dt_2 \\ 
&\lesssim_s \int_0^\infty \int_0^\infty (t_1 t_2)^{\frac{s}{2} -1} e^{-(t_1 + t_2)} \frac{1}{\left|B(x,(t_1 + t_2)^{1/2})\right|} \, dt_1\, dt_2 \\ 
&= |B(x, 1)|^{-1} \int_0^\infty \int_0^\infty (t_1 t_2)^{\frac{s}{2} -1} e^{-(t_1 + t_2)} \frac{|B(x, 1)|}{\left|B(x,(t_1 + t_2)^{1/2})\right|} \, dt_1\, dt_2 \\
&\lesssim |B(x, 1)|^{-1} \int_0^\infty \int_0^\infty (t_1 t_2)^{\frac{s}{2} -1} e^{-(t_1 + t_2)} \left( 1 + \frac{1}{(t_1 + t_2)^{1/2}}\right)^{Q} \, dt_1\, dt_2 \\
&\lesssim_s |B(x, 1)|^{-1},
\end{align*}
where the last integral converges since $s > Q/2$. 

This completes the proof of Lemma \ref{grushin-Sobolev-embed}. 
\end{proof}


\subsection{Weighted Plancherel estimates}
We recall first the following two results from the existing literature about weighted Plancherel estimates for Grushin multipliers with compact support. 

\begin{lem} \label{lem:weighted-plancherel-L-2}
For all $s > r \geq 0$, we have 
$$ \left\| \left( 1 + R d(x, \cdot) \right)^{r} k_{F(G)} (x, \cdot) \right\|_2 \lesssim_{r, s} \left| B(x, R^{-1}) \right|^{-1/2} \left\| F \left(R^2 \cdot \right) \right\|_{W^s_2},$$
for every $R > 0$, and for all bounded Borel functions $F : \mathbb{R} \to \mathbb{C}$ such that $\supp F \subseteq [R^2, 4 R^2].$ 
\end{lem} 

\begin{lem} \label{lem:weighted-plancherel-L-infty} 
For all $s > r \geq 0$, we have 
\begin{align} \label{eq:heat-kernel-weighted-plancherel}
\left( 1 + R d(x,y) \right)^r \left| k_{F(G)} (x,y) \right| \lesssim_{r,s} \left| B \left( x, R^{-1} \right) \right|^{-1/2} \left| B \left( y, R^{-1} \right) \right|^{-1/2} \left\| F (R^2 \cdot ) \right\|_{W^s_2}, 
\end{align} 
for every $R > 0$, and for all bounded Borel functions $F : \mathbb{R} \to \mathbb{C}$ such that $\supp F \subseteq [R^2, 4 R^2].$ 
\end{lem}

We refer to Lemma 4.3 (a) of \cite{DuongOuhabazSikoraWeightedPlancherel2002JFA} for the above written Lemma \ref{lem:weighted-plancherel-L-2}, whereas Lemma \ref{lem:weighted-plancherel-L-infty} is taken from \cite{AnhBuiDuongSpectralMultipliersBesovTriebelLizorkin} (see Lemma 4.3 of \cite{AnhBuiDuongSpectralMultipliersBesovTriebelLizorkin}). In both of the above lemmas, we have taken the integrability index 2 in the Sobolev norm, as it is known that the Grushin multipliers satisfy the unweighted Plancherel estimate with this index (see, for example, Proposition 12 of \cite{MartiniMullerGrushinRevistaMath} or Proposition 10 of \cite{MartiniSikoraGrushinMRL}). 

Next, for $\boldsymbol{L} = (L_1, L_2,\ldots, L_{n_1})$ and $\boldsymbol{U} = (U_1, U_2, \ldots, U_{n_2})$, with $L_j$ and $U_k$ as in \eqref{eq:operatorsLandU}, the integral kernel $k_{F(\boldsymbol{L},\boldsymbol{U})}$ of the multiplier $F(\boldsymbol{L}, \boldsymbol{U})$ associated with a compactly supported bounded Borel function $F$ on $\mathbb{R}_+^{n_1} \times (\mathbb{R}^{n_2}\setminus\{0\})$ can be given as follows (see also Proposition 5 in \cite{MartiniSikoraGrushinMRL}): 
$$ k_{F(\boldsymbol{L},\boldsymbol{U})}(x,y) = (2\pi)^{-n_2} \int_{\mathbb{R}^{n_2}} \sum_\mu F((2\mu+\tilde{1})|\lambda|, - \lambda) E_{\mu, \lambda} (x, y) \, d\lambda,$$ 
where $\tilde{1} = (1, 1, \ldots, 1) \in \mathbb{R}^{n_1},$ and 
$E_{\mu, \lambda} (x, y) = \Phi_\mu^\lambda(x^\prime) \Phi_\mu^\lambda(y^\prime) e^{- i \lambda \cdot (x^{\prime \prime} - y^{\prime \prime})}.$ 

One can deduce the following weighted Plancherel estimate for the joint functional calculus of $(\boldsymbol{L}, \boldsymbol{U})$ from Theorem 6.1 of \cite{MartiniJointFunctionalCalculiMathZ}. 

\begin{lem} \label{lem:weighted-plancherel-joint-functional}
For all $s > r \geq 0$ and $2 \leq p \leq \infty$, we have 
$$ \left\| \left( 1 + R d(x, \cdot) \right)^{r} k_{F(\boldsymbol{L}, \boldsymbol{U})} (x, \cdot) \right\|_p \lesssim_{p, r, s} \left| B(x, R^{-1}) \right|^{-1/{p^{\prime}}} \left\| F \left(R^2 \cdot \right) \right\|_{W^s_{\infty}},$$
for every $R > 0$, and for all bounded Borel functions $F : \mathbb{R}^{n_1 + n_2} \to \mathbb{C}$ such that $\supp F \subseteq [-R^2, R^2]^{n_1 + n_2}.$ 
\end{lem} 

Let us sketch a proof of Lemma \ref{lem:weighted-plancherel-joint-functional}. Following the terminology of Section 6 of \cite{MartiniJointFunctionalCalculiMathZ}, let us consider the dilation operators $\epsilon_R$ on $\mathbb{R}^{n_1} \times \mathbb{R}^{n_2}$ by $\epsilon_R (\tau, \kappa) = \left( R^2 \tau, R^2 \kappa \right)$, for $R>0$, and the continuous map $\Psi = \left( \Psi_1, \ldots, \Psi_{n_1 + n_2 + 1} \right) : \mathbb{R}^{n_1 + n_2} \to \mathbb{R}^{n_1 + n_2 + 1}$ defined by 
$$ \Psi(\tau, \kappa) = \left( e^{-|\tau|_1}, \tau_1 e^{-|\tau|_1}, \ldots, \tau_{n_1} e^{-|\tau|_1}, \kappa_1 e^{-|\tau|_1}, \ldots, \kappa_{n_2} e^{-|\tau|_1} \right).$$
Here the notation $|\tau|_1$ is used for $\sum_{1 \leq j \leq n_1} \tau_j$. 

For all $ p \in [2, \infty], r \in [0, \infty), R \in (0, \infty)$, let us define the mixed weighted Lebesgue norms for functions $K : \mathbb{R}^{n_1 + n_2} \times \mathbb{R}^{n_1 + n_2} \to \mathbb{C}$ by 
\begin{align} \label{def-mixed-weighted}
\vertiii{K}_{p,r,R} = \text{ess sup}_y \left| B \left(y, R^{-1} \right) \right|^{1/p^{\prime}} \left\| K(\cdot, y) \left( 1 + R d(\cdot, y) \right)^r \right\|_{L^p \left( \mathbb{R}^{n_1 + n_2}\right)}.
\end{align}

Now, in order to apply Theorem 6.1 of \cite{MartiniJointFunctionalCalculiMathZ}, we need to check that 
$$ \sup_{R>0} \vertiii{K_{\Psi_j \circ \epsilon_{R^{-1}} \left( L_1, L_2, \ldots, L_{n_1}, U_1, U_2, \ldots, U_{n_2} \right)}}_{2,r,R} < \infty $$
for all $r \in [0, \infty)$ and $ 1 \leq j \leq n_1 + n_2 + 1$. 

Note that 
\begin{align*}
& K_{\Psi_j \circ \epsilon_{R^{-1}} \left( L_1, L_2, \ldots, L_{n_1}, U_1, U_2, \ldots, U_{n_2} \right)} (x,y) \\ 
& = 
\begin{cases}
p_{R^{-2}}(x,y), & \quad \text{if} \quad j = 1; \\ 
R^{-2} \left( L_j p_{R^{-2}} (\cdot, y) \right) (x), & \quad \text{if} \quad 2 \leq j \leq n_1 + 1; \\ 
R^{-2} \left( U_j p_{R^{-2}} (\cdot, y) \right) (x), & \quad \text{if} \quad n_1 + 2 \leq j \leq n_1 + n_2 + 1. 
\end{cases}
\end{align*}

It follows from the Gaussian bounds of the heat kernel corresponding to the Grushin operator that $\vertiii{p_{R^{-2}}}_{2,r,R}$ is finite (see, for example, \cite{MartiniSikoraGrushinMRL}). The same would be true for $K (x,y) = R^{-2} \left( L_j p_{R^{-2}} (\cdot, y) \right) (x)$ or $K(x,y) = R^{-2} \left( U_j p_{R^{-2}} (\cdot, y) \right) (x)$ provided that we have analogous pointwise bounds for gradients $L_j p_{R^{-2}}$ and $U_j p_{R^{-2}}$. To this end, we can make use of the approach of relating the pointwise bounds on the heat kernel (and its left-invariant gradients) for the Grushin operator with the analogous pointwise bounds on the heat kernel (and its Left-invariant gradients) for Heisenberg-Reiter group (see Section 2 of \cite{MartiniSikoraGrushinMRL}). The above idea has also been recently used in \cite{DziubanskiSikoraLieApproach} to estimate heat kernels in a more general set-up of Grushin-type operators. For Lie groups of polynomial growth, pointwise estimates for the heat kernel and it's left-invariant gradients (of arbitrary order) are well known (see \cite{VaropoulosSaloffCosteCoulhonAnalysisGeometryGroupsBook92}). Also, for the central variable derivative, note that $U_k = [X_{j, k}, X_j].$ Overall, making use of the above ideas, one can show the following estimates: 
\begin{align*} 
R^{-2} \left| \left(L_j p_{R^{-2}} (\cdot, y) \right) (x) \right| & \lesssim \left| B \left(x, R^{-1} \right) \right|^{-1} \exp \left(-b R^{2} d(x,y)^2 \right), \\ 
\text{and} \quad R^{-2} \left| \left(U_j p_{R^{-2}} (\cdot, y) \right) (x) \right| & \lesssim \left| B \left(x, R^{-1} \right) \right|^{-1} \exp \left(-b R^{2} d(x,y)^2 \right). 
\end{align*}

With these estimates at hand, Lemma \ref{lem:weighted-plancherel-joint-functional} follows from Theorem 6.1 of \cite{MartiniJointFunctionalCalculiMathZ}. 


\subsection{Partition of unity}
We now discuss the existence of a smooth partition of unity associated to the metric $d$. For this, we closely follow the ideas from Lemma $5.7.5$ from page $387$ onwards of \cite{FischerRuzhanskyQuantizationBook}, with the changes essentially pertaining to the challenges that are related to the behaviour of the control distance function as well as the fact that the volume of the balls in our case may also depend on their centers while it is not the case when one works with balls associated to a homogeneous norm on a Lie group. 

Let $C_0$, $C_1$ and $c_1$ be same as in \eqref{quasi-metric-constant} and \eqref{grushin-ball-growth}. By Zorn's lemma, there exists a maximal family $\left\{B\left(x_J, \frac{1}{2C_0} \right)\right\}_{J \in \mathbb{N}}$ of disjoint balls of radius $\frac{1}{2C_0}$. We choose and fix one such family. 

\begin{lem} \label{lem-grushin-partition-part-fix}
The following assertions hold. 
\renewcommand{\theenumi}{\roman{enumi}}
\begin{enumerate}
\item The balls $\left\{B\left(x_J, 1 \right)\right\}_{J \in \mathbb{N}}$ cover $\mathbb{R}^{n_1 + n_2}$. 

\item For any $C \geq 1,$ no point of $\mathbb{R}^{n_1 + n_2}$ belongs to more than $\frac{C_1}{c_1} \left(4 C^2_0 C\right)^{Q}$ number of balls from the family $\left\{B\left(x_J, C \right)\right\}_{J \in \mathbb{N}}$. 

\item For every real number $s > Q$, there exists a finite constant $C_s > 0$ such that 
$$ \sup_{x \in \mathbb{R}^{n_1 + n_2}} \sum_{J \in \mathbb{N}} (1 + d\left(x, x_J\right))^{-s} \leq C_s.$$ 

\item There exists a sequence $\left\{\chi_J \right\}_{J \in \mathbb{N}}$ of functions belonging to $C_c^\infty(\mathbb{R}^{n_1 + n_2})$ such that for each $J$, the function $\chi_J$ is supported in $B\left(x_J, 2 \right)$ and satisfies $0 \leq \chi_J \leq 1$ and $\sum_{J \in \mathbb{N}} \chi_J = 1.$ Moreover, for any $\Gamma \in \mathbb{N}^{n_1 + n_1 n_2}$, there exists a finite constant $C_{\Gamma} > 0$ such that 
$$ \sup_{J \in \mathbb{N}} \sup_{x \in \mathbb{R}^{n_1 + n_2}} \left|X^\Gamma \chi_J (x)\right| \leq C_{\Gamma}. $$
\end{enumerate}
\end{lem}
\begin{proof} 
\textbf{(i)} Take and fix $x \in \mathbb{R}^{n_1 + n_2}$. Since $\left\{B\left(x_J, \frac{1}{2C_0} \right)\right\}_{J \in \mathbb{N}}$ is a maximal family of disjoint balls of radius $\frac{1}{2C_0}$, there exists $J_0 \in \mathbb{N}$ such that $B\left(x, \frac{1}{2C_0} \right) \cap B\left(x_{J_0}, \frac{1}{2C_0} \right) \neq \emptyset.$ Take an element $y$ from the intersection of these two balls. Then,  
$$d(x, x_{J_0}) \leq C_0 \left(d(x, y) + d(y, x_{J_0})\right) < 1.$$

Hence, we have that the balls $\left\{B\left(x_J, 1 \right)\right\}_{J \in \mathbb{N}}$ cover $\mathbb{R}^{n_1 + n_2}$. 

\medskip \noindent \textbf{(ii)} Take and fix $C \geq 1$ and $x \in \mathbb{R}^{n_1 + n_2}$. Let $r_0 \geq 1$ be such that $x \in \cap_{r=1}^{r_0} B(x_{J_r}, C).$ Then, for any $1 \leq r \leq r_0$ and $y \in B(x_{J_r}, C),$ we have 
$$ d(x,y) \leq C_0 \left( d(x, x_{J_r}) + d(x_{J_r}, y) \right) \leq 2 C_0 C, $$ 
implying that $\cup_{r=1}^{r_0} B(x_{J_r}, C) \subseteq B(x, 2C_0 C).$ As a consequence, we get that the disjoint union $\cup_{r=1}^{r_0} B\left(x_{J_r}, \frac{1}{2C_0} \right)$ is contained in the ball $B\left(x, 2 C_0 C \right)$. Thus 
\begin{align} \label{finite-overlap-ineq1}
\sum_{r=1}^{r_0} \left|B\left(x_{J_r}, \frac{1}{2C_0} \right)\right| \leq \left|B\left(x, 2 C_0 C \right)\right|.
\end{align}

Now, in view of the estimates (\ref{grushin-ball-growth}), we have 
\begin{align*} 
\sum_{r=1}^{r_0} \left|B\left(x_{J_r}, \frac{1}{2C_0} \right)\right| &\geq c_1 \left(\frac{1}{2C_0}\right)^{n_1 + n_2} \sum_{r=1}^{r_0} \max\left\{\frac{1}{2C_0}, \left|x^{\prime}_{J_r}\right|\right\}^{n_2}, \\
\textup{and} \quad \quad \quad \quad \left|B\left(x, 2C_0 C \right)\right| &\leq C_1 \left(2C_0 C\right)^{n_1 + n_2} \max\left\{2C_0 C, \left|x^{\prime}\right|\right\}^{n_2}.
\end{align*}

Putting these estimates in (\ref{finite-overlap-ineq1}), we get 
\begin{align} \label{finite-overlap-ineq2}
c_1 \left(\frac{1}{2C_0}\right)^{n_1 + n_2} \sum_{r=1}^{r_0} \max\left\{\frac{1}{2C_0}, \left|x^{\prime}_{J_r}\right| \right\}^{n_2} \leq C_1 \left(2C_0 C\right)^{n_1 + n_2} \max\left\{2C_0 C, \left|x^{\prime}\right|\right\}^{n_2}.
\end{align}

\medskip \noindent {\bf Case 1:} When $|x^{\prime}| \geq 2C_0 C$. 

\medskip Since $x \in B\left(x_{J_r}, C \right)$, we have $d(x, x_{J_r}) < C$, and therefore $\left|x^{\prime} - x^{\prime}_{J_r}\right| \leq d(x, x_{J_r}) < C.$ From this, we see that  
$$\left| x^{\prime}_{J_r}\right| \, \geq \, \left|x^{\prime}\right| - \left|x^{\prime} - x^{\prime}_{J_r}\right| \, \geq \, \left(1 - \frac{1}{2C_0}\right)|x^{\prime}| \, \geq \, \frac{1}{2} |x^{\prime}|.$$
and therefore estimate (\ref{finite-overlap-ineq2}) implies that 
\begin{align*} 
r_0 c_1 \left(\frac{1}{2C_0}\right)^{n_1 + n_2} \left(\frac{1}{2} |x^{\prime}|\right)^{n_2} \leq C_1 \left(2C_0 C\right)^{n_1 + n_2} \left|x^{\prime}\right|^{n_2},
\end{align*} 
which is same as 
$$r_0 \leq \frac{C_1}{c_1} 2^{n_2} \left(4 C^2_0 C\right)^{n_1 + n_2}\leq  \frac{C_1}{c_1}  \left(4 C^2_0 C\right)^{Q},$$
and this completes the claim of part (ii) in Case 1. 

\medskip \noindent 
{\bf Case 2:} When $|x^{\prime}| < 2C_0 C$. 

\medskip In this case, estimate (\ref{finite-overlap-ineq2}) implies that 
\begin{align*}
r_0 c_1 \left(\frac{1}{2C_0}\right)^{n_1 + n_2} \left(\frac{1}{2C_0}\right)^{n_2} \leq C_1 \left(2C_0 C\right)^{n_1 + n_2} \left(2C_0 C\right)^{n_2},
\end{align*} 
which is same as 
$$r_0 \leq \frac{C_1}{c_1} \left(4 C^2_0 C\right)^{Q},$$
and this completes the claim of part (ii) in Case 2. 

\medskip \noindent \textbf{(iii)} Take and fix $x \in \mathbb{R}^{n_1 + n_2}$. Now, we know from part (ii) that no point of $\mathbb{R}^{n_1 + n_2}$ belongs to more than $\frac{C_1}{c_1} \left(4 C^2_0 C\right)^{Q}$ number of balls from the family $\left\{B\left(x_J, C \right)\right\}_{J \in \mathbb{N}}$ for any $C \geq 1$. Therefore, for any $k \in \mathbb{N}$, we have 
$$ \#\{J \in \mathbb{N} : d(x, x_J) \in [2^k, 2^{k+1})\} \leq \#\{J \in \mathbb{N} : d(x, x_J) < 2^{k+1} \} \leq \frac{C_1}{c_1} \left(4 C^2_0 \, 2^{k+1} \right)^{Q}.$$

Hence, 
\begin{align*} 
\sum_{J \in \mathbb{N}} \left( 1 + d\left(x, x_J\right) \right)^{-s} & = \sum_{d\left(x, x_J\right) \in [0, 1)} \left( 1 + d\left(x, x_J \right) \right)^{-s} + \sum_{k = 0}^\infty \sum_{d\left(x, x_J\right) \in [2^k, 2^{k+1})} \left( 1 + d\left(x, x_J \right) \right)^{-s} \\ 
&\leq \frac{C_1}{c_1} \left(4 C^2_0 \right)^{Q} \left[ 1 + \sum_{k = 0}^\infty 2^{(k+1)Q} 2^{-k s} \right] = C_s < \infty, 
\end{align*}
whenever $s > Q.$ This completes the proof of part (iii) of the lemma. 

\medskip \noindent \textbf{(iv)} Take and fix $\chi \in C_c^\infty([-2,2])$ satisfying $0 \leq \chi \leq 1,$ and $\chi \equiv 1$ on $[-1,1]$. We then define 
\begin{align} \label{eq-part-unity-1}
f_J(x) := \frac{\chi(d(x_J,x))}{\sum\limits_{J^{\prime} \in \mathbb{N}} \chi(d(x_{J^{\prime}},x))}.
\end{align} 

Clearly, $1 \equiv \sum_{J \in \mathbb{N}} f_J(x)$ on $\mathbb{R}^{n_1 + n_2}$, and the smoothness of each $f_J$ is a consequence of the fact that $d$ is smooth away from the diagonal. Furthermore, it follows from the definition of $\chi$ that any $x_j^\prime$ or $x_k^{\prime \prime}$ derivative of $\chi (d(x_J, \cdot))$ may survive only when $1 < d(x_J, x) < 2$, and in this range of $x$ it suffices to show that derivatives of $ \left( d(x_J, x) \right)^4$ are uniformly bounded for all $J \in \mathbb{N}$. Now, $\left|x^\prime - x_J^\prime \right| \leq d(x_J, x) < 2$ and therefore it is straightforward to see that the derivatives of $\left|x^\prime - x_J^\prime \right|^4$ are uniformly bounded in $J$. Therefore, it remains to analyse the derivatives of $\left|x^{\prime \prime} - x_J^{\prime \prime}\right|^4 / w(x, x_J)$, where $w(x, x_J) = \left| x^{\prime \prime} - x_J^{\prime \prime}\right|^2 + \left(\left|x^\prime \right|^2 + \left|x_J^\prime \right|^2\right)^2$. 

Since $X^\Gamma$ corresponds to a finite linear combination of ${x^\prime}^\gamma \partial^\alpha_{x^{\prime}} \partial^{\beta}_{x^{\prime \prime}}$ with $|\gamma| \leq |\beta|$, we need to analyse the action of one such operator on $\left| x^{\prime \prime} - x_J^{\prime \prime} \right|^4 / w(x, x_J)$. For any $\beta \in \mathbb{N}^{n_2}$ we have 
\begin{align*} 
\partial^\beta_{x^{\prime \prime}} \left( \frac{ \left| x^{\prime \prime} - x_J^{\prime \prime} \right|^4} {w(x, x_J)}\right) 
= \sum_{l_2 = 0}^{|\beta|} \frac{P^{n_2, +}_{4 + 2 l_2 - |\beta|} \left(x^{\prime \prime} - x_J^{\prime \prime}\right)}{\left(w(x, x_J)\right)^{1 + l_2}}
\end{align*}
where $P^{n_2, +}_l$ stands for a polynomial of degree $l$ on $\mathbb{R}^{n_2}$ when $l \geq 0$, otherwise the function $P^{n_2, +}_l$ is $0$. 

Similarly, for any $\alpha \in \mathbb{N}^{n_1}$ and $k \in \mathbb{N}$ we have 
\begin{align*} 
\partial^\alpha_{x^{\prime}} \left( \frac{1} {\left( w(x, x_J) \right)^k} \right) 
= \sum_{l_1 = 0}^{|\alpha|} \frac{P^{n_1, n_1, +}_{4 l_1 - |\alpha|} \left(x^{\prime}, x_J^{\prime} \right)}{\left( w(x, x_J) \right)^{k + l_1}}
\end{align*} 
where $P^{n_1, n_1, +}_l$ stands for a polynomial of degree $l$ on $\mathbb{R}^{n_1} \times \mathbb{R}^{n_1}$ when $l \geq 0$, otherwise it is $0$. 

Thus, we get that 
$$ {x^\prime}^\gamma \partial^\alpha_{x^{\prime}} \partial^\beta_{x^{\prime \prime}} \left( \frac{ \left| x^{\prime \prime} - x_J^{\prime \prime} \right|^4} {w(x, x_J)}\right) = \sum_{l_1 = 0}^{|\alpha|} \sum_{l_2 = 0}^{|\beta|}  \frac{{x^\prime}^\gamma \, P^{n_1, n_1, +}_{4 l_1 - |\alpha|} \left(x^{\prime}, x_J^{\prime} \right) P^{n_2, +}_{4 + 2 l_2 - |\beta|} \left(x^{\prime \prime} - x_J^{\prime \prime}\right)}{\left( w(x, x_J) \right)^{1 + l_1 + l_2}} $$ 
with each term in the above sum being bounded by 
\begin{align*} 
\frac{\left| x^\prime\right|^{|\beta|}}{\left( w(x, x_J) \right)^{\frac{|\beta|}{4}}} 
\times 
\frac{\left| P^{n_1, n_1, +}_{4 l_1 - |\alpha|} \left(x^{\prime}, x_J^{\prime} \right) \right|}{\left( w(x, x_J) \right)^{l_1 - \frac{|\alpha|}{4}}}
\times 
\frac{\left| P^{n_2, +}_{4 + 2 l_2 - |\beta|} \left(x^{\prime \prime} - x_J^{\prime \prime}\right) \right|}{\left( w(x, x_J) \right)^{1 + \frac{l_2}{2} - \frac{|\beta|}{4}}}
\times 
\frac{1}{\left( w(x, x_J) \right)^{\frac{l_2}{2} + \frac{|\alpha|}{4}}}. 
\end{align*}
Each of the first three terms in the above product is bounded by $1$, whereas for the last term note that $1 < d(x, x_J)^4 \leq \left|x^{\prime} - x_J^\prime \right|^4 + \left|x^{\prime \prime} - x_J^{\prime \prime}\right|^2 $ implies that either $\left| x^{\prime} - x_J^{\prime} \right| > 2^{-1/4}$ or $\left|x^{\prime \prime} - x_J^{\prime \prime}\right| > 2^{-1/2}$. In any case, $w(x, x_J) \geq \max \{ \left| x^{\prime \prime} - x_J^{\prime \prime} \right|^4, \, \left(\left|x^\prime \right|^2 + \left|x_J^\prime \right|^2\right)^2\} \geq 1/32$. This completes the proof of the lemma. 
\end{proof}  


\section{\texorpdfstring{$L^2$}--boundedness in the case of \texorpdfstring{$S^0_{\rho, 0}(\boldsymbol{L}, \boldsymbol{U}), \, \rho > 0$}{}} \label{sec:rho-0-proof} 

In this section we prove that the operators corresponding to symbols from $S^0_{\rho, 0}(\boldsymbol{L}, \boldsymbol{U})$ class with $\rho > 0$, are bounded on $L^2 \left( \mathbb{R}^{n_1 + n_2} \right)$. We carefully follow the proofs of Theorem 1 of Section 2 of Chapter VI of \cite{SteinHarmonicBook93} and Theorem 5.4.17 of \cite{FischerRuzhanskyQuantizationBook}. 

\begin{thm} \label{thm:S0-rho-0-L2}
For $m \in S^0_{\rho, 0}(\boldsymbol{L}, \boldsymbol{U})$ with $\rho > 0$, the associated operator $T = m(x, \boldsymbol{L}, \boldsymbol{U})$ is bounded on $L^2 \left( \mathbb{R}^{n_1 + n_2} \right)$. In fact, it suffices to assume the symbol condition \eqref{def:grushin-symb} for $\lfloor\frac{Q}{2}\rfloor+1$ number of derivatives in $x$-variable and a total of $\lfloor\frac{Q}{2 (2\rho - 1)}\rfloor + 1$ number of derivatives in $\tau$ and $\kappa$-variables.
\end{thm}
\begin{proof} 
Choose and fix $\psi_0 \in C_c^\infty((-2,2))$ and $\psi_1 \in C_c^\infty((1/2,2))$ such that $0 \leq \psi_0, \psi_1 \leq 1$, and for all $\eta \geq 0$, 
$$\sum_{l=0}^\infty \psi_l(\eta) = 1,$$ 
where $\psi_l (\eta) = \psi_1(2^{-(l-1)}\eta),$ for $l \geq 2.$ 

We decompose the operator $m(x, \boldsymbol{L}, \boldsymbol{U})$ into a countable sum of operators as follows. For each $l \in \mathbb{N}$, define 
\begin{align} \label{S-rho-0-symbol-decompose}
m_{l}(x, \tau, \kappa) & := m(x, \tau, \kappa) \psi_l(|(\tau, \kappa)|), 
\end{align}
and denote the operators $T_{l} = m_{l}(x, \boldsymbol{L}, \boldsymbol{U})$ and $\mathcal{T}_{l} = \sum_{j = 0}^l T_{j}$, then (formally) we have 
\begin{align*} 
T = \lim_{l \to \infty} \mathcal{T}_{l}. 
\end{align*} 

We shall be done if we could show that 
$$ \sup_{l \in \mathbb{N}} \| \mathcal{T}_{l} \|_{op} < \infty. $$ 

Note that $ T_{l}$ is an integral operator. Let us denote it's kernel by $ k^{x}_{l}(x,y),$ where 
$$ k_{l}^{z} (x,y) = (2\pi)^{-n_2} \int_{\mathbb{R}^{n_2}} \sum_{\mu} m_{l} \left( z, (2 \mu + \tilde{1}) |\lambda|, - \lambda \right) E_{\mu, \lambda} (x, y) \, d\lambda. $$

In fact, the above integral is absolutely convergent. To see this note that 
\begin{align*}
& \int_{\mathbb{R}^{n_2}} \sum_{\mu} \left| m_{l} \left( z, (2 \mu + \tilde{1}) |\lambda|, - \lambda \right) E_{\mu, \lambda} (x, y) \right| \, d\lambda \\ 
& \leq \left( \int_{\mathbb{R}^{n_2}} \sum_{\mu} \left| m_{l} \left( z, (2 \mu + \tilde{1}) |\lambda|, - \lambda \right)\right| \left| \Phi_{\mu}^{\lambda} (x) \right|^2 \, d\lambda \right)^{1/2} \\ 
& \quad \quad \times \left( \int_{\mathbb{R}^{n_2}} \sum_{\mu} \left| m_{l} \left( z, (2 \mu + \tilde{1}) |\lambda|, - \lambda \right)\right| \left| \Phi_{\mu}^{\lambda} (y) \right|^2 \, d\lambda \right)^{1/2}. 
\end{align*}

For fixed $z$, one can then invoke part (i) of Theorem 6.1 of \cite{MartiniJointFunctionalCalculiMathZ} along with the orthogonality of Hermite functions $\Phi_{\mu}^{\lambda}$, to conclude that the above is bounded by 
$$\left| B \left(x, 2^{-l} \right) \right|^{-1/2} \left| B \left(y, 2^{-l} \right) \right|^{-1/2} \|m_l (z, \cdot)\|^2_{\infty}.$$

With $C_0$ as in \eqref{quasi-metric-constant}, let $\mathcal{V} \in C^\infty_c (\mathbb{R})$ be such that $0 \leq \mathcal{V} \leq 1$, $\supp (\mathcal{V}) \subseteq [-4 C_0, 4 C_0]$ and $\mathcal{V} \equiv 1$ on the interval $[-2 C_0, 2 C_0]$. Given $u \in \mathbb{R}^{n_1 + n_2}$, we define $\mathcal{V}_u(x) = \mathcal{V} \left( d(u,x) \right)$. Clearly, $\supp \left(\mathcal{V}_u \right) \subseteq B(u, 4 C_0)$ and $\mathcal{V}_u \equiv 1$ on $B(u, 2C_0)$. As done in the proof of Lemma \ref{lem-grushin-partition-part-fix}, one can show that for every $\Gamma \in \mathbb{N}^{n_1 + 2 n_2}$, there exists a constant $C_{\Gamma} > 0$ such that $\left\| X^{\Gamma} \mathcal{V}_u \right\|_\infty \leq C_{\Gamma}$ for all $u \in \mathbb{R}^{n_1 + n_2}$. 

We shall show that there exists a constant $C > 0$ such that 
$$ \| \mathcal{T}_{l} f\|_2 \leq C \|f\|_2, $$
for all $l \in \mathbb{N}$ and for any $f \in \mathcal{S} \left( \mathbb{R}^{n_1 + n_2} \right)$. 
For this we claim that the following estimate holds true: There exists $r > \frac{Q}{2\rho}$ such that 
\begin{align} \label{claim-xx}
\int_{B(u,1)} | \mathcal{T}_{l} f(y)|^2 \, dy \lesssim_r \|m\|_{\rho, 0}^2 \int_{\mathbb{R}^{n_1 + n_2}} \left( 1 + d(y, u) \right)^{-2r} \left| f(y) \right|^2 \, dy, 
\end{align}
for all $u \in \mathbb{R}^{n_1 + n_2}$. The exact choice of $r$ will become clear towards the end of the proof of the theorem.

Assuming inequality \eqref{claim-xx} for the time being, we multiply both sides of it by $ |B(u,1)|^{-1} $. Then, integrating the left hand side of the inequality over $u \in \mathbb{R}^{n_1 + n_2}$ we get 
\begin{align*}
& \int_{\mathbb{R}^{n_1 + n_2}} |B(u,1)|^{-1} \left( \int_{B(u,1)} | \mathcal{T}_{l} f (y)|^2 \, dy \right) du \\ 
& = \int_{\mathbb{R}^{n_1 + n_2}} \int_{\mathbb{R}^{n_1 + n_2}} |B(u,1)|^{-1} \mathbbm{1}_{d(u,y) < 1} | \mathcal{T}_{l} f (y)|^2 \, du \, dy \\ 
& \gtrsim \int_{\mathbb{R}^{n_1 + n_2}} |B(y,1)|^{-1} \left( \int_{\mathbb{R}^{n_1 + n_2}} \mathbbm{1}_{d(u,y) < 1} \, du \right) | \mathcal{T}_{l} f (y)|^2 \, dy \\ 
& = \int_{\mathbb{R}^{n_1 + n_2}} | \mathcal{T}_{l} f (y)|^2 \, dy. 
\end{align*}

Similarly, integrating the right hand side of inequality \eqref{claim-xx} (multiplied by $|B(u,1)|^{-1} $) over $u \in \mathbb{R}^{n_1 + n_2}$ we get 
\begin{align*}
& \|m\|_{\rho, 0}^2 \int_{\mathbb{R}^{n_1 + n_2}} |B(u,1)|^{-1} \left( \int_{\mathbb{R}^{n_1 + n_2}} \left( 1 + d(y, u) \right)^{-2r} \left| f(y) \right|^2 \, dy \right) du \\ 
& = \|m\|_{\rho, 0}^2 \int_{\mathbb{R}^{n_1 + n_2}} \left(\int_{\mathbb{R}^{n_1 + n_2}}  |B(u,1)|^{-1} \left( 1 + d(y, u) \right)^{-2r} \, du \right) \left| f(y) \right|^2 \, dy \\ 
& \lesssim_r \|m\|_{\rho, 0}^2 \int_{\mathbb{R}^{n_1 + n_2}} \left| f(y) \right|^2 \, dy, 
\end{align*}
whenever $r > Q/2$. Here the last inequality could be proved by a slight modification of Lemma \ref{lem-integrability-ball-growth}. 

The two estimates above complete the proof of Theorem \ref{thm:S0-rho-0-L2}. So, we will be done if we could establish the inequality \eqref{claim-xx}. For this, let us fix $u \in \mathbb{R}^{n_1 + n_2}$, and write $f = f_1 + f_2$ where $f_1 = f \mathcal{V}_u$. 

\medskip \noindent \underline{\textbf{Step 1}} : We claim that for any $r \geq 0$
$$ \int_{B(u, 1)} | \mathcal{T}_{l} f _1(x)|^2 \, dx \lesssim_r \|m\|^2_{S^0_{\rho, 0}} \int_{\mathbb{R}^{n_1 + n_2}} \left( 1 + d(y, u) \right)^{-2r} |f_1(y)|^2 \, dy. $$

To prove this claim, take $x \in B(u,1)$, then 
\begin{align*}
\mathcal{T}_{l} f_1(x) = \int_{\mathbb{R}^{n_1 + n_2}} \sum_{j = 0}^l k_{j}^x (x,y) f_1(y) \, dy = \int_{\mathbb{R}^{n_1 + n_2}} \mathcal{V}_u(x) \sum_{j = 0}^l k_{j}^x (x,y) f_1(y) \, dy,
\end{align*}
which implies that 
\begin{align*}
\left| \mathcal{T}_{l} f_1(x) \right|^2 &= \left| \int_{\mathbb{R}^{n_1 + n_2}} \mathcal{V}_u(x) \sum_{j = 0}^l k_{j}^x (x,y) f_1(y) \, dy \right|^2 \\ 
& \lesssim \sup_z |B(z,1)|^{-1} \left| |B(z,1)|^{1/2} \int_{\mathbb{R}^{n_1 + n_2}} \mathcal{V}_u(z) \sum_{j = 0}^l k_{j}^z (x,y) f_1(y) \, dy \right|^2. 
\end{align*}

Now, by Lemma \ref{grushin-Sobolev-embed} we have 
\begin{align*}
& \left| |B(z,1)|^{1/2} \int_{\mathbb{R}^{n_1 + n_2}} \mathcal{V}_u(z) \sum_{j = 0}^l k_{j}^z(x,y) f_1(y) \, dy \right|^2 \\ 
& \lesssim \sum_{|\Gamma| \leq 2 \left( 1 + \lfloor\frac{Q}{4}\rfloor \right)} \int_{\mathbb{R}^{n_1 + n_2}} \left| \int_{\mathbb{R}^{n_1 + n_2}} X_{z_0}^\Gamma \left( \mathcal{V}_u(z_0) \sum_{j = 0}^l k_{j}^{z_0}(x,y) \right) f_1(y) \, dy \right|^2 dz_0 \\ 
& = \sum_{|\Gamma| \leq 2 \left( 1 + \lfloor\frac{Q}{4}\rfloor \right)} \int_{B(u, 4 C_0)} \left| \int_{\mathbb{R}^{n_1 + n_2}} X_{z_0}^\Gamma \left( \mathcal{V}_u(z_0) \sum_{j = 0}^l k_{j}^{z_0}(x,y) \right) f_1(y) \, dy \right|^2 dz_0. 
\end{align*}

Finally, since $|B(z,1)| \sim |B(u, 4 C_0)|$ for all $z \in B(u, 4 C_0)$, we get 
\begin{align*}
& |B(u, 4 C_0)| \int_{x \in B(u,1)} \left| \mathcal{T}_{l} f_1(x) \right|^2 \, dx \\ 
& \lesssim \sum_{|\Gamma| \leq 2 \left( 1 + \lfloor\frac{Q}{4}\rfloor \right)} \int\limits_{B(u, 4 C_0)} \left( \int_{B(u,1)} \left| \int_{\mathbb{R}^{n_1 + n_2}} X_{z_0}^\Gamma \left( \mathcal{V}_u(z_0) \sum_{j = 0}^l k_{j}^{z_0}(x,y) \right) f_1(y) \, dy \right|^2 dx \right) dz_0. 
\end{align*} 
Dominating the above by taking the integration in $x$-variable over full space $\mathbb{R}^{n_1 + n_2}$, by Plancherel's theorem and the conditions on $\mathcal{V}_u$ and the symbol function $m$ we get that the quantity in the parenthesis of the above estimate is bounded by $\| \sum_{j = 0}^l m_{j} \|^2_{S^0_{0, 0}} \|f_1\|^2_2$, uniformly in $z_0 \in B(u, 4 C_0)$. But, we also have $\| \sum_{j = 0}^l m_{j} \|_{S^0_{0, 0}} \leq \| m \|_{S^0_{0, 0}} \leq \| m \|_{S^0_{\rho, 0}}$, and hence  
\begin{align*} 
\int_{B(u, 1)} | \mathcal{T}_{l} f_1(x)|^2 \, dx \lesssim \|m\|^2_{S^0_{\rho, 0}} \int_{\mathbb{R}^{n_1 + n_2}} |f_1 (y)|^2 \, dy. 
\end{align*}
 
Now, since $f_1$ is supported in $B(u, 4 C_0)$, we have $\left( 1 + 4 C_0 \right)^{-2r} \lesssim_r \left( 1 + d(y, u) \right)^{-2r} \leq 1$ on support of $f_1$, and therefore the above estimate implies that 
\begin{align*} 
\int_{B(u, 1)} | \mathcal{T}_{l} f_1(x)|^2 \, dx \lesssim_r \|m\|^2_{S^0_{\rho, 0}} \int_{\mathbb{R}^{n_1 + n_2}} \left( 1 + d(y, u) \right)^{-2r} |f_1 (y)|^2 \, dy, 
\end{align*}
completing the proof of the claim of step 1. 

\medskip \noindent \underline{\textbf{Step 2}} : We claim that there exists $r> 0$ (whose exact choice will become clear towards the end of the proof of the theorem) such that   
$$ \int_{B(u, 1)} | \mathcal{T}_{l}  f_2 (x)|^2 \, dx \lesssim_r \|m\|^2_{S^0_{\rho, 0}} \int_{\mathbb{R}^{n_1 + n_2}} \left( 1 + d(y,u) \right)^{-2r} |f_2 (y)|^2 \, dy. $$

For this, note that 
$$ \mathcal{T}_{l} f_2 (x) = \int_{y \notin B(u, 2C_0)} \left( d(x,y) \right)^r \sum_{j = 0}^l k_{j}^x(x,y) d(x,y)^{-r} f_2(y) \, dy. $$ 

Now, for any $x \in B(u,1)$ and $y \notin B(u, 2C_0)$, since 
$$d(y,u) \leq C_0 \left( d(y,x) + d(x,u) \right) < C_0 \left( d(y,x) + 1 \right) < C_0 d(y,x) + \frac{1}{2} d(y,u),$$ 
therefore 
\begin{align*}
d(y,x) \geq \frac{1}{2 C_0} d(y, u) \geq \frac{1}{4 C_0} (1 + d(y, u)). 
\end{align*}

Altogether, for $x \in B(u,1)$, by Cauchy-Schwarz inequality we have 
$$\left| \mathcal{T}_{l} f_2 (x) \right|^2 \lesssim \left( \int_{\mathbb{R}^{n_1 + n_2}} \left( d(x,y) \right)^{2r} \left| \sum_{j = 0}^l k_{j}^x(x,y) \right|^2 dy \right) \left( \int_{\mathbb{R}^{n_1 + n_2}} \left( 1 + d(y, u) \right)^{-2r} \left| f_2(y) \right|^2 \, dy \right). $$

As earlier, we can make use of Lemma \ref{grushin-Sobolev-embed} to show that 
\begin{align*}
& \left| B(u, 4 C_0) \right| \int_{\mathbb{R}^{n_1 + n_2}} \left( d(x,y) \right)^{2r} \left| \sum_{j = 0}^l k_{j}^x(x,y) \right|^2 dy \\ 
& = \left| B(u, 4 C_0) \right| \int_{\mathbb{R}^{n_1 + n_2}} \left( d(x,y) \right)^{2r} \left| \mathcal{V}_u(x) \sum_{j = 0}^l k_{j}^x(x,y) \right|^2 dy \\ 
& \lesssim \sum_{|\Gamma| \leq 2 \left( 1 + \lfloor\frac{Q}{4}\rfloor \right)} \int_{B(u, 4 C_0)} \left( \int_{\mathbb{R}^{n_1 + n_2}} \left( d(x,y) \right)^{2r} \left| X^{\Gamma}_{z_0} \left\{ \mathcal{V}_u (z_0) \sum_{j = 0}^l k_{j}^{z_0}(x,y) \right\} \right|^2 dy \right) dz_0 \\ 
& = \sum_{|\Gamma_1 + \Gamma_2| \leq 2 \left( 1 + \lfloor\frac{Q}{4}\rfloor \right)} \int_{B(u, 4 C_0)} \left| X^{\Gamma_1}_{z_0} \mathcal{V}_u (z_0) \right|^2 \left( \int_{\mathbb{R}^{n_1 + n_2}} \left( d(x,y) \right)^{2r} \left| \sum_{j = 0}^l X^{\Gamma_2}_{z_0} k_{j}^{z_0}(x,y) \right|^2 dy \right) dz_0. 
\end{align*}

Since $X^{\Gamma_1} \mathcal{V}_u$ are uniformly bounded in $u$ for all $|\Gamma_1| \leq 2 \left( 1 + \lfloor\frac{Q}{4}\rfloor \right) $, step 2 will be established if we could show that  
\begin{align} \label{ineq:claim2-part}
|B(x, 1)| \int_{\mathbb{R}^{n_1 + n_2}} \left( d(x,y) \right)^{2r} \left| \sum_{j = 0}^l k_{j}^{z_0, \Gamma} (x,y) \right|^2 dy \lesssim_r \|m\|^2_{S^0_{\rho, 0}}, \quad \textup{uniformly in } z_0 \textup{ and } l. 
\end{align}
In the above expression, $k_{j}^{z_0 ,\Gamma}(x,y) := X^{\Gamma_2}_{z_0} k_{j}^{z_0}(x,y)$. 

Applying Lemma \ref{lem:weighted-plancherel-joint-functional} (with $R = 2^{j/2}, \, j \geq 1$) we get that for all $s > r \geq 0$ we have
\begin{align*}
\left\| \left( 2^{j/2} d(x, \cdot) \right)^{r} k_{j}^{z_0, \Gamma} (x, \cdot) \right\|_2 
& \lesssim_{r, s} \left| B(x, 2^{-j/2}) \right|^{-1/2} \left\| X^{\Gamma}_{z_0} m(z_0, 2^j (\cdot, \cdot) ) \psi_1 \left( 2 |(\cdot, \cdot)| \right) \right\|_{W^{s}_{\infty}} \\ 
& \lesssim_{r, s} |B(x, 1)|^{-1/2} 2^{j Q/4} 2^{j s} \left(2^j\right)^{-s (1 + \rho)/2} \|m\|_{S^0_{\rho, 0}} \\ 
& = |B(x, 1)|^{-1/2} \left( 2^j \right)^{\frac{Q}{4} + \frac{s(1 - \rho)}{2}} \|m\|_{S^0_{\rho, 0}}. 
\end{align*}
A similar estimate holds for $k_0^{z_0, \Gamma} (x, \cdot)$ as well. 

Altogether, for all $l \geq 0$ we have 
\begin{align*}
|B(x, 1)|^{1/2} \left\| \left( d(x, \cdot) \right)^{r} k^{z_0}_{j} (x, \cdot) \right\|_2 
& \lesssim_{r, s} \left( 2^j \right)^{-\left(\frac{r}{2} - \frac{Q}{4} - \frac{s(1 - \rho)}{2} \right)} \|m\|^2_{S^0_{\rho, 0}}.
\end{align*}

Since $0 \leq 1 - \rho < 1$, taking $s = \lfloor \frac{Q}{2 (2 \rho - 1)} \rfloor + 1$, it is possible to choose $ r \in \left( Q/2, s \right)$ such that $\frac{r}{2} - \frac{Q}{4} - \frac{s(1 - \rho)}{2} > 0$. Overall, with such a choice of $s$ and $r$, the series $\sum\limits_{j \geq 0} \left( 2^j \right)^{-\left(\frac{r}{2} - \frac{Q}{4} - \frac{s(1 - \rho)}{2} \right)}$ is summable, completing the proof of the claim of step 2. 

Steps $1$ and $2$ together imply estimate \eqref{claim-xx} which completes the proof of Theorem \ref{thm:S0-rho-0-L2}. 
\end{proof}


\section{Kernel estimates} \label{sec-kernel-estimates}
The main objective of this section is to establish Proposition \ref{prop:lambda-diff-theorem} and Lemma \ref{first-layer-lem}. Through these results we establish the action of powers of $(x^{\prime} - y^{\prime})$ and that of the derivatives with respect to the $\lambda$-variable on kernels of scaled Hermite multipliers. These are the key results that will be used in later sections to get precise kernel estimates of certain operators in the proof of Theorem \ref{thm:old-joint-calc}. 

\medskip \noindent \textbf{Convention and Remark:} We have listed the parameters that we are going to use in Appendix \ref{app:notations}. Below, we write some of them for quick set-up as well as some of the conventions that we are going to frequently use, sometimes even without mentioning. 
\begin{itemize}
\item For Hermite functions $\Phi_{\mu}^{\lambda}$, we treat the index $\mu$ in all of $\mathbb{Z}^{n_1}$ with the understanding that $\Phi_{\mu}^{\lambda} = 0$ whenever $\mu \in \mathbb{Z}^{n_1} \setminus \mathbb{N}^{n_1}$. 

\item We use the notation $E_{\mu, \lambda} (x, y) = \Phi_\mu^{\lambda}(x^{\prime}) \Phi_\mu^{\lambda}(y^{\prime}) e^{- i \lambda \cdot \left( x^{\prime \prime} - y^{\prime \prime} \right)}.$ 

\item For kernel results in this section, we consider the symbol function $m (\tau)$ (respectively $m (\tau, \kappa)$) to be defined on $\mathbb{R}^{n_1}$ (respectively on $\mathbb{R}^{n_1 + n_2}$) with support of $m (\tau)$ (respectively $m (\tau, \kappa)$) being compact and contained in $\left( \mathbb{R}_+ \right)^{n_1}$ (respectively on $\left( \mathbb{R}_+ \right)^{n_1} \times \left( \mathbb{R}^{n_2} \setminus \{0\} \right)$). This is only to ensure the validity of various pointwise infinite sums in $\mu$-variable and the integrability in $\lambda$-variable. However, it can be immediately checked from the proofs that none of the constants or the number of terms in various finite linear combinations depend on the assumed compact support. 

\item We shall denote by $\Delta_j^+$ and $\Delta_j^-$ the forward and backward-difference operators defined by $\Delta_j^+ m(\mu |\lambda|) = m((\mu + e_j) |\lambda|) - m(\mu |\lambda|)$ and $\Delta_j^- m(\mu |\lambda|) = m(\mu |\lambda|) - m((\mu - e_j) |\lambda|)$, and their multi-variable versions are given by $\Delta_{+}^{\gamma} m = \left(\Delta_1^+ \right)^{\gamma_1} \cdots \left(\Delta_{n_1}^+ \right)^{\gamma_{n_1}} m$ and $\Delta_{-}^{\gamma} m = \left(\Delta_1^- \right)^{\gamma_1} \cdots \left(\Delta_{n_1}^- \right)^{\gamma_{n_1}} m$. 

\item Unless specified, the number of terms in various finite linear combinations depend only on $n_1, n_2$ and the order of $\lambda$-derivatives or powers of $(x^{\prime} - y^{\prime})$ or $(x^{\prime \prime} - y^{\prime \prime})$. 
\end{itemize}

\begin{prop} \label{prop:lambda-diff-theorem}
For any integer $N \geq 1$, we have
\begin{equation} \label{lambda-diff-eq-1}
\frac{\partial^N}{\partial |\lambda|^N} \left\{ \sum_{\mu} m((2\mu + \tilde{1})|\lambda|) \Phi_{\mu}^{\lambda}(x^{\prime}) \Phi_{\mu}^{\lambda}(y^{\prime}) \right\}
\end{equation}
is a finite linear combination of terms of the form
\begin{align} \label{eq-lambda-diff-1-1}
& \left({x^{\prime}}^{\alpha_{1}} {y^{\prime}}^{\alpha_{1}} \right) {x^{\prime}}^{\alpha_{2}} \left(|\lambda| {y^{\prime}} \right)^{\alpha_{2}} |\lambda|^{\nu} \int\limits_{[0,1]^{\nu_1} \times \Omega^{\nu_2}} \sum_{\mu} \left(\tau^{\theta_1} \partial_\tau^{\theta_2} m\right) ((2\mu + \vec{a}(\omega) + \tilde{1})|\lambda|) \\ 
\nonumber & \quad \quad \quad \quad \quad \quad \quad \quad \quad \quad \quad \quad \quad \quad \quad \quad \quad \quad \Phi_{\mu}^{\lambda}(x^{\prime}) \Phi_{\mu}^{\lambda}(y^{\prime}) g(\omega) \, d \omega, 
\end{align}
where $\alpha_1, \alpha_2, \theta_1, \theta_2 \in \mathbb{N}^{n_1}$, $\nu, \nu_1, \nu_2 \in \mathbb{N}$ with $\nu_1\leq N$, $\nu_2\leq N$, $ \left( 2 |\theta_1| + |\alpha_2| + 2 \nu \right) \leq |\theta_2| \leq 2N$, $\left|2 \alpha_{1} + \alpha_{2} \right| \leq N$, and $\left( |\theta_2| + |\alpha_1| \right) - \left( |\theta_1| + \nu \right) = N$. Also, $\Omega$ is a compact subset of $\mathbb{R}^2$, $\vec{a}(w)$ and $g(w)$ are continuous functions on $[0,1]^{\nu_1} \times \Omega^{\nu_2}$, $|\vec{a}(w)|_1 \leq 2N$, and $dw$ is the Lebesgue measure $[0,1]^{\nu_1} \times \Omega^{\nu_2}$. 
\end{prop}

\begin{rem} \label{rem-indices-clarity}
Let us closely analyse the indices in the previous proposition's statement. Consider the annihilation and creation operators 
\begin{align} \label{def-ann-cre-op}
A_j(\lambda) = \frac{\partial}{\partial x^{\prime}_j} + |\lambda| x^{\prime}_j \quad \textup{and} \quad A_j(\lambda)^* = - \frac{\partial}{\partial x^{\prime}_j} + |\lambda| x^{\prime}_j.
\end{align}  
We know that 
\begin{align} \label{annihi-creation}
A_j(\lambda) \Phi_{\mu}^{\lambda} = \left( (2 \mu_j) |\lambda| \right)^{1/2} \Phi_{\mu - e_j}^{\lambda} \quad \textup{and} \quad A_j(\lambda)^* \Phi_{\mu}^{\lambda} = \left( (2 \mu_j + 2) |\lambda|\right)^{1/2} \Phi_{\mu + e_j}^{\lambda}. 
\end{align} 

One gets the following recursion formula for Hermite functions from \eqref{annihi-creation}: 
\begin{align*} 
\left( \mu_j + 1 \right)^{1/2} \Phi_{\mu}^{\lambda}(x^{\prime}) = \left( 2|\lambda| \right)^{1/2} x^{\prime}_j \Phi_{\mu + e_j}^{\lambda}(x^{\prime}) - \left( \mu_j + 2 \right)^{1/2} \Phi_{\mu + 2 e_j}^{\lambda}(x^{\prime}),
\end{align*}
which implies that 
\begin{align*} 
|\lambda| y^{\prime}_j \Phi_{\mu}^{\lambda}(y^{\prime}) = \frac{1}{\sqrt{2}} \sqrt{|\lambda|(\mu_j + 1)} \Phi_{\mu + e_j}^{\lambda}(y^{\prime}) + \frac{1}{\sqrt{2}} \sqrt{|\lambda| \mu_j} \Phi_{\mu - e_j}^{\lambda}(y^{\prime}).
\end{align*}

In view of the above relation, one can see that $\left( |\lambda| {y^{\prime}} \right)^{\alpha_{2}} \Phi_{\mu}^{\lambda}(y^{\prime})$ can be expressed as a finite linear combination of terms of the form 
\begin{align} \label{reverse-step-lambda-2} 
d_{\mu, \omega} \left((2\mu + \vec{a}(\omega) + \tilde{1}) |\lambda| \right)^{\frac{\alpha_2}{2}} \Phi^\lambda_{\mu + \tilde{\mu}}(y^{\prime}), 
\end{align}
where $d_{\mu, \omega}$ is a bounded function of $\mu$ and $\omega$ and $-|\alpha_{2}| \leq |\tilde{\mu}| \leq |\alpha_{2}|$. 

Also, $|\lambda|^{\nu}$ could be rewritten as a linear combination of the terms which are of the form $\left|2\mu + \vec{a} (\omega) + \tilde{1} \right|_1^{-\nu} ((2\mu+\vec{a}(\omega) + \tilde{1})|\lambda|)^{\alpha_3}$, where $|\alpha_3|=\nu$. 

Collecting all powers of $\tau = (2\mu+\vec{a}(\omega) + \tilde{1}) |\lambda|$ and using the relations between the indices given in \eqref{eq-lambda-diff-1-1}, one can verify that \eqref{lambda-diff-eq-1} can be rewritten as a finite linear combination of the terms which are of the form
\begin{align} \label{eq-lambda-diff-1-2}
& \left({x^{\prime}}^{\alpha_{1}} {y^{\prime}}^{\alpha_{1}} \right) {x^{\prime}}^{\alpha_{2}} 
\int\limits_{[0,1]^{\nu_1} \times \Omega^{\nu_2}} \sum_{\mu} C_{\mu, \vec{a}} \left(\tau^{\frac{1}{2} \mathcal{\theta}_3} \partial_\tau^{{\theta}_2} m\right) ((2\mu + \vec{a}(\omega) + \tilde{1})|\lambda|) \Phi_{\mu}^{\lambda}(x^{\prime}) \Phi_{\mu+\tilde{\mu}}^{\lambda}(y^{\prime}) g(\omega) \, d \omega, 
\end{align}
where $ |\mathcal{\theta}_3| \leq |\mathcal{\theta}_2| \leq 2 N$, $|\mathcal{\theta}_2| - \frac{1}{2} |\mathcal{\theta}_3| = N -|\alpha_1| - \frac{|\alpha_2|}{2} \geq \frac{N}{2}$, and $C_{\mu, \vec{a}}$ is a bounded function of $\mu$ and $\vec{a}$. 

The main advantage of having an expression of the form \eqref{eq-lambda-diff-1-1} is that there is no shift in the Hermite functions $\Phi_{\mu}^{\lambda}(y^{\prime})$ in it which is not the case in \eqref{eq-lambda-diff-1-2}. Although \eqref{eq-lambda-diff-1-2} is good enough for our purposes, we have written the expression in the form of \eqref{eq-lambda-diff-1-1} believing that it may be useful in some other analysis. 
\end{rem}


\begin{proof}[Proof of Proposition \ref{prop:lambda-diff-theorem}] We shall prove the proposition by mathematical induction. 

\medskip \noindent {\bf Step I.} When $N=1.$ 

Recall first that 
\begin{align} \label{hermite-derivative}
\frac{\partial}{\partial |\lambda|} \Phi_{\mu}^{\lambda} & = \frac{1}{4|\lambda|} \sum_{j=1}^{n_1} \left\{-\sqrt{(\mu_j + 1)(\mu_j + 2)} \Phi_{\mu + 2e_j}^{\lambda} +  \sqrt{\mu_j(\mu_j - 1)} \Phi_{\mu - 2e_j}^{\lambda}\right\}. 
\end{align}

Therefore, 
\begin{align} \label{first-der}
& \frac{\partial}{\partial |\lambda|} \left\{\sum_{\mu} m((2\mu + \tilde{1})|\lambda|) \Phi_{\mu}^{\lambda}(x^{\prime}) \Phi_{\mu}^{\lambda}(y^{\prime}) \right\} \\ 
\nonumber &\quad \quad \quad = \sum_{\mu} \sum_{j=1}^{n_1} (2 \mu_j +1) \left(\partial_{\tau_j} m\right) ((2\mu + \tilde{1})|\lambda|) \Phi_{\mu}^{\lambda}(x^{\prime}) \Phi_{\mu}^{\lambda}(y^{\prime}) \\
\nonumber &\quad \quad \quad \quad - \frac{1}{4|\lambda|}  \sum_{\mu} \sum_{j=1}^{n_1} \sqrt{(\mu_j + 1)(\mu_j + 2)} m((2\mu + \tilde{1})|\lambda|) \Phi_{\mu + 2e_j}^{\lambda}(x^{\prime})  \Phi_{\mu}^{\lambda}(y^{\prime}) \\
\nonumber &\quad \quad \quad \quad +  \frac{1}{4|\lambda|}  \sum_{\mu} \sum_{j=1}^{n_1} \sqrt{\mu_j(\mu_j - 1)} m((2\mu + \tilde{1})|\lambda|) \Phi_{\mu - 2e_j}^{\lambda}(x^{\prime}) \Phi_{\mu}^{\lambda}(y^{\prime}) \\
\nonumber &\quad \quad \quad \quad -  \frac{1}{4|\lambda|}  \sum_{\mu} \sum_{j=1}^{n_1} \sqrt{(\mu_j + 1)(\mu_j + 2)} m((2\mu + \tilde{1})|\lambda|) \Phi_{\mu}^{\lambda}(x^{\prime}) \Phi_{\mu + 2e_j}^{\lambda}(y^{\prime}) \\
\nonumber &\quad \quad \quad \quad +  \frac{1}{4|\lambda|} \sum_{\mu} \sum_{j=1}^{n_1} \sqrt{\mu_j(\mu_j - 1)} m((2\mu + \tilde{1})|\lambda|) \Phi_{\mu}^{\lambda}(x^{\prime}) \Phi_{\mu - 2e_j}^{\lambda}(y^{\prime}) \\
\nonumber &\quad \quad =  I_1 + I_2 + I_3 + I_4 + I_5.
\end{align}

We decompose $I_1$ as follows: 
\begin{align} \label{I1-decompose}
I_1 &= \sum_{\mu} \sum_{j=1}^{n_1} \mu_j \left(\partial_{\tau_j} m\right)((2\mu + \tilde{1})|\lambda|) \Phi_{\mu}^{\lambda}(x^{\prime}) \Phi_{\mu}^{\lambda}(y^{\prime}) \\
\nonumber &\quad + \sum_{\mu} \sum_{j=1}^{n_1} (\mu_j +1) \left(\partial_{\tau_j} m\right)((2\mu + \tilde{1})|\lambda|) \Phi_{\mu}^{\lambda}(x^{\prime}) \Phi_{\mu}^{\lambda}(y^{\prime}) \\
\nonumber &= I^1_1 + I^2_1.
\end{align}

Next, we rewrite $I_3$ and $I_4$ as follows: 
\begin{align} \label{I3-I4-rewrite}
I_3 &= \frac{1}{4 |\lambda|} \sum_{\mu} \sum_{j=1}^{n_1} \sqrt{\mu_j(\mu_j + 1)} m((2\mu + 2e_j + \tilde{1})|\lambda|) \Phi_{\mu - e_j}^{\lambda}(x^{\prime}) \Phi_{\mu + e_j}^{\lambda}(y^{\prime}), \\
\nonumber \quad I_4 &= -\frac{1}{4 |\lambda|} \sum_{\mu} \sum_{j=1}^{n_1} \sqrt{\mu_j (\mu_j + 1)} m((2\mu - 2e_j + \tilde{1})|\lambda|) \Phi_{\mu - e_j}^{\lambda}(x^{\prime}) \Phi_{\mu + e_j}^{\lambda}(y^{\prime}).
\end{align}

Therefore, 
\begin{align} \label{I3I4}
I_3 + I_4 &= \frac{1}{4 |\lambda|} \sum_{\mu} \sum_{j=1}^{n_1} \sqrt{\mu_j (\mu_j + 1)} \left(\left(\Delta_j^+ + \Delta_j^-\right) m((2\mu + \tilde{1})|\lambda|)\right) \Phi_{\mu - e_j}^{\lambda}(x^{\prime}) \Phi_{\mu + e_j}^{\lambda}(y^{\prime})\\ 
\nonumber &= \frac{1}{2 \sqrt{2} |\lambda|^{1/2}} \sum_{\mu} \sum_{j=1}^{n_1} x_j^{\prime}  \sqrt{\mu_j + 1} \left(\left(\Delta_j^+ + \Delta_j^-\right)m((2\mu + \tilde{1})|\lambda|)\right) \Phi_{\mu}^{\lambda}(x^{\prime}) \Phi_{\mu + e_j}^{\lambda}(y^{\prime}) \\
\nonumber & \quad - \frac{1}{4 |\lambda|} \sum_{\mu} \sum_{j=1}^{n_1} (\mu_j + 1) \left(\left(\Delta_j^+ + \Delta_j^-\right)m((2\mu + \tilde{1})|\lambda|)\right) \Phi_{\mu + e_j}^{\lambda}(x^{\prime}) \Phi_{\mu + e_j}^{\lambda}(y^{\prime}) \\
\nonumber &= \frac{1}{2 \sqrt{2} |\lambda|^{1/2}} \sum_{\mu} \sum_{j=1}^{n_1} x_j^{\prime} \sqrt{\mu_j + 1} \left(\Delta_j^+ m\right)((2\mu + \tilde{1})|\lambda|) \Phi_{\mu}^{\lambda}(x^{\prime}) \Phi_{\mu + e_j}^{\lambda}(y^{\prime}) \\
\nonumber &\quad + \frac{1}{2 \sqrt{2} |\lambda|^{1/2}} \sum_{\mu} \sum_{j=1}^{n_1} x_j^{\prime} \sqrt{\mu_j + 1} \left(\Delta_j^- m\right)((2\mu + \tilde{1})|\lambda|) \Phi_{\mu}^{\lambda}(x^{\prime}) \Phi_{\mu + e_j}^{\lambda}(y^{\prime}) \\
\nonumber &\quad - \frac{1}{4 |\lambda|} \sum_{\mu} \sum_{j=1}^{n_1} \mu_j \left(\Delta_j^+ m\right)((2\mu - 2e_j + \tilde{1})|\lambda|) \Phi_{\mu}^{\lambda}(x^{\prime}) \Phi_{\mu}^{\lambda}(y^{\prime}) \\
\nonumber &\quad - \frac{1}{4 |\lambda|} \sum_{\mu} \sum_{j=1}^{n_1} \mu_j \left(\Delta_j^- m\right)((2\mu - 2e_j + \tilde{1})|\lambda|)  \Phi_{\mu}^{\lambda}(x^{\prime}) \Phi_{\mu}^{\lambda}(y^{\prime}) \\
\nonumber &= I^1_{3,4,+} + I^1_{3,4,-} + I^2_{3,4,+} + I^2_{3,4,-}.
\end{align}

Note that 
\begin{align} \label{I134}
I^1_{3,4,+} &= \frac{1}{2 \sqrt{2} |\lambda|^{1/2}} \sum_{\mu} \sum_{j=1}^{n_1} x_j^{\prime} \sqrt{\mu_j + 1} \left(\Delta_j^+ m\right)((2\mu + \tilde{1})|\lambda|) \Phi_{\mu}^{\lambda}(x^{\prime}) \Phi_{\mu + e_j}^{\lambda}(y^{\prime}) \\
\nonumber &= \frac{1}{2} \sum_{\mu} \sum_{j=1}^{n_1} x_j^{\prime} y_j^{\prime} \left(\Delta_j^+ m\right)((2\mu + \tilde{1})|\lambda|) \Phi_{\mu}^{\lambda}(x^{\prime}) \Phi_{\mu}^{\lambda}(y^{\prime}) \\ 
\nonumber &\quad -\frac{1}{2 \sqrt{2} |\lambda|^{1/2}} \sum_{\mu} \sum_{j=1}^{n_1} x_j^{\prime} \sqrt{\mu_j} \left(\Delta_j^+ m\right)((2\mu + \tilde{1})|\lambda|) \Phi_{\mu}^{\lambda}(x^{\prime}) \Phi_{\mu - e_j}^{\lambda}(y^{\prime}). 
\end{align}
One can estimate $I^1_{3,4,-}$ analogously. 

We perform an analysis analogous to (\ref{I3I4}) on $I_2 + I_5$ to get 
\begin{align} \label{I2I5}
I_2 + I_5 &= \frac{1}{2 \sqrt{2} |\lambda|^{1/2}} \sum_{\mu} \sum_{j=1}^{n_1} x_j^{\prime} \sqrt{\mu_j} \left(\Delta_j^+ m\right)((2\mu + \tilde{1}) |\lambda|) \Phi_{\mu}^{\lambda}(x^{\prime}) \Phi_{\mu - e_j}^{\lambda}(y^{\prime}) \\ 
\nonumber &\quad + \frac{1}{2 \sqrt{2} |\lambda|^{1/2}} \sum_{\mu} \sum_{j=1}^{n_1} x_j^{\prime} \sqrt{\mu_j} \left(\Delta_j^- m\right)((2\mu + \tilde{1}) |\lambda|) \Phi_{\mu}^{\lambda}(x^{\prime}) \Phi_{\mu - e_j}^{\lambda}(y^{\prime}) \\ 
\nonumber &\quad - \frac{1}{4 |\lambda|} \sum_{\mu} \sum_{j=1}^{n_1} (\mu_j +1) \left(\Delta_j^+ m\right)((2\mu + 2e_j + \tilde{1}) |\lambda|) \Phi_{\mu}^{\lambda}(x^{\prime}) \Phi_{\mu}^{\lambda}(y^{\prime}) \\
\nonumber &\quad - \frac{1}{4 |\lambda|} \sum_{\mu} \sum_{j=1}^{n_1} (\mu_j +1) \left(\Delta_j^- m\right)((2\mu + 2e_j + \tilde{1}) |\lambda|) \Phi_{\mu}^{\lambda}(x^{\prime}) \Phi_{\mu}^{\lambda}(y^{\prime}) \\
\nonumber &= I^1_{2,5,+} +I^1_{2,5,-} + I^2_{2,5,+} + I^2_{2,5,-}.
\end{align}

From (\ref{I134}) and (\ref{I2I5}), we get that
\begin{align} \label{I134I2I5}
I^1_{3,4,+} + I^1_{2,5,+} & = \frac{1}{2} \sum_{\mu} \sum_{j=1}^{n_1} x_j^{\prime} y_j^{\prime} \left(\Delta_j^+ m\right)((2\mu + \tilde{1}) |\lambda|) \Phi_{\mu}^{\lambda}(x^{\prime}) \Phi_{\mu}^{\lambda}(y^{\prime}) \\
\nonumber & = \sum_\mu \sum_{j=1}^{n_1} (x_j^{\prime} y_j^{\prime}) |\lambda| \left\{ \frac{m((2\mu + 2e_j + \tilde{1}) |\lambda|) - m((2\mu + \tilde{1})|\lambda|)}{2 |\lambda|} \right\} \Phi_{\mu}^{\lambda}(x^{\prime}) \Phi_{\mu}^{\lambda}(y^{\prime}) \\
\nonumber & = \sum_\mu \sum_{j=1}^{n_1} (x_j^{\prime} y_j^{\prime}) |\lambda| \left\{ \int_{0}^1 \frac{\partial}{\partial \tau_j} \{m((2(\mu + s e_j) + \tilde{1})|\lambda|)\} \, ds \right\} \Phi_{\mu}^{\lambda}(x^{\prime}) \Phi_{\mu}^{\lambda}(y^{\prime}).
\end{align}
Similarly, we can estimate $I^1_{3,4,-} + I^1_{2,5,-}$. 

Finally, we shall estimate $\frac{1}{2} I^1_1 + I^2_{3,4,+}$, $\frac{1}{2} I^1_1 + I^2_{3,4,-}$, $\frac{1}{2} I^2_1 + I^2_{2,5,+}$, and $\frac{1}{2} I^2_1 + I^2_{2,5,-}$. For this, note that 
\begin{align*}
\frac{1}{2} I^1_1 + I^2_{3,4,+} &= \frac{1}{2} \sum_{\mu} \sum_{j=1}^{n_1} \mu_j \left(\partial_{\tau_j} m\right) ((2\mu + \tilde{1})|\lambda|) \Phi_{\mu}^{\lambda}(x^{\prime}) \Phi_{\mu}^{\lambda}(y^{\prime}) \\
& \quad - \frac{1}{4 |\lambda|} \sum_{\mu} \sum_{j=1}^{n_1} \mu_j \Delta_j^+ \left(m((2\mu -2e_j + \tilde{1})|\lambda|)\right) \Phi_{\mu}^{\lambda}(x^{\prime}) \Phi_{\mu}^{\lambda}(y^{\prime}). 
\end{align*}

Now, applying the fundamental theorem of calculus on the terms involving the difference operators $\Delta_j^+$, the above expression implies that 
\begin{align} \label{I11I234}
& \frac{1}{2} I^1_1 + I^2_{3,4,+} \\ 
\nonumber &= \frac{1}{2} \sum_{\mu} \sum_{j=1}^{n_1} \mu_j \int_{-1}^0 \left\{ \left(\partial_{\tau_j} m\right) ((2\mu + \tilde{1})|\lambda|) - \left(\partial_{\tau_j} m\right) ((2(\mu + s e_j) + \tilde{1})|\lambda|) \right\} \, ds \Phi_{\mu}^{\lambda}(x^{\prime}) \Phi_{\mu}^{\lambda}(y^{\prime}) \\ 
\nonumber &= \sum_{\mu} \sum_{j=1}^{n_1} |\lambda| \mu_j \int_{-1}^0 \int_{s}^0 \left(\partial^2_{\tau_j} m\right) ((2(\mu + s^{\prime} e_j) + \tilde{1})|\lambda|) \, ds^\prime \, ds \, \Phi_{\mu}^{\lambda}(x^{\prime}) \Phi_{\mu}^{\lambda}(y^{\prime}) \\ 
\nonumber &= \frac{1}{2} \sum_{\mu} \sum_{j=1}^{n_1} \int_{-1}^0 \int_{s}^0 (2(\mu_j + s^{\prime}) + 1) |\lambda| \left(\partial^2_{\tau_j} m\right) ((2(\mu + s^{\prime} e_j) + \tilde{1}) |\lambda|) \, ds^\prime \, ds \, \Phi_{\mu}^{\lambda}(x^{\prime}) \Phi_{\mu}^{\lambda}(y^{\prime}) \\ 
\nonumber &\quad - \frac{1}{2} \sum_{\mu} \sum_{j=1}^{n_1} \int_{-1}^0 \int_{s}^0 (2s^\prime+1) |\lambda| \left(\partial^2_{\tau_j} m\right) ((2(\mu + s^{\prime} e_j) + \tilde{1}) |\lambda|) \, ds^\prime \, ds \, \Phi_{\mu}^{\lambda}(x^{\prime}) \Phi_{\mu}^{\lambda}(y^{\prime}). 
\end{align}

We can estimate $\frac{1}{2} I^1_1 + I^2_{3,4,-}, \, \frac{1}{2} I^2_1 + I^2_{2,5,+}$ and $\frac{1}{2} I^2_1 + I^2_{2,5,-}$ in a similar way. Putting the estimates from (\ref{I134I2I5}) and (\ref{I11I234}) into (\ref{first-der}), the claimed estimate (\ref{lambda-diff-eq-1}) holds true for $N=1$.

\medskip \noindent {\bf Step II.} When $N \geq 2.$ 

Let the proposition be true for some $N\in \mathbb{N}$. We will prove it for $N+1$. Let us consider a prototype term from the $N$-th derivative (as in (\ref{lambda-diff-eq-1})). We write $\theta = (\theta_1, \theta_2)$ and $F_{\theta, \vec{a}}(|\lambda|) := \sum_{\mu} m_{\theta}((2 \mu + \vec{a}(s,s^{\prime}) + \tilde{1}) |\lambda|) \Phi_{\mu}^{\lambda}(x^{\prime}) \Phi_{\mu}^{\lambda}(y^{\prime}),$ where $m_{\theta}(\tau) = \left(\tau^{\theta_1} \partial_\tau^{\theta_2} m\right)(\tau).$

Now, 
\begin{align*} 
& \frac{\partial}{\partial |\lambda|} \left\{ \left({x^{\prime}}^{\alpha_{1}} {y^{\prime}}^{\alpha_{1}} \right) {x^{\prime}}^{\alpha_{2}} \left(|\lambda| {y^{\prime}} \right)^{\alpha_{2}} |\lambda|^{\nu}  F_{\theta, \vec{a}} \right\} (|\lambda|) \\ 
&\quad = \left({x^{\prime}}^{\alpha_{1}} {y^{\prime}}^{\alpha_{1}} \right) {x^{\prime}}^{\alpha_{2}} \frac{\partial}{\partial |\lambda|} \left\{ \left(|\lambda| {y^{\prime}} \right)^{\alpha_{2}} |\lambda|^{\nu} \right\} F_{\theta, \vec{a}} (|\lambda|) \\ 
&\quad \quad + \left({x^{\prime}}^{\alpha_{1}} {y^{\prime}}^{\alpha_{1}} \right) {x^{\prime}}^{\alpha_{2}} \left(|\lambda| {y^{\prime}} \right)^{\alpha_{2}} |\lambda|^{\nu} \frac{\partial}{\partial |\lambda|} F_{\theta, \vec{a}}(|\lambda|) \\ 
&\quad = |\alpha_{2}| \left({x^{\prime}}^{\alpha_{1} + e_j} {y^{\prime}}^{\alpha_{1} + e_j} \right) {x^{\prime}}^{\alpha_{2} - e_j} \left(|\lambda| {y^{\prime}} \right)^{\alpha_{2} - e_j} |\lambda|^{\nu} F_{\theta, \vec{a}}(|\lambda|) \\ 
&\quad \quad + \nu \left({x^{\prime}}^{\alpha_{1}} {y^{\prime}}^{\alpha_{1}} \right) {x^{\prime}}^{\alpha_{2}} \left(|\lambda| {y^{\prime}} \right)^{\alpha_{2}} |\lambda|^{\nu - 1} F_{\theta, \vec{a}}(|\lambda|) \\ 
&\quad \quad + \left({x^{\prime}}^{\alpha_{1}} {y^{\prime}}^{\alpha_{1}} \right) {x^{\prime}}^{\alpha_{2}} \left(|\lambda| {y^{\prime}} \right)^{\alpha_{2}} |\lambda|^{\nu} \frac{\partial}{\partial |\lambda|} F_{\theta, \vec{a}}(|\lambda|), 
\end{align*}
where the first term appears only when there exists some $j$ such that $(\alpha_2)_j \geq 1$ and the second term appears only when $\nu \geq 1.$

Clearly, using the induction hypothesis, the indices in the first term satisfy $|2(\alpha_{1} + e_j) + (\alpha_{2} - e_j)| = |2\alpha_{1} + \alpha_{2}| + 1 \leq N+1$, 
$ \left( 2 |\theta_1| + |\alpha_2 - e_j| + 2 \nu \right) \leq \left( 2 |\theta_1| + |\alpha_2| + 2 \nu \right) \leq |\theta_2| \leq 2N \leq 2(N+1)$, and $\left( |\theta_2| + |\alpha_1 + e_j| \right) - \left( |\theta_1| + \nu \right) = N+1$. Similarly, in the second term, the appearance of $\nu - 1$ in the power of $|\lambda|$ implies that the term corresponding to it in the $(N+1)$-th derivative is of the claimed form. So, we are only left to analyse $\frac{\partial}{\partial |\lambda|} F_{\theta, \vec{a}}(|\lambda|),$ and it follows from the analysis of the case of $N=1$. To see this, note that 
\begin{align} \label{induction-est-2}
\frac{\partial}{\partial |\lambda|} F_{\theta, \vec{a}}(|\lambda|) &= \sum_{\mu} \sum_{j=1}^{n_1} (2\mu_j + a_j(s,s^{\prime}) + 1) \left(\partial_{\tau_j} m_{\theta}\right) ((2\mu + \vec{a}(s,s^{\prime}) + \tilde{1})|\lambda|) \Phi_{\mu}^{\lambda}(x^{\prime}) \Phi_{\mu}^{\lambda}(y^{\prime}) \\
\nonumber &\quad + \sum_{\mu} m_{\theta} (x, (2\mu + \vec{a}(s,s^{\prime}) + \tilde{1})|\lambda|) \frac{\partial}{\partial |\lambda|} \left(\Phi_{\mu}^{\lambda}(x^{\prime}) \Phi_{\mu}^{\lambda}(y^{\prime})\right) \\
\nonumber &= \sum_{\mu} \sum_{j=1}^{n_1} a_j(s,s^{\prime}) \left(\partial_{\tau_j} m_{\theta}\right) ((2\mu + \vec{a}(s,s^{\prime}) + \tilde{1}) |\lambda|) \Phi_{\mu}^{\lambda}(x^{\prime}) \Phi_{\mu}^{\lambda}(y^{\prime}) \\
\nonumber &\quad + \sum_{\mu} \sum_{j=1}^{n_1} (2\mu_j + 1) \left(\partial_{\tau_j} m_{\theta}\right) ((2\mu + \vec{a}(s,s^{\prime}) + \tilde{1}) |\lambda|) \Phi_{\mu}^{\lambda}(x^{\prime}) \Phi_{\mu}^{\lambda}(y^{\prime}) \\
\nonumber &\quad + \sum_{\mu} m_{\theta} ((2\mu + \vec{a}(s,s^{\prime}) + \tilde{1})|\lambda|) \frac{\partial}{\partial |\lambda|} \left(\Phi_{\mu}^{\lambda}(x^{\prime}) \Phi_{\mu}^{\lambda}(y^{\prime})\right) \\
\nonumber &= I_{0, N+1} + I_{1,N+1} + (I_{2,N+1} + I_{3,N+1} + I_{4,N+1} + I_{5,N+1}), 
\end{align}
where $I_{2, N+1}$, $I_{3, N+1}$, $I_{4, N+1}$ and $I_{5, N+1}$ are similar to the terms $I_2$, $I_3$, $I_4$ and $I_5$ appearing in \eqref{first-der} with some obvious changes. 

We first show that the term corresponding to $I_{0, N+1}$ of the above sum in the $(N+1)$-th derivative is of the claimed form. In fact, the term corresponding to $I_{0, N+1}$ equals to a constant multiple of $\left({x^{\prime}}^{\alpha_{1}} {y^{\prime}}^{\alpha_{1}} \right) {x^{\prime}}^{\alpha_{2}} \left(|\lambda| {y^{\prime}} \right)^{\alpha_{2}} |\lambda|^{\nu}$ and 
\begin{align*}
& \int_{[0,1]^{\nu_1} \times \Omega^{\nu_2}} \sum_{\mu} \sum_{j=1}^{n_1} \left(\partial_{\tau_j} m_{\theta} \right) ((2\mu + \vec{a}(s,s^{\prime}) + \tilde{1})|\lambda|) \Phi_{\mu}^{\lambda}(x^{\prime}) \Phi_{\mu}^{\lambda}(y^{\prime}) a_j(s,s^{\prime}) g(s, s^{\prime}) \, ds \, ds^{\prime} \\ 
& = \sum_{j=1}^{n_1} (\theta_1)_j \int_{[0,1]^{\nu_1} \times \Omega^{\nu_2}} \sum_{\mu} \left(\tau^{\theta_1 - e_j} \partial_{\tau_j}^{\theta_2} m\right) ((2\mu + \vec{a}(s,s^{\prime}) + \tilde{1})|\lambda|) \Phi_{\mu}^{\lambda}(x^{\prime}) \Phi_{\mu}^{\lambda}(y^{\prime}) \\ 
& \quad \quad \quad \quad \quad \quad \quad \quad \quad \quad \quad \quad \quad a_j(s,s^{\prime}) g(s, s^{\prime}) \, ds \, ds^{\prime} \\ 
& \quad + \sum_{j=1}^{n_1} \int_{[0,1]^{\nu_1} \times \Omega^{\nu_2}} \sum_{\mu} \left(\tau^{\theta_1} \partial_{\tau_j}^{\theta_2 + e_j} m\right) ((2\mu + \vec{a}(s,s^{\prime}) + \tilde{1})|\lambda|) \Phi_{\mu}^{\lambda}(x^{\prime}) \Phi_{\mu}^{\lambda}(y^{\prime}) \\ 
& \quad \quad \quad \quad \quad \quad \quad \quad \quad \quad a_j(s,s^{\prime}) g(s, s^{\prime}) \, ds \, ds^{\prime}. 
\end{align*}
In the above expression, the $j$-th term in the first sum will occur only if $(\theta_1)_j > 0$. Clearly, both the terms in the above sum are of the claimed form. 
 
Finally, one can verify that the analysis for $I_{1,N+1} + (I_{2,N+1} + I_{3,N+1} + I_{4,N+1} + I_{5,N+1})$ of (\ref{induction-est-2}) could be done exactly same as that for (\ref{first-der}) of case $N=1.$ Importantly, the dependence of the multiplier function $m_{\theta}(\tau) = \left(\tau^{\theta_1} \partial_\tau^{\theta_2} m\right)(\tau)$ on $\vec{a}(s,s^{\prime})$ does not affect the analysis leading to the estimates (\ref{I134I2I5}) and (\ref{I11I234}). Notice that the operator acting on the multiplier corresponding to the estimate (\ref{I134I2I5}) is a finite linear combination of terms of the form $x^{\prime}_j y^{\prime}_j |\lambda| \partial{\tau_j} \left(\tau^{\theta_1} \partial_{\tau}^{\theta_2} \right),$ whereas there are two prototype operators corresponding to the estimate (\ref{I11I234}), namely, $|\lambda| \partial_{\tau_j}^2 \left(\tau^{\theta_1} \partial_{\tau}^{\theta_2}\right)$ and $\tau_j \partial_{\tau_j}^2 \left(\tau^{\theta_1} \partial_{\tau}^{\theta_2}\right).$ After applying Leibniz rule on $\tau$-derivative(s), one gets that all these terms independently satisfy the claimed form with some of the $\nu$, $\alpha_j$'s and $\theta_j$'s possibly altered. In summary, the conditions in the terms corresponding to the above mentioned prototype operators are satisfied as per the the following rules, where the indices which are not mentioned remain unchanged: 

\begin{itemize}
\item $ x^{\prime}_j y^{\prime}_j |\lambda| \partial_{\tau_j} \left( \tau^{\theta_1} \partial_{\tau}^{\theta_2} \right) $ : $(\alpha_2, \theta_1, \theta_2)$ replaced by either $(\alpha_2 + e_j, \theta_1 - e_j, \theta_2)$ or $(\alpha_2 + e_j, \theta_1, \theta_2 + e_j)$; 

\item $ |\lambda| \partial^2_{\tau_j} \left( \tau^{\theta_1} \partial_{\tau}^{\theta_2} \right) $ : $(\nu, \theta_1, \theta_2)$ replaced by either $(\nu + 1, \theta_1 - 2 e_j, \theta_2)$, $(\nu + 1, \theta_1 - e_j, \theta_2 + e_j)$ or $(\nu + 1, \theta_1, \theta_2 + 2 e_j)$; 

\item $ \tau_j \partial_{\tau_j}^2 \left( \tau^{\theta_1} \partial_{\tau}^{\theta_2} \right) $ : $(\theta_1, \theta_2)$ replaced by either $(\theta_1 - e_j, \theta_2)$, $(\theta_1, \theta_2 + e_j)$ or $(\theta_1 + e_j, \theta_2 + 2 e_j)$. 
\end{itemize} 

This completes the proof of Proposition \ref{prop:lambda-diff-theorem}. 
\end{proof}


\begin{rem} \label{remark-lambda-diff-theorem}
The conclusion of Proposition \ref{prop:lambda-diff-theorem} holds true if instead of $m((2\mu + \tilde{1})|\lambda|)$ one works with the symbol function $m((2\mu + \tilde{a})|\lambda|)$, for a fixed $\tilde{a} \in \mathbb{R}^{n_1}$. This could be checked, for example, from the estimation of $I_{0, N+1}$ which was done immediately after \eqref{induction-est-2}. When $N=1$, this extra term would correspond to the case of $|\theta_2| = 1$. Note also that while the implicit constants may depend on $\tilde{a}$ but the number of terms in the finite linear combination claimed in \eqref{lambda-diff-eq-1} won't depend on $\tilde{a}$. 
\end{rem}


From Proposition \ref{prop:lambda-diff-theorem} and Remark \ref{remark-lambda-diff-theorem} one can easily derive the following corollary.

\begin{cor} \label{lambda-diff-lem-cor}
Let $\tilde{a}, \tilde{b} \in \mathbb{R}^{n_1}$ be fixed. For any $\beta \in \mathbb{N}^{n_2}$, we have 
\be 
\frac{\partial^{\beta}}{\partial \lambda^{\beta}} \left\{\sum_{\mu} m_1((2\mu + \tilde{a})|\lambda|, - \lambda) m_2((2\mu + \tilde{b})|\lambda|, - \lambda) \, \Phi_{\mu}^{\lambda}(x^{\prime}) \Phi_{\mu}^{\lambda}(y^{\prime}) \right\} 
\ee
is a finite linear combination of terms of the form
\begin{align*}
& \Theta_l(\lambda) (x^\prime-y^\prime)^{\alpha_1-\alpha_3} {x^{\prime}}^{\alpha_{1}+\alpha_2+\alpha_3}\left(|\lambda| {y^{\prime}} \right)^{\alpha_{2}} |\lambda|^{\nu} \\ 
& \int_{[0,1]^{\nu_1} \times \Omega^{\nu_2}} \sum_{\mu} \left(\tau^{\theta_1} \partial_\tau^{\theta_2 - \theta_3} \partial_\kappa^{\beta_1} m_1\right) ((2\mu + \vec{a}(\omega) + \tilde{a})|\lambda|, - \lambda) \left( \partial_{\tau}^{\theta_3}\partial_{\kappa}^{\beta_2} m_2 \right) ((2\mu + \vec{a}(\omega) + \tilde{b})|\lambda|, - \lambda) \\ 
& \quad \quad \quad \quad \quad \quad \Phi_{\mu}^{\lambda}(x^{\prime}) \Phi_{\mu}^{\lambda}(y^{\prime}) \, g(\omega) \, d \omega, 
\end{align*} 
where $\nu_1, \nu_2, l \in \mathbb{N}$ and $\alpha_1, \alpha_2, \alpha_3, \beta_1, \beta_2, \theta_1, \theta_2, \theta_3 \in \mathbb{N}^{n_1}$ be such that $0 \leq l < |\beta|$, $\alpha_3 \leq \alpha_1$, $\theta_3 \leq \theta_2$, with $\nu_1, \nu_2 \leq |\beta| - ( |\beta_1| + |\beta_2| + l)$, $ \left( 2 |\theta_1| + |\alpha_2| + 2 \nu \right) \leq |\theta_2| \leq 2|\beta|- 2 ( |\beta_1| + |\beta_2| + l)$, $\left|2 \alpha_{1} + \alpha_{2} \right| \leq |\beta| - ( |\beta_1| + |\beta_2| + l)$,  $\left( |\theta_2| + |\alpha_1| \right) - \left( |\theta_1| + \nu \right) = |\beta| - ( |\beta_1| + |\beta_2| + l)$ and $\Theta_l$ is a continuous function on $\R^{n_2} \setminus \{0\}$ which is homogeneous of degree $-l$.
\end{cor}

Note that in order to get the above expression, we have also used the binomial expansion to replace ${y^{\prime}}^{\alpha_1}$ by linear combination of $(y^{\prime} - x^{\prime})^{\alpha_1 - \alpha_3} \, {x^{\prime}}^{\alpha_3}$ with $\alpha_3 \leq \alpha_1$. Such an expression is helpful in the proofs in the later sections. 


We now look at the action of $\left( x^{\prime} - y^{\prime} \right)^{\alpha}$ on kernels of the scaled Hermite multipliers. Using annihilation and creation operators $A_j(\lambda)$ and $A_j(\lambda)^*$ (as in \eqref{def-ann-cre-op}), let us define the non-commutative derivatives of any bounded operator $M$ on $L^2(\R^{n_1})$ in the following way:
\begin{align} \label{def:non-comm-derivative}
\delta_j(\lambda) M = |\lambda|^{-1/2} [M, A_j(\lambda)] \quad \textup{and} \quad \bar{\delta}_j(\lambda)M = |\lambda|^{-1/2} [A^*_j(\lambda), M].
\end{align} 
For $\alpha = \left( \alpha^{(1)}, \alpha^{(2)}, \ldots, \alpha^{(n_1)} \right), \theta = \left( \theta^{(1)}, \theta^{(2)}, \ldots, \theta^{(n_1)} \right) \in \N^{n_1}$, we write 
\begin{align*} 
\delta (\lambda)^\alpha = \delta_1 (\lambda)^{\alpha^{(1)}} \delta_2 (\lambda)^{\alpha^{(2)}} \cdots \delta_{n_1} (\lambda)^{\alpha^{(n_1)}} \quad \textup{and} \quad \bar{\delta} (\lambda)^\theta = \bar{\delta}_1 (\lambda)^{\theta^{(1)}} \bar{\delta}_2 (\lambda)^{\theta^{(2)}} \cdots \bar{\delta}_{n_1} (\lambda)^{\theta^{(n_1)}}. 
\end{align*}

\begin{lem} \label{first-layer-lem}
Let $\tilde{a} \in \mathbb{R}^{n_1}$ be fixed. For any $\alpha \in \mathbb{N}^{n_1}$, we have 
$$ \left(x^{\prime} - y^{\prime} \right)^{\alpha} \sum_{\mu} m((2\mu + \tilde{a})|\lambda|) \Phi_{\mu}^{\lambda}(x^{\prime}) \Phi_{\mu}^{\lambda}(y^{\prime})$$
is a finite linear combination of terms of the form 
\begin{equation} 
\int_{[0,1]^{|\alpha|}} \sum_{\mu} C_{\mu, \tilde{a}, \vec{c}} \left(\tau^{\frac{1}{2} \gamma_2} \partial_{\tau}^{ \gamma_1 } m\right) ((2\mu  + \tilde{a} + \vec{c}(s))|\lambda|) \Phi_{\mu + \tilde{\mu}}^{\lambda}(x^{\prime}) \Phi_{\mu}^{\lambda}(y^{\prime}) \, ds, 
\end{equation} 
where $|\gamma_2| \leq |\gamma_1| \leq |\alpha|$, $|\gamma_1| - \frac{1}{2} |\gamma_2| = \frac{|\alpha|}{2}$, $|\tilde{\mu}| \leq |\alpha|$, and $C_{\mu, \tilde{a}, \vec{c}}$ is a bounded function of $\mu$, $\tilde{a}$ and $\vec{c}$. 
\end{lem}
\begin{proof} 
It is easy to see that 
$$ (x_j^{\prime} - y_j^{\prime}) m(H(\lambda)) = \frac{1}{2|\lambda|^{1/2}} (\bar{\delta}_j(\lambda) - \delta_j(\lambda)) m(H(\lambda)), $$ 
where $\delta_j(\lambda)$ and $\bar{\delta}_j(\lambda)$ are given by \eqref{def:non-comm-derivative}. Thus, it is enough to show that the kernel of the operators $|\lambda|^{-|\alpha|/2}  \delta(\lambda)^{\gamma_1}\bar{\delta}^{\gamma_2}(\lambda)m(H(\lambda))$, with $|\gamma_1 + \gamma_2|= |\alpha|$, can be written as a finite linear combination of terms of our required form. Now the kernel of $|\lambda|^{-|\alpha|/2} \delta(\lambda)^{\gamma_1}\bar{\delta}^{\gamma_2}(\lambda)m(H(\lambda))$ can be written explicitly (see Lemma 2.3 of \cite{BagchiFourierMultipliersHeisenbergStudia} and Lemma 2.1 of \cite{MauceriWeylTransformJFA80}) as 
\begin{align*}
& |\lambda|^{-|\alpha|/2} \sum_{\gamma_3} C_{\gamma_1, \gamma_2, \gamma_3} |\lambda|^{-(|\gamma_2|+2|\gamma_3|-|\gamma_1|)/2} \left(\Delta_{-}^{\gamma_3} \Delta_{+}^{\gamma_2} m((2\mu + \tilde{a})|\lambda|)\right) \\ 
& \quad \quad \quad \quad \quad  \left(A^*(\lambda)^{\gamma_2+\gamma_3-\gamma_1}  A(\lambda)^{\gamma_3} \Phi_\mu^\lambda \right) (x^{\prime}) \, \Phi_{\mu}^\lambda (y^{\prime}),
\end{align*}
where $0 \leq |\gamma_3| \leq |\gamma_1| \leq |\gamma_2 + \gamma_3| \leq |\alpha|$ and $|\gamma_1 + \gamma_2| = |\alpha|$. Clearly, $|\gamma_1 - \gamma_3| \leq |\alpha|/2$. 

Now, each term from the above summation is of the form 
\begin{flalign*}
& |\lambda|^{-|\alpha|/2} |\lambda|^{-|\gamma_2 + 2 \gamma_3 - \gamma_1|/2}  \sum_{\mu} C_{\mu} ((2|\mu|+n_1)|\lambda|)^{|\gamma_2 + 2  \gamma_3 -  \gamma_1|/2} \left(\Delta_{-}^{\gamma_3} \Delta_{+}^{\gamma_2} m((2\mu + \tilde{a})|\lambda|)\right) \\ 
& \quad \quad \quad \quad \quad \quad \quad \quad \quad \quad \quad \Phi_{\mu + \tilde{\mu}}^{\lambda}(x^{\prime}) \Phi_{\mu}^{\lambda}(y^{\prime}),
\end{flalign*}
for a bounded function $C_\mu$ of $\mu$-variable, which can be rewritten (using the fundamental theorem of calculus) as 
\begin{align*}
& |\lambda|^{-|\alpha|/2}|\lambda|^{-|\gamma_2 + 2 \gamma_3 - \gamma_1|/2} \sum_{\mu} C_{\mu} ((2|\mu|+n_1)|\lambda|)^{|\gamma_2 + 2 \gamma_3 - \gamma_1|/2} \\ 
&\quad \quad \quad \quad \quad \quad \quad \quad (2|\lambda|)^{|\gamma_3 + \gamma_2|} \left(\int_{[0,1]^{|\gamma_3 + \gamma_2|}} \partial^{\gamma_3 + \gamma_2}_{\tau} m((2\mu  + \tilde{a} +  \vec{c}(s))|\lambda|) \, ds\right) \, \Phi_{\mu + \tilde{\mu}}^{\lambda}(x^{\prime}) \Phi_{\mu}^{\lambda}(y^{\prime}) \\
& = C_{\gamma_2, \gamma_3} \sum_{\mu} C_{\mu, \tilde{a}, \vec{c}} ((2|\mu|+n_1)|\lambda|)^{|\gamma_2 + 2  \gamma_3 - \gamma_1|/2} \int_{[0,1]^{|\gamma_3 + \gamma_2|}} \partial^{\gamma_3 + \gamma_2}_{\tau} m((2\mu  + \tilde{a} + + \vec{c}(s))|\lambda|) \, ds \\ 
&\quad \quad \quad \quad \Phi_{\mu + \tilde{\mu}}^{\lambda}(x^{\prime}) \Phi_{\mu}^{\lambda}(y^{\prime}), \\ 
& = C_{\gamma_2, \gamma_3} \int_{[0,1]^{|\gamma_3 + \gamma_2|}} \sum_{\mu} C_{\mu, \tilde{a}, \vec{c}} \left(\tau^{(\gamma_2 + 2 \gamma_3 - \gamma_1)/2} \partial^{\gamma_3 + \gamma_2}_{\tau} m\right) ((2\mu  + \tilde{a} + + \vec{c}(s))|\lambda|) \Phi_{\mu + \tilde{\mu}}^{\lambda}(x^{\prime}) \Phi_{\mu}^{\lambda}(y^{\prime}) \, ds \\
& = C_{\gamma_2, \gamma_3} \int_{[0,1]^{|\alpha|}} \sum_{\mu} C_{\mu, \tilde{a}, \vec{c}} \left(\tau^{\frac{1}{2}(\gamma_1 + \gamma_2) - (\gamma_1 - \gamma_3)} \partial_{\tau}^{(\gamma_1 + \gamma_2)-(\gamma_1 - \gamma_3)} m\right) ((2\mu  + \tilde{a} + + \vec{c}(s))|\lambda|) \\ 
&\quad \quad \quad \quad \quad \Phi_{\mu + \tilde{\mu}}^{\lambda}(x^{\prime}) \Phi_{\mu}^{\lambda}(y^{\prime}) \, ds,
\end{align*}
where $0 \leq |\gamma_1 - \gamma_3| \leq |\alpha|/2$ and $C_{\mu, \tilde{a}, \vec{c}}$ is a bounded function of bounded function of indexing variables. 

The claimed form of Lemma \ref{first-layer-lem} holds true by suitably relabeling $\gamma_1, \gamma_2, \gamma_3$. 
\end{proof}


\begin{rem} \label{first-second-layer-rem} 
Analogous to Corollary \ref{lambda-diff-lem-cor}, one could rewrite Lemma \ref{first-layer-lem} for product symbol functions of the form $m_1((2\mu + \tilde{a})|\lambda|) \, m_2((2\mu + \tilde{b})|\lambda|)$ or $m_1((2\mu + \tilde{a})|\lambda|, - \lambda) \, m_2((2\mu + \tilde{b})|\lambda|, - \lambda)$. 
\end{rem}


We combine the results obtained thus far to analyse the action of $\left|x^\prime - y^\prime \right|^{4 N_1} \left|x^{\prime \prime} - y^{\prime \prime}\right|^{2 N_2}$ on kernel functions. 

\begin{lem} \label{weighted-kernel-estimate-3} 
Let $\tilde{a} \in \mathbb{R}^{n_1}$ be fixed and $m$ is a compactly supported smooth function on $(\R_+)^{n_1}\times(\R^{n_2}\setminus \{0\})$. For any $N_1, N_2\in \mathbb{N}$, we can write 
$$ \left|x^\prime - y^\prime \right|^{4 N_1} \left|x^{\prime \prime} - y^{\prime \prime} \right|^{2 N_2} \int_{\mathbb{R}^{n_2}} \sum_{\mu} m((2\mu + \tilde{a})|\lambda|, - \lambda) E_{\mu, \lambda} (x, y) \, d \lambda $$ 
as a finite linear combination of terms of the form 
\begin{align} \label{y-prime-last---}
& {x^{\prime}}^{\alpha_{1}}  \int_{[0,1]^{\nu_1} \times \Omega^{\nu_2} \times [0,1]^{4 N_1}} \int_{\mathbb{R}^{n_2}} \Theta_l(\lambda) \left(|\lambda| {y^{\prime}} \right)^{\alpha_{2}} \\ 
\nonumber & \sum_{\mu} C_{\mu, \tilde{a}, \vec{a}} \left(\tau^{\frac{1}{2} \theta_1} \partial_\tau^{\theta_2} \partial_{\kappa}^{\beta} m\right) ((2\mu + \tilde{a} + \vec{a}(\omega)) |\lambda|, - \lambda) \Phi_{\mu + \tilde{\mu}}^{\lambda}(x^{\prime}) \Phi_{\mu}^{\lambda}(y^{\prime}) e^{- i \lambda \cdot (x^{\prime \prime} - y^{\prime \prime})} \, g(\omega) \, d \lambda \, d \omega, 
\end{align}
where $\nu_1, \nu_2 \leq 2 N_2 - (|\beta| + l)$, $|\tilde{\mu}| \leq 4 N_1$, $ |\theta_1| + |\alpha_2| \leq |\theta_2| \leq 4 N_1 + 4 N_2 - 2 (|\beta| + l)$, $\left| \alpha_{1}  \right| \leq 2 N_2 - (|\beta| + l)$, $|\theta_2| - \frac{1}{2} |\theta_1| + \frac{1}{2}(|\alpha_1|-|\alpha_2|) = 2 N_1 + 2 N_2 - (|\beta| + l)$ and $\Theta_l$ is a continuous function on $\R^{n_2} \setminus \{0\}$ which is homogeneous of degree $-l$. 
\end{lem}
\begin{proof} Integrating by parts, we get 
\begin{align*}
& \left|x^{\prime \prime} - y^{\prime \prime}\right|^{2 N_2} \int_{\mathbb{R}^{n_2}} \sum_{\mu} m((2\mu + \tilde{a})|\lambda|, - \lambda)  E_{\mu, \lambda} (x, y) \, d \lambda \\ 
& = (-1)^{N_2} \int_{\mathbb{R}^{n_2}} \left\{ \left(\sum_{j=1}^{n_2} \frac{\partial^2}{\partial \lambda_j^2} \right) \left(\sum_{\mu} m((2\mu + \tilde{a})|\lambda|, - \lambda) \Phi_{\mu}^{\lambda}(x^{\prime}) \Phi_{\mu}^{\lambda}(y^{\prime})\right) \right\} e^{- i \lambda \cdot (x^{\prime \prime} - y^{\prime \prime})} \, d \lambda. 
\end{align*} 
Clearly, the integrand in the parenthesis in the last expression is a finite linear combination of terms of the form 
$$ \frac{\partial^{\beta}}{\partial \lambda^{\beta}} \left\{\sum_{\mu} m((2\mu + \tilde{a})|\lambda|, - \lambda) \, \Phi_{\mu}^{\lambda}(x^{\prime}) \Phi_{\mu}^{\lambda}(y^{\prime}) \right\} $$
with $|\beta| = 2 N_2$. We can estimate each such term with the help of Corollary \ref{lambda-diff-lem-cor} (with $m_1 = m$ and without the second multiplier function $m_2$). 

Finally, using Lemma \ref{first-layer-lem}, one can write 
$$\left|x^\prime - y^\prime \right|^{4 N_1} \left|x^{\prime \prime} - y^{\prime \prime}\right|^{2 N_2} \int_{\mathbb{R}^{n_2}} \sum_{\mu} m((2\mu + \tilde{a})|\lambda|, - \lambda) E_{\mu, \lambda} (x, y) \, d \lambda$$ 
as a finite linear combination of terms of the form
\begin{align*} 
&  \left({x^{\prime}}^{\alpha_{1}} {y^{\prime}}^{\alpha_{1}} \right) {x^{\prime}}^{\alpha_{2}} \iiint\limits_{[0,1]^{\nu_1} \times \Omega^{\nu_2} \times [0,1]^{4 N_1}} \int_{\mathbb{R}^{n_2}} \Theta_l(\lambda) \left(|\lambda| {y^{\prime}} \right)^{\alpha_{2}} |\lambda|^{\nu} \sum_{\mu} C_{\mu, \tilde{a}, \vec{a}} \\ 
& \left( \tau^{\gamma_1 - \gamma_2} \partial_{\tau}^{2 \gamma_1 - \gamma_2} \left( \tau^{\theta_1} \partial_\tau^{\theta_2} \partial_\kappa^{\beta_1} m \right)\right) ((2\mu + \tilde{a} + \vec{a}(\omega))|\lambda|, - \lambda) \Phi_{\mu + \tilde{\mu}}^{\lambda}(x^{\prime}) \Phi_{\mu}^{\lambda}(y^{\prime}) e^{- i \lambda \cdot (x^{\prime \prime} - y^{\prime \prime})} \, g(\omega) \, d \lambda \, d\omega, 
\end{align*}
where $|\tilde{\mu}| \leq 4 N_1$ and $C_{\mu, \tilde{a}, \vec{a}}$ is a bounded function of indexing variables. 

Using Leibniz formula for derivatives with respect to $\tau$-variable, and absorbing $|\lambda|^{\nu}$ in the powers of $\tau^{\theta_1}$ (as already seen in Remark \ref{rem-indices-clarity}), one can  verify that the above expression can be rewritten as a finite linear combination of the terms which are of the form as claimed in \eqref{y-prime-last---}, and this completes the proof of Lemma \ref{weighted-kernel-estimate-3}. 
\end{proof}


\begin{rem} \label{rem2-indices-clarity} 
Following the discussion of Remark \ref{rem-indices-clarity}, one can verify that  \eqref{y-prime-last---} is of the form 
\begin{align*} 
& {x^{\prime}}^{\alpha_{1}}\int_{[0,1]^{\nu_1} \times \Omega^{\nu_2} \times [0,1]^{4 N_1}} \int_{\mathbb{R}^{n_2}} \Theta_l(\lambda) \sum_{\mu} C_{\mu, \tilde{a}, \vec{a}} \left(\tau^{\frac{1}{2} \theta_1} \partial_\tau^{\theta_2}\partial_{\kappa}^{\beta} m\right) ((2\mu + \tilde{a} + \vec{a}(\omega)) |\lambda|, - \lambda) \\ 
\nonumber & \quad \quad \quad \quad \quad \quad \quad \quad \quad \quad \quad \quad \quad \quad \quad \quad \quad \Phi_{\mu + \tilde{\mu}}^{\lambda}(x^{\prime}) \Phi_{\mu}^{\lambda}(y^{\prime}) e^{- i \lambda \cdot (x^{\prime \prime} - y^{\prime \prime})} \, g(\omega) \, d \lambda \, d\omega, 
\end{align*}
where $\nu_1, \nu_2 \leq 2 N_2 - (|\beta| + l)$, $| \alpha_1 | \leq 2 N_2 - (|\beta|+l)$, $ |\mathcal{\theta}_1| \leq |\mathcal{\theta}_2| \leq 4 N_1+ 4 N_2 - 2  (|\beta|+l)$, $|\mathcal{\theta}_2| - \frac{1}{2} |\mathcal{\theta}_1| = 2 N_1 + 2 N_2 -  (|\beta|+l)- \frac{|\alpha_1|}{2} \geq 2 N_1 + N_2 - \frac{1}{2} (|\beta|+l)$, $C_{\mu, \tilde{a}, \vec{a}}$ is a bounded function of $\mu$ and $\vec{a}$, and $\Theta_l$ is a continuous function on $\R^{n_2} \setminus \{0\}$ which is homogeneous of degree $-l$.
\end{rem} 

\begin{rem} \label{weighted-kernel-estimate-cor} 
Analogous to Corollary \ref{lambda-diff-lem-cor}, one could rewrite Lemma \ref{weighted-kernel-estimate-3} (and Remark \ref{rem2-indices-clarity}) for product symbol functions of the form $m_1((2\mu + \tilde{a})|\lambda|, - \lambda) \, m_2((2\mu + \tilde{b})|\lambda|, - \lambda)$. 
\end{rem}


\section{An application to weighted Plancherel estimates} \label{Sec:weighted-Planch}
In this section, we shall prove Theorem \ref{thm:weighted-Plancherel-L2}. For that, recall first from the asymptotic behaviour of the distance function $\tilde{d}$ that in order to study the action of $\tilde{d}(x,y)^{4N}$ on the integral kernels, we need to look at the action of $|x^{\prime \prime} - y^{\prime \prime}|^{2N}$ and $|x^{\prime \prime} - y^{\prime \prime}|^{4N} / (|x^{\prime}| + |y^{\prime}|)^{4N}$ in two different regimes. Now, we have seen that $|x^{\prime \prime} - y^{\prime \prime}|^{2N}$ leads to $2N$ number of derivatives with respect to the $\lambda$-variable and this further amounts to a maximum of $4N$ number of derivatives of the symbol $m$ in $\tau$-variable. So, in this case, the number of $\tau$-derivatives on $m$ are at most equal to the power of the distance function. But then in the other regime, if we apply the same $\lambda$-derivative analysis for $|x^{\prime \prime} - y^{\prime \prime}|^{4N} / (|x^{\prime} + y^{\prime}|^{4N})$, then the number of $\tau$-derivatives on $m$ could become double of the power of the distance function. Since the main purpose of Theorem \ref{thm:weighted-Plancherel-L2} is to have sharpness in terms of the number of derivatives on the symbol function, we must revisit the action of the powers of $|x^{\prime \prime} - y^{\prime \prime}|$ on the integral kernels, aiming to control the number of $\tau$-derivatives in the above mentioned second regime.

\begin{prop} \label{prop:lambda-diff-lem-half-derivative-full-dim}
Let $\tilde{a} \in \mathbb{R}^{n_1}$ be fixed and $m$ a compactly supported smooth function on $(\R_+)^{n_1}\times(\R^{n_2}\setminus \{0\})$. For any $N \in \mathbb{N}$, we can write 
$$ \left| x^{\prime \prime} - y^{\prime \prime} \right|^{2N} \int_{\mathbb{R}^{n_2}} \sum_{\mu} m((2\mu + \tilde{a} + \tilde{1})|\lambda|, - \lambda) E_{\mu, \lambda} (x, y) \, d \lambda $$ 
as a finite linear combination of terms of the form 
\begin{align} \label{lambda-diff-eq-3} 
& \left( x^{\prime \prime} - y^{\prime \prime} \right)^{\beta_1} \left({x^{\prime}}^{\alpha_{1}} {y^{\prime}}^{\alpha_{1}} \right) {x^{\prime}}^{\alpha_{2}} \int_{[0,1]^{\nu_1} \times \Omega^{\nu_2}} \int_{\mathbb{R}^{n_2}} \Theta_l (\lambda) \left( |\lambda| {y^{\prime}} \right)^{\alpha_{2}} |\lambda|^{\nu} \\ 
\nonumber & \quad \quad \quad \quad \sum_{\mu} \left(\tau^{\theta_1} \partial_\tau^{\theta_2} \partial_{\kappa}^{\beta_2} m \right) ((2\mu + \tilde{a} + \vec{a}(\omega)) |\lambda|, - \lambda) E_{\mu, \lambda} (x, y) \, d \lambda \, d \omega, 
\end{align} 
where $|\beta_1| \leq N$, $l < 2N - |\beta_1|$, $ 2 |\theta_1| + |\alpha_2| + 2 \nu \leq |\theta_2| \leq 2N$, $\left|2 \alpha_{1} + \alpha_{2} \right| \leq 2N - 2 |\beta_1|$, and $\left( |\beta_1| + |\beta_2| + |\theta_2| + |\alpha_1| + l \right) - \left( |\theta_1| + \nu \right) = 2N$ and $\Theta_l$ is a continuous function on $\R^{n_2} \setminus \{0\}$ which is homogeneous of degree $-l$. 
\end{prop} 

\begin{rem} \label{rem:homogeneity-dim-1}
When $n_2 = 1$, the homogeneity $l$ of $\Theta_l$ in the above proposition is always $0$. 
\end{rem}

\begin{proof}[Proof of Proposition \ref{prop:lambda-diff-lem-half-derivative-full-dim}] We build the proof of this lemma by carefully analysing the prototype terms from the proof of Proposition \ref{prop:lambda-diff-theorem}, using which Proposition \ref{prop:lambda-diff-lem-half-derivative-full-dim} can be proved by mathematical induction. First, we show that it is true for $N=1$, which also gives us an insight as to why an expression of the type \eqref{lambda-diff-eq-3} should hold true. 

Note that an application of the integration by parts give 
\begin{align}
& \left( x_k^{\prime \prime} - y_k^{\prime \prime} \right) \int_{\mathbb{R}^{n_2}} \sum_{\mu} m((2\mu + \tilde{a})|\lambda|, - \lambda) E_{\mu, \lambda} (x, y) \, d \lambda \\ 
\nonumber & = i \int_{\mathbb{R}^{n_2}} \left( \frac{\partial}{\partial \lambda_k} \sum_{\mu} m((2\mu + \tilde{a})|\lambda|, - \lambda) \Phi_{\mu}^{\lambda}(x^{\prime}) \Phi_{\mu}^{\lambda}(y^{\prime})\right) e^{- i \lambda \cdot (x^{\prime \prime} - y^{\prime \prime})} \, d \lambda \\
\nonumber & = i \int_{\mathbb{R}^{n_2}} \sum_{\mu} \partial_{\kappa}^{e_k} m ((2\mu + \tilde{a})|\lambda|, - \lambda) \Phi_{\mu}^{\lambda}(x^{\prime}) \Phi_{\mu}^{\lambda}(y^{\prime}) e^{- i \lambda \cdot (x^{\prime \prime} - y^{\prime \prime})} \, d \lambda \\ 
\nonumber & \quad + i \int_{\mathbb{R}^{n_2}} \frac{\lambda_k}{|\lambda|} \left. \left( \frac{\partial}{\partial |\lambda|} \sum_{\mu} m((2\mu + \tilde{a})|\lambda|, \kappa) \Phi_{\mu}^{\lambda}(x^{\prime}) \Phi_{\mu}^{\lambda}(y^{\prime}) \right) \right|_{\kappa = - \lambda} e^{- i \lambda \cdot (x^{\prime \prime} - y^{\prime \prime})} \, d \lambda. 
\end{align}

Clearly, the first term is of the claimed form. As for the second term, recall that while applying $\frac{\partial}{\partial |\lambda|}$ on the kernel, we get four prototype terms, namely, $I_{0, N+1}$ from \eqref{induction-est-2}, the two terms from \eqref{I11I234}, and $I^1_{3,4,+} + I^1_{2,5,+}$ from \eqref{I134I2I5}. We would like to point out that $I_{0, N+1}$ comes simply because of the presence of $\tilde{a}$ in the symbol function. For convenience, we write these terms below with new name tags: 
\begin{align} \label{prop:sharp-dim-der-joint-calc}
\mathcal{J}_0^{\lambda} (x^{\prime}, y^{\prime}) & : = \tilde{a}_j \sum_{\mu} \left(\partial_{\tau_j} m \right) ((2\mu + \tilde{a} + \tilde{1}) |\lambda| , - \lambda) \Phi_{\mu}^{\lambda}(x^{\prime}) \Phi_{\mu}^{\lambda}(y^{\prime}) \\ 
\nonumber \mathcal{J}_1^{\lambda} (x^{\prime}, y^{\prime}) & : = \int_{-1}^0 \int_{s}^0 \sum_{\mu}  \left(\tau_j \partial^2_{\tau_j} m\right) ((2\mu + \vec{a}(s, s^{\prime}) + \tilde{1}) |\lambda| , - \lambda) \Phi_{\mu}^{\lambda}(x^{\prime}) \Phi_{\mu}^{\lambda}(y^{\prime}) \, g(s, s^\prime) \, ds^\prime \, ds \\ 
\nonumber \mathcal{J}_2^{\lambda} (x^{\prime}, y^{\prime}) & : = \int_{-1}^0 \int_{s}^0 \sum_{\mu} |\lambda| \left(\partial^2_{\tau_j} m\right) ((2\mu + \vec{b}(s, s^{\prime}) + \tilde{1}) |\lambda| , - \lambda) \, \Phi_{\mu}^{\lambda}(x^{\prime}) \Phi_{\mu}^{\lambda}(y^{\prime}) \, g(s, s^\prime) \, ds^\prime \, ds \\ 
\nonumber \mathcal{J}_3^{\lambda} (x^{\prime}, y^{\prime}) & : = (x_j^{\prime} y_j^{\prime}) |\lambda| \int_{0}^1 \sum_\mu \left( \partial_{\tau_j} m \right) ((2\mu + \vec{c}(s) + \tilde{1}) |\lambda| , - \lambda)  \Phi_{\mu}^{\lambda}(x^{\prime}) \Phi_{\mu}^{\lambda}(y^{\prime}) \, g(s) \, ds. 
\end{align}

Now, note that we already have two $\tau$-derivatives on the symbol function $m$ in $\mathcal{J}_1^{\lambda} (x^{\prime}, y^{\prime})$ and $\mathcal{J}_2^{\lambda} (x^{\prime}, y^{\prime})$. So, we do note perform the next $|\lambda|$-derivative on these terms. We do not take $|\lambda|$-derivative of $\mathcal{J}_0^{\lambda} (x^{\prime}, y^{\prime})$ too. That is, we keep the three terms $\left( x_k^{\prime \prime} - y_k^{\prime \prime} \right) \mathcal{J}_0^{\lambda} (x^{\prime}, y^{\prime})$, $\left( x_k^{\prime \prime} - y_k^{\prime \prime} \right) \mathcal{J}_1^{\lambda} (x^{\prime}, y^{\prime})$, and $\left( x_k^{\prime \prime} - y_k^{\prime \prime} \right) \mathcal{J}_2^{\lambda} (x^{\prime}, y^{\prime})$ as it is. Similarly, we do keep the term in  \eqref{prop:sharp-dim-der-joint-calc} involving $\partial_{\kappa}^{e_k} m$ also as it is. In this way, these four are of the claimed form with 
\begin{itemize}
\item $\left( x_k^{\prime \prime} - y_k^{\prime \prime} \right) \partial_{\kappa}^{e_k} m$ : $|\beta_1| = 1, |\beta_2| = 1$;

\item $\left( x_k^{\prime \prime} - y_k^{\prime \prime} \right) \mathcal{J}_0^{\lambda} (x^{\prime}, y^{\prime})$ : $|\beta_1| = 1, |\theta_2| = 1$;

\item $\left( x_k^{\prime \prime} - y_k^{\prime \prime} \right) \mathcal{J}_1^{\lambda} (x^{\prime}, y^{\prime})$ : $|\beta_1| = 1, |\theta_2| = 2, |\theta_1| = 1$;

\item $\left( x_k^{\prime \prime} - y_k^{\prime \prime} \right) \mathcal{J}_2^{\lambda} (x^{\prime}, y^{\prime})$ : $|\beta_1| = 1, |\theta_2| = 2, \nu = 1$.
\end{itemize}

For computational purposes, we introduce the notation 
$$ \mathcal{J}_3^{\lambda, \kappa} (x^{\prime}, y^{\prime}) := (x_j^{\prime} y_j^{\prime}) |\lambda| \int_{0}^1 \sum_\mu \left( \partial_{\tau_j} m \right) ((2\mu + \vec{c}(s) + \tilde{1}) |\lambda| , \kappa)  \Phi_{\mu}^{\lambda}(x^{\prime}) \Phi_{\mu}^{\lambda}(y^{\prime}) \, g(s) \, ds, $$
so that $\mathcal{J}_3^{\lambda} = \mathcal{J}_3^{\lambda, - \lambda} $. 

Now, 
\begin{align*}
& \left( x_k^{\prime \prime} - y_k^{\prime \prime} \right) \int_{\mathbb{R}^{n_2}} \frac{\lambda_k}{|\lambda|} \mathcal{J}_3^{\lambda, - \lambda} (x^{\prime}, y^{\prime}) e^{- i \lambda \cdot (x^{\prime \prime} - y^{\prime \prime})} \, d\lambda \\ 
& = i \int_{\mathbb{R}^{n_2}} \frac{\partial}{\partial \lambda_k} \left( \frac{\lambda_k}{|\lambda|} \mathcal{J}_3^{\lambda, - \lambda} (x^{\prime}, y^{\prime}) \right) e^{- i \lambda \cdot (x^{\prime \prime} - y^{\prime \prime})} \, d \lambda \\ 
& = i \int_{\mathbb{R}^{n_2}} \left. \left\{ \frac{\lambda_k}{|\lambda|} \partial_{ \kappa}^{e_k} \mathcal{J}_3^{\lambda, \kappa} (x^{\prime}, y^{\prime}) + \left( \frac{1}{|\lambda|} - \frac{\lambda_k^2}{|\lambda|^3} \right) \mathcal{J}_3^{\lambda, \kappa} (x^{\prime}, y^{\prime}) + \frac{\lambda_k^2}{|\lambda|^2} \frac{\partial}{\partial |\lambda|} \mathcal{J}_3^{\lambda, \kappa} (x^{\prime}, y^{\prime}) \right\} \right|_{\kappa = - \lambda} \\ 
& \quad \quad \quad \quad e^{- i \lambda \cdot (x^{\prime \prime} - y^{\prime \prime})} \, d \lambda. 
\end{align*}
So, we are left with analysing $\frac{\partial}{\partial |\lambda|} \mathcal{J}_3^{\lambda, \kappa} (x^{\prime}, y^{\prime})$ as the other terms in the above expression are of the claimed form. Now,

\begin{align*}
& \frac{\partial}{\partial |\lambda|} \mathcal{J}_3^{\lambda, \kappa} (x^{\prime}, y^{\prime}) \\ 
& = (x_j^{\prime} y_j^{\prime}) \int_{0}^1 \sum_\mu \left( \partial_{\tau_j} m \right) ((2\mu + \vec{c}(s) + \tilde{1}) |\lambda|, \kappa)  \Phi_{\mu}^{\lambda}(x^{\prime}) \Phi_{\mu}^{\lambda}(y^{\prime}) \, g(s) \, ds \\ 
& \quad + (x_j^{\prime} y_j^{\prime}) |\lambda| \int_{0}^1 \frac{\partial}{\partial |\lambda|} \left\{ \sum_\mu \left( \partial_{\tau_j} m \right) ((2\mu + \vec{c}(s) + \tilde{1}) |\lambda|, \kappa)  \Phi_{\mu}^{\lambda}(x^{\prime}) \Phi_{\mu}^{\lambda}(y^{\prime}) \right\} \, g(s) \, ds.
\end{align*}

Clearly, the first of the two terms in the sum on the right hand side of the above is of the claimed form with $|\alpha_1| =1$ and $|\theta_2| = 1$. For the second term, following the calculations of the proof of Proposition \ref{prop:lambda-diff-theorem}, it is enough to consider the following prototypes terms: 
$(x_j^{\prime} y_j^{\prime}) |\lambda| I_1$, $(x_j^{\prime} y_j^{\prime}) |\lambda| I^2_{3,4,+}$, and $(x_j^{\prime} y_j^{\prime}) |\lambda| (I^1_{3,4,+} + I^1_{2,5,+})$ with $m$ replaced by $\partial_{\tau_j} m$. The outcome of these is as follows: 

\begin{itemize}
\item $(x_j^{\prime} y_j^{\prime}) |\lambda| I_1$ : $|\alpha_1| = 1, |\theta_1| = 1, |\theta_2| = 2$; 

\item $(x_j^{\prime} y_j^{\prime}) |\lambda| I^2_{3,4,+}$ : $|\alpha_1| = 1, |\theta_1| = 1, |\theta_2| = 2$;

\item $(x_j^{\prime} y_j^{\prime}) |\lambda| (I^1_{3,4,+} + I^1_{2,5,+})$ : $|\alpha_2| = 2, |\theta_2| = 2$. 
\end{itemize} 
This completes the proof of the claim of the lemma for $N=1$. 

It can now be easily verified that the Lemma could be proved for any $N \geq 1$, using mathematical induction with repeating the idea of the proof of $N=1$. We leave the details. 
\end{proof}


We are now in a position to prove Theorem \ref{thm:weighted-Plancherel-L2}. As earlier, let us choose and fix $\psi_1 \in C_c^\infty((1/2,2))$ such that $0 \leq \psi_1 \leq 1$, and $\sum_{j=-\infty}^\infty \psi_1(2^j \eta) = 1$ for all $\eta \geq 0$. Let us decompose the symbol $m$ in pieces as follows. For $S \in \mathbb{N}$ we define 
\begin{align*}
m^S(\tau, \kappa) := m(\tau, \kappa) \sum_{j = - \infty}^S \psi_1 (2^j |\kappa|).
\end{align*}
Then, $m = \lim_{S \to \infty} m^S$. 

\begin{proof}[Proof of Theorem \ref{thm:weighted-Plancherel-L2}] 
We shall be done if we could establish the theorem for each $m^S$, with the bound being uniform in $S \in \mathbb N$. The sole purpose of introducing these cut-offs is to ensure the convergence of various integrals that we shall estimate. Furthermore, in view of the homogeneity of the distance function $d$ and the Grushin operator, it suffices to establish the theorem for $R=1$. So, let us assume that $\supp(m) \subseteq [-1, 1]^{n_1 + n_2}$. 

\medskip \noindent \underline{\bf Step 1:} When $p=2$. 

As mentioned earlier, the theorem is true in the case of $r=0$ from Theorem 6.1 (i) of \cite{MartiniJointFunctionalCalculiMathZ}. We shall prove it for any $0 < r \leq 4 \lfloor{ \frac{r_0}{4} \rfloor}$ with the help of the $r=0$ case. It suffices to consider the case of $r = 4N$ where $ N = \lfloor{ \frac{r_0}{4} \rfloor}$, as the result would follow in the intermediate range from interpolation. Recall that $d(x,y) \sim \min \{\varrho_1(x,y), \varrho_2(x,y) \}$, where $\varrho_1(x,y)$ and $\varrho_2(x,y)$ are given by \eqref{def:distance-2}. So, with $x \in \mathbb{R}^{n_1 + n_2}$ fixed, if we write 
\begin{align*}
\mathcal{A} & = \{ y \in \mathbb{R}^{n_1 + n_2} : |x^\prime| + |y^\prime| \leq 1 \}, \quad \mathcal{B} = \{ y \in \mathbb{R}^{n_1 + n_2} : |x^\prime| + |y^\prime| > 1 \} \\ 
\quad \textup{and} \quad \mathcal{C} & = \{ y \in \mathbb{R}^{n_1 + n_2} : \left|x^{\prime \prime} - y^{\prime \prime}\right|^{1/2} \leq \left|x^{\prime} \right| + \left|y^\prime \right| \}, 
\end{align*}
then the theorem will be established if we could show that under the given assumption 
\begin{align} \label{ineq1:weighted-Plancherel-L2} 
& \left( \int_{\mathcal{A}} \left| \varrho_1(x,y)^{4N} k_{m^S( \boldsymbol{L}, \boldsymbol{U} )} (x, y) \right|^2 \, dy \right)^{1/2} \lesssim_{4N} \left| B(x, 1) \right|^{-1/2} \left\| m \right\|_{W^{4N}_{\infty}}, 
\end{align}
and 
\begin{align} \label{ineq2:weighted-Plancherel-L2} 
\left( \int_{\mathcal{B}} \left| \mathbbm{1}_{\mathcal{C}}(y) \varrho_2(x,y)^{4N} k_{m^S( \boldsymbol{L}, \boldsymbol{U} )} (x, y) \right|^2 \, dy \right)^{1/2} \lesssim_{N} \left| B(x, 1) \right|^{-1/2} \left\| m  \right\|_{W^{4N}_{\infty}}. 
\end{align}

Note that 
$$ \varrho_1(x,y)^{4N} \lesssim_N \left|x^\prime - y^\prime \right|^{4N} + \left|x^{\prime \prime} - y^{\prime \prime} \right|^{2N}\quad \text{and} \quad \varrho_2(x,y)^{4N} \lesssim_N \left|x^\prime - y^\prime \right|^{4N} + \frac{\left|x^{\prime \prime} - y^{\prime \prime} \right|^{4N}}{\left|x^\prime \right| + \left|y^\prime \right|^{4N}}.$$ 

To establish \eqref{ineq1:weighted-Plancherel-L2}, we make use of Remark \ref{rem2-indices-clarity} with either $(N_1, N_2) = (N,0)$ or $(N_1, N_2) = (0,N)$, together with the condition of the set $\mathcal{A}$ and the  orthogonality of scaled Hermite functions to get that $\left( \int_{\mathcal{A}} \left| \varrho_1(x,y)^{4N} k_{m( \boldsymbol{L}, \boldsymbol{U} )} (x, y) \right|^2 \, dy \right)^{1/2}$ is dominated by a finite linear combination of terms of the form 
$$ \int_{\Lambda} \int_{\mathbb{R}^{n_2}} \left| \Theta_l(\lambda) \right|^2 \sum_{\mu} \left| \left(\tau^{\frac{1}{2} \theta_1} \partial_\tau^{\theta_2}\partial_{\kappa}^{\beta} m^S\right) ((2\mu + \tilde{a} + \vec{a}(\omega)) |\lambda|, - \lambda) \right|^2 \left| \Phi_{\mu + \tilde{\mu}}^{\lambda}(x^{\prime}) \right|^2 \, d \lambda \, d\omega, $$ 
where the conditions on parameters and the compact set $\Lambda = [0,1]^{\nu_1} \times \Omega^{\nu_2} \times [0,1]^{4 N_1}$ are same as in Remark \ref{rem2-indices-clarity}. 

From now on wards we drop the integration over the compact set $\Lambda$ as it will become clear from the arguments below that the estimates that we obtain are uniform in $\omega$ over $\Lambda$. 

In view of Leibniz formula in $\kappa$-derivatives, for all $\beta_1 \leq \beta$ it suffices to estimate 
\begin{align} \label{ineq:appli-reduc-1}
\int_{\mathbb{R}^{n_2}} \left| \Theta_l(\lambda) \right|^2 \left| \zeta^{\left( \beta-\beta_1\right)}_S (\lambda) \right|^2 \sum_{\mu} \left| \left(\tau^{\frac{1}{2} \theta_1} \partial_\tau^{\theta_2} \partial_{\kappa}^{\beta_1} m\right) ((2\mu + \tilde{a} + \vec{a}(\omega)) |\lambda|, - \lambda) \right|^2 \left| \Phi_{\mu + \tilde{\mu}}^{\lambda}(x^{\prime}) \right|^2 \, d \lambda. 
\end{align}

We make use of Taylor's expansion of $m$ in $\kappa$-variable in order to control the growth of $\zeta^{\beta-\beta_1}_S (\lambda)$ and $\Theta_l (\lambda)$ for $\lambda$ near zero. For this, note that using the cancellation condition $\lim_{\kappa \to 0} \partial_\kappa^{\beta} m(\tau, \kappa) = 0$ for all $|\beta| \leq r_0$, we get as an application of Taylor's theorem that 
$$ \partial_\tau^{\theta_2} \partial_{\kappa}^{\beta_1} m (\tau,\kappa) = \sum_{\tilde{\beta} \geq \beta_1 : |\tilde{\beta}| = |\beta| + l} \frac{ \kappa^{\tilde{\beta} - \beta_1} }{\left( \tilde{\beta} - \beta_1 \right)!} \int_0^1 \partial_\tau^{\theta_2} \partial_{\kappa}^{\tilde{\beta}} m (\tau, t\kappa) \, dt,$$ 
which implies that 
$$ \left| \partial_\tau^{\theta_2} \partial_{\kappa}^{\beta_1} m (\tau,\kappa) \right|^2 \lesssim_{r_0} \left( |\kappa|^{|\beta| + l - |\beta_1|} \right)^2 \sum_{\tilde{\beta} \geq \beta_1 : |\tilde{\beta}| = |\beta| + l} \int_0^1 \left| \partial_\tau^{\theta_2} \partial_{\kappa}^{\tilde{\beta}} m (\tau, t\kappa) \right|^2 \, dt.$$ 

The above estimate along with the fact that 
$$\sup_{S \in \mathbb{N}} \sup_{\lambda \in \mathbb{R}^{n_2}} |\lambda|^{|\beta| + l - |\beta_1|} \left| \Theta_l(\lambda) \right| \left| \zeta^{\left( \beta-\beta_1\right)}_S (\lambda) \right| \lesssim C_{r_0} < \infty $$
implies that \eqref{ineq:appli-reduc-1} is dominated (uniformly in $S \in \mathbb{N}$) by a finite sum of terms of the form 
\begin{align} \label{ineq:appli-reduc-2}
\nonumber & \int_0^1 \int_{\mathbb{R}^{n_2}} \sum_{\mu} \left| \left(\tau^{\frac{1}{2} \theta_1} \partial_\tau^{\theta_2}\partial_{\kappa}^{\tilde{\beta}} m\right) ((2\mu + \tilde{a} + \vec{a}(\omega)) |\lambda|, - t \lambda) \right|^2 \left| \Phi_{\mu + \tilde{\mu}}^{\lambda}(x^{\prime}) \right|^2 \, d \lambda \, dt \\ 
\nonumber & \lesssim \int_{\mathbb{R}^{n_2}} \sum_{\mu} \sup_{t \in [0,1]} \left| \left(\tau^{\frac{1}{2} \theta_1} \partial_\tau^{\theta_2}\partial_{\kappa}^{\tilde{\beta}} m\right) ((2\mu + \tilde{a} + \vec{a}(\omega)) |\lambda|, - t \lambda) \right|^2 \left| \Phi_{\mu + \tilde{\mu}}^{\lambda}(x^{\prime}) \right|^2 \, d \lambda \\ 
& \lesssim \int_{\mathbb{R}^{n_2}} \sum_{\mu} \sup_{\kappa \in \mathbb{R}^{n_2}} \left| \left(\tau^{\frac{1}{2} \theta_1} \partial_\tau^{\theta_2}\partial_{\kappa}^{\tilde{\beta}} m\right) ((2\mu + \tilde{a} + \vec{a}(\omega)) |\lambda|, \kappa) \right|^2 \left| \Phi_{\mu + \tilde{\mu}}^{\lambda}(x^{\prime}) \right|^2 \, d \lambda. 
\end{align}

We can now invoke the methodology developed in Lemmas 10 and 11 in \cite{MartiniMullerGrushinRevistaMath}. Note that in Lemma 10 of \cite{MartiniMullerGrushinRevistaMath} the authors make pointwise estimations of Hermite functions with shifts involved, and making use of it, they established Proposition 12. While doing so, in the proof of Lemma 11 of \cite{MartiniMullerGrushinRevistaMath}, in the final estimates at the end of Page 1276, they deal with shifts in the estimates with the help of their Lemma 10. These calculations are immediately applicable in our case too, and in fact, in much more simplicity, and with it we can conclude that \eqref{ineq:appli-reduc-2} can be dominated by 
\begin{align*}
\left| B(x, 1) \right|^{-1} \left\| \sup_{\kappa} \left| \tau^{\frac{1}{2} \theta_1} \partial_\tau^{\theta_2} \partial_\kappa^{\beta} m (\tau, \kappa) \right| \right\|^2_{L^\infty (d \tau)} & \lesssim_{N} \left| B(x, 1) \right|^{-1} \left\| m \right\|^2_{_{W^{4N}_{\infty}}}. 
\end{align*} 

Next, we can prove \eqref{ineq2:weighted-Plancherel-L2} repeating the calculations as above, with the only difference being that this time we make use of the estimates from Lemma \ref{first-layer-lem} and Proposition \ref{prop:lambda-diff-lem-half-derivative-full-dim} in place of the estimate of Remark \ref{rem2-indices-clarity}. While doing so, we utilize the fact that the conditions of the sets $\mathcal{B}$ and $\mathcal{C}$ imply that 
\begin{align*} 
\frac{|x^{\prime \prime} - y^{\prime \prime}|^L \left| \left({x^{\prime}}^{\alpha_{1}} {y^{\prime}}^{\alpha_{1}} \right) {x^{\prime}}^{\alpha_{2}} \right|}{\left( \left|x^\prime \right| + \left|y^\prime \right| \right)^{4N}} 
& = \frac{|x^{\prime \prime} - y^{\prime \prime}|^L}{\left( \left|x^\prime \right| + \left|y^\prime \right| \right)^{2L}} 
\frac{ \left| \left({x^{\prime}}^{\alpha_{1}} {y^{\prime}}^{\alpha_{1}} \right) {x^{\prime}}^{\alpha_{2}} \right|}{\left( \left|x^\prime \right| + \left|y^\prime \right| \right)^{4N-2L}} \\ 
& \leq \frac{ \left| \left({x^{\prime}}^{\alpha_{1}} {y^{\prime}}^{\alpha_{1}} \right) {x^{\prime}}^{\alpha_{2}} \right|}{\left( \left|x^\prime \right| + \left|y^\prime \right| \right)^{4N-2L}} \leq \left( \left|x^\prime \right| + \left|y^\prime \right| \right)^{-4N + 2L + (|2 \alpha_1 + \alpha_5|)} \leq 1,
\end{align*}
because $|2\alpha_1+\alpha_2|\leq 4N-2L$ from Proposition \ref{prop:lambda-diff-lem-half-derivative-full-dim}. Rest of the proof then follows similar to that of \eqref{ineq1:weighted-Plancherel-L2}. 

\medskip \noindent \underline{\bf Step 2:} When $p=\infty$. 

We can essentially repeat the proof of Step 1, with the help of Cauchy-Schwartz inequality on the pointwise estimates. For the same note that 
\begin{align*} 
& \left| \int_{\Lambda} \int_{\mathbb{R}^{n_2}} \Theta_l(\lambda) \sum_{\mu} \left(\tau^{\frac{1}{2} \theta_1} \partial_\tau^{\theta_2}\partial_{\kappa}^{\beta} m^S\right) ((2\mu + \tilde{a} + \vec{a}(\omega)) |\lambda|, - \lambda) \Phi_{\mu + \tilde{\mu}}^{\lambda}(x^{\prime}) \Phi_{\mu}^{\lambda}(y^{\prime}) e^{- i \lambda \cdot (x^{\prime \prime} - y^{\prime \prime})} \, d \lambda \, d\omega \right| \\ 
& \lesssim \left( \int_{\Lambda} \int_{\mathbb{R}^{n_2}} \left| \Theta_l(\lambda) \right| \sum_{\mu} \left| \left(\tau^{\frac{1}{2} \theta_1} \partial_\tau^{\theta_2}\partial_{\kappa}^{\beta} m^S\right) ((2\mu + \tilde{a} + \vec{a}(\omega)) |\lambda|, - \lambda) \right| \left| \Phi_{\mu + \tilde{\mu}}^{\lambda}(x^{\prime}) \right|^2  \, d \lambda \, d\omega \right)^{1/2} \\ 
& \quad \times \left( \int_{\Lambda} \int_{\mathbb{R}^{n_2}} \left| \Theta_l(\lambda) \right| \sum_{\mu} \left| \left(\tau^{\frac{1}{2} \theta_1} \partial_\tau^{\theta_2}\partial_{\kappa}^{\beta} m^S\right) ((2\mu + \tilde{a} + \vec{a}(\omega)) |\lambda|, - \lambda) \right| \left| \Phi_{\mu}^{\lambda}(y^{\prime}) \right|^2  \, d \lambda \, d\omega \right)^{1/2}. 
\end{align*} 

\medskip \noindent \underline{\bf Step 3:} When $2 < p < \infty$. 

The claimed estimate in this case follows from the interpolation of the estimates for $p=2$ and $p = \infty$.

This completes the proof of Theorem \ref{thm:weighted-Plancherel-L2}. 
\end{proof} 


\section{Proof of Theorem \ref{thm:old-joint-calc} (2) in the case of \texorpdfstring{$S^0_{0, 0}(\boldsymbol{L}, \boldsymbol{U})$}.} \label{Sec-S00-proof}
In this section, we start the proof of Theorem \ref{thm:old-joint-calc} (2) in the case of $S^0_{0, 0}(\boldsymbol{L}, \boldsymbol{U})$ using Cotlar-Stein Lemma (see, for example, Section 2 of Chapter VII in \cite{SteinHarmonicBook93}). Given a Hilbert space $\mathcal{H}$, let $\mathcal{B}(\mathcal{H})$ stands for the Banach space of all bounded linear operators on $\mathcal{H}$ with the operator norm denoted by $\| \cdot \|_{\mathcal{B}(\mathcal{H})}$. 

\begin{thm}[Cotlar-Stein lemma] \label{Cotlar-Stein-lemma}
Let $\mathcal{H}$ be a Hilbert space. Let $\{T_j\}_{j \in \mathbb{Z}} \subset \mathcal{B}(\mathcal{H})$. If there exists a sequence of positive constants $\{c(j)\}_{j \in \mathbb{Z}}$ with 
$$A = \sum_{j \in \mathbb{Z}} c(j) < \infty,$$ 
such that the operators $T_j$ satisfy the estimate 
$$ \max \left\{ \left\| T^*_j T_k \right\|_{\mathcal{B}(\mathcal{H})}, \left\| T_j T^*_k \right\|_{\mathcal{B}(\mathcal{H})} \right\} \leq \left( c(j - k) \right)^2, \textup{ for all } j, k \in \mathbb{Z},$$
then $\sum_{j \in \mathbb{Z}} T_j$ converges in the strong operator topology to some $T \in \mathcal{B}(\mathcal{H})$ and $\| T \|_{\mathcal{B}(\mathcal{H})} \leq A.$
\end{thm} 

Choose and fix $\psi_0 \in C_c^\infty((-2,2))$ and $\psi_1 \in C_c^\infty((1/2,2))$ such that $0 \leq \psi_0, \psi_1 \leq 1$, and for all $\eta \geq 0$, 
$$\sum_{l=0}^\infty \psi_l(\eta) = 1 \quad \textup{and} \sum_{j=-\infty}^\infty \psi_1(2^j \eta) = 1,$$ 
where $\psi_l (\eta) = \psi_1(2^{-(l-1)}\eta),$ for $l \geq 2.$ 

Let $T = m(x, \boldsymbol{L}, \boldsymbol{U})$ be the operator corresponding to the symbol $m$. We decompose $T$ into a countable sum of operators with the help of the approximate identity chosen above. 
\begin{rem} \label{rem:cut-off-spectral-variable-comparison}
We shall see below that we decompose the symbol function $m$ in the spectral-variables with the help of $\psi_l(|\tau|_1)$. One can also do the analysis by making use of $\psi_l(|(\tau, \kappa)|_1)$. The changes in the proofs that follow, would only be minor, and we leave the details. 
\end{rem} 

For each $J, l \in \mathbb{N}$ and $S = 1, 2, 3, \ldots$, define 
\begin{align} \label{0-symbol-decompose}
m^0_{J}(x, \tau, \kappa) & := \chi_J(x) m(x, \tau, \kappa) \sum_{j=-\infty}^0 \psi_1(2^j |\kappa|) \\ 
\nonumber m^0_{l}(x, \tau, \kappa) & := m(x, \tau, \kappa) \psi_l(|\tau|_1) \sum_{j=-\infty}^0 \psi_1(2^j |\kappa|), \\ 
\nonumber m^0_{J,l}(x, \tau, \kappa) & := \chi_J(x) m(x, \tau, \kappa) \psi_l(|\tau|_1) \sum_{j=-\infty}^0 \psi_1(2^j |\kappa|), \\ 
\nonumber m^S_{J}(x, \tau, \kappa) & := \chi_J(x) m(x, \tau, \kappa) \sum_{j = 1}^S \psi_1 (2^j |\kappa|), \\ 
\nonumber m^S_{l}(x, \tau, \kappa) & := m(x, \tau, \kappa) \psi_l(|\tau|_1) \sum_{j = 1}^S \psi_1 (2^j |\kappa|), \\ 
\nonumber m^S_{J,l}(x, \tau, \kappa) & := \chi_J(x) m(x, \tau, \kappa) \psi_l(|\tau|_1) \sum_{j = 1}^S \psi_1 (2^j |\kappa|), 
\end{align}
where $\chi_J$ are given by Lemma \ref{lem-grushin-partition-part-fix}. 

For $S \in \mathbb{N}$, let us denote by $T^S_{J}$, $T^S_{l}$, and $T^S_{J,l}$ the operators $m^S_{J}(x, \boldsymbol{L}, \boldsymbol{U})$, $m^S_{l}(x, \boldsymbol{L}, \boldsymbol{U})$, and $m^S_{J,l}(x, \boldsymbol{L}, \boldsymbol{U})$ respectively, then (formally) we have 
\begin{align*} 
& T^S_J = \sum_{l} T^S_{J,l}, \quad T^S_l = \sum_{J} T^S_{J,l}, \quad T^S = \sum_{J,l} T^S_{J,l}, \quad \text{and} \quad T = T^0 + \lim_{S \to \infty} T^S.
\end{align*} 

We shall be done if we could show that 
$$ \sup_{S \in \mathbb{N}} \|T^S\|_{op} < \infty. $$ 

Making use of the support condition of $\chi_J$, we begin with proving that $T^S_J$ (and hence $T^S_{J,l}$) are bounded operators on $L^2(\mathbb{R}^{n_1 + n_2})$ with their operator norms being uniform in $J$ and $S$ (respectively, in $J$, $l$, and $S$). 

\begin{lem} \label{freeze-lem}
There exists a constant $C>0$ such that  
\begin{equation} \label{freeze-estimate}
\|T^S_J\|_{op} \leq C \sup_{|\Gamma| \leq 2 \left( 1 + \lfloor\frac{Q}{4}\rfloor \right)} \sup_{x, \tau, \kappa} \left|X^{\Gamma} m(x, \tau, \kappa)\right| \quad \textup{for all } J, S \in \mathbb{N}.
\end{equation}
\end{lem} 
\begin{proof} 
It suffices to show that 
$$ \left\| \sum_{l = 0}^L T^S_{J,l} \right\|_{op} \lesssim \sup_{|\Gamma| \leq 2 \left( 1 + \lfloor\frac{Q}{4}\rfloor \right)} \sup_{x, \tau, \kappa} \left|X^{\Gamma} m(x, \tau, \kappa)\right|,$$ 
with the implicit bound uniform in $J, L, S \in \mathbb{N}$. 

Note that $ T^S_{J,l}$ is an integral operator. Let us denote it's kernel by $ k^{x,S}_{J,l}(x,y),$ where 
$$ k_{J,l}^{z,S} (x,y) = (2\pi)^{-n_2} \int_{\mathbb{R}^{n_2}} \sum_{\mu} m^S_{J,l} \left( z, (2 \mu + \tilde{1})|\lambda|, - \lambda \right) E_{\mu, \lambda} (x, y) \, d\lambda. $$

For any $f \in L^2(\mathbb{R}^{n_1 + n_2})$ we have $T^S_{J} f (x) = 0$ for all $x \in B(x_J, 2)^c$. Also note that $|B(x,1)| \sim |B(x_J,1)|$ for all $x \in B(x_J, 2)$, with the implicit bounds independent of $J$. Now,   
\begin{align*}
\left| \sum_{l = 0}^L T^S_{J, l} f (x) \right|^2 &= \left|\int_{\mathbb{R}^{n_1 + n_2}} \sum_{l = 0}^L k^{x,S}_{J,l}(x,y) f(y) \, dy \right|^2 \\ 
& \lesssim |B(x_J,1)|^{-2} \left|\int_{\mathbb{R}^{n_1 + n_2}} |B(x,1)| \sum_{l = 0}^L k^{x,S}_{J,l} (x,y) f(y) \, dy \right|^2 \\ 
& \leq |B(x_J,1)|^{-2} \left( \sup_{x_0 \in B(x_J, 2)} \left|\int_{\mathbb{R}^{n_1 + n_2}} |B(x_0,1)| \sum_{l = 0}^L k^{x_0,S}_{J,l} (x,y) f(y) \, dy \right| \right)^2 \\ 
& \lesssim |B(x_J,1)|^{-2} \sum_{|\Gamma| \leq 2 \left( 1 + \lfloor\frac{Q}{4}\rfloor \right)} \left\|\int_{\mathbb{R}^{n_1 + n_2}} X^{\Gamma}_{x_0} \sum_{l = 0}^L k^{x_0,S}_{J,l} (x,y) f(y) \, dy \right\|^2_{L^2(d x_0)},
\end{align*}
where the last inequality follows from Lemma \ref{grushin-Sobolev-embed}. 

Thus, 
\begin{align} \label{freeze-lem-ineq1}
\|T^S_{J} f\|^2_{L^2} & \lesssim |B(x_J,1)|^{-2} \sum_{|\Gamma| \leq 2 \left( 1 + \lfloor\frac{Q}{4}\rfloor \right)} \left\| \left\|\int_{\mathbb{R}^{n_1 + n_2}} X^{\Gamma}_{x_0} \left( \sum_{l = 0}^L k^{x_0,S}_{J,l}(x,y) \right) f(y) \, dy \right\|^2_{L^2(d x_0)} \right\|^2_{L^2(dx)} \\ 
\nonumber & = |B(x_J,1)|^{-2} \sum_{|\Gamma| \leq 2 \left( 1 + \lfloor\frac{Q}{4}\rfloor \right)} \left\| \left\|\int_{\mathbb{R}^{n_1 + n_2}} X^{\Gamma}_{x_0} \left( \sum_{l = 0}^L k^{x_0,S}_{J,l}(x,y) \right) f(y) \, dy \right\|^2_{L^2(dx)} \right\|^2_{L^2(d x_0)} \\
\nonumber & \lesssim \sum_{|\Gamma| \leq 2 \left( 1 + \lfloor\frac{Q}{4}\rfloor \right)} \sup_{x_0 \in B(x_J, 2)} \left\|\int_{\mathbb{R}^{n_1 + n_2}} X^{\Gamma}_{x_0} \left( \sum_{l = 0}^L k^{x_0,S}_{J,l}(x,y) \right) f(y) \, dy \right\|^2_{L^2(dx)}. 
\end{align}

Now, by Plancherel's theorem for the Grushin operator, 
\begin{align*} 
\left\|\int_{\mathbb{R}^{n_1 + n_2}} X^{\Gamma}_{x_0} \left( \sum_{l = 0}^L k^{x_0,S}_{J,l}(x,y) \right) f(y) \, dy \right\|_{L^2(dx)} & \leq \|f\|_2 \sup_{x_0, \tau, \kappa} \left|X^{\Gamma}_{x_0} \left( \sum_{l = 0}^L m^S_{J,l} (x_0, \tau, \kappa) \right) \right| \\
& = \|f\|_2 \sup_{x_0, \tau, \kappa} \left|X^{\Gamma}_{x_0} m^S_J (x_0, \tau, \kappa) \sum_{l = 0}^L \psi_l (|\tau|_1) \right| \\ 
& \lesssim \|f\|_2 \sup_{x_0, \tau, \kappa} \left|X^{\Gamma}_{x_0} m(x_0, \tau, \kappa)\right|, 
\end{align*}
where the last inequality follows from the uniform boundedness of partition of unities involving $\psi_j$'s, and of derivatives of functions $\chi_J$ (as shown in Lemma \ref{lem-grushin-partition-part-fix}).

Putting the above estimate in \eqref{freeze-lem-ineq1}, we get 
\begin{align*} 
\left\| \sum_{l = 0}^L T^S_{J,l} f \right\|_{L^2} & \lesssim \left( \sum_{|\Gamma| \leq 2 \left( 1 + \lfloor\frac{Q}{4}\rfloor \right)} \left(\sup_{x_0, \tau, \kappa} \left|X^{\Gamma}_{x_0} m(x_0, \tau, \kappa)\right|\right)^2 \right)^{1/2} \|f\|_2.
\end{align*}
This completes the proof of Lemma \ref{freeze-lem}.
\end{proof}

We claim the following $L^2$-operator norm estimates for $\left( T^S_{J,l} \right)^* T^S_{J^\prime,l^\prime}$ and $T^S_{J,l} \left( T^S_{J^\prime,l^\prime}\right)^*$. 

\begin{prop} \label{TT*-and-T*T-main-prop} 
We have 
\begin{align*}
\left\| \left( T^S_{J,l} \right)^* T^S_{J^\prime,l^\prime}\right\|_{op} & \lesssim_N \mathbbm{1}_{d(x_J, x_{J^{\prime}}) \leq 4 C_0} \, \,  2^{-N |l - l^{\prime}|} \|m\|^2_{S^0_{0,0}}, \, \text{for every} \, N \in \mathbb{N}, \\ 
\left \|T^S_{J,l} \left( T^S_{J^\prime,l^\prime} \right)^* \right\|_{op} & \lesssim_{N_0} \mathbbm{1}_{|l-l^{\prime}| \leq 2} \, \left(1 + d\left(x_J, x_{J^{\prime}}\right) \right)^{-4 N_0} \|m\|^2_{S^0_{0,0}}.
\end{align*}
with the implicit constant being independent of $J, J^{\prime}, l, l^{\prime}, S \in \mathbb{N}$. Here $C_0$ is the constant same as in \eqref{quasi-metric-constant}, and $N_0 = \left( \lfloor\frac{Q}{4}\rfloor + 1 \right)$. 
\end{prop}

Since we have from Lemma \ref{lem-grushin-partition-part-fix} that $ \displaystyle \sup_{x, J} \sum_{J \in \mathbb{N}} (1 + d \left( x, x_J \right) )^{-N} < \infty,$ whenever $N > Q$, the proof of $S^0_{0, 0}$ case in Theorem \ref{thm:old-joint-calc} (2) follows from Theorem \ref{Cotlar-Stein-lemma} via estimates of Lemma \ref{freeze-lem} and Proposition \ref{TT*-and-T*T-main-prop} combined together. 

We shall prove Proposition \ref{TT*-and-T*T-main-prop} over the next two sections. 


\section{Operator norm of \texorpdfstring{$\left( T^S_{J,l} \right)^* T^S_{J^\prime,l^\prime}$}.} \label{Sec-S00-proof-1} 
\begin{proof}[Proof of the first part of Proposition \ref{TT*-and-T*T-main-prop} (operator norm of $\left( T^S_{J,l} \right)^* T^S_{J^\prime,l^\prime}$)]

Let $\chi_l$ be the characteristic function on the real line of the interval $(0, 2^{l+2})$. For any $N \in \mathbb{N}$, we can write 
\begin{align*} 
\left( T^S_{J,l} \right)^* T^S_{J^\prime,l^\prime} & = \left( T^S_{J} \psi_l (G) \right)^* T^S_{J^\prime} \psi_{l^\prime} (G) = \left( T^S_{J} \chi_l (G) \psi_l (G) \right)^* T^S_{J^\prime} \chi_{l^\prime} (G) \psi_{l^\prime} (G) \\ 
& = \left\{\psi_{l} (G) (I+G)^{-N}\right\} \left\{(I+G)^{N} \left( T^S_{J} \chi_l (G) \right)^* (I+G)^{-N}\right\} \\ 
& \quad \quad \left\{(I+G)^{N} \left( T^S_{J^\prime} \chi_{l^\prime} (G) \right) (I+G)^{-N}\right\} \left\{(I+G)^{N} \psi_{l^{\prime}} (G) \right\}.
\end{align*} 

We claim the following estimates on each one of the four operators on the right hand side of the above expression: 
\begin{align} 
\label{decay-1} \left\| \psi_{l} (G) (I+G)^{-N} \right\|_{op} & \lesssim_N 2^{- N l}. \\ 
\label{uniform-1} \left\| (I+G)^{N} \left( T^S_{J} \chi_l (G) \right)^* (I+G)^{-N} \right\|_{op} & \lesssim_N \|m\|_{S^0_{0, 0}}. \\
\label{uniform-2}\left\|(I+G)^{N} \left( T^S_{J^\prime} \chi_{l^{\prime}} (G) \right) (I+G)^{-N} \right\|_{op} & \lesssim_N \|m\|_{S^0_{0, 0}} \\
\label{growth-1}\left\|(I+G)^{N} \psi_{l^{\prime}} (G) \right\|_{op} & \lesssim_N 2^{N l^{\prime}}.
\end{align}
 
Estimates (\ref{decay-1}) and (\ref{growth-1}) are straightforward. We shall prove estimates (\ref{uniform-1}) and (\ref{uniform-2}). In order to prove (\ref{uniform-1}), note first that 
\begin{align*} 
& \left( \left( T^S_{J} \psi_l (G) \right)^* (I+G)^{-N} f \right) (x) \\ 
& = (2\pi)^{-n_2} \int_{\mathbb{R}^{n_2}} \sum_{\mu} \left(\int_{\mathbb{R}^{n_1 + n_2}} \overline{ \left( \chi_l m^S_{J} \right) (z, (2 \mu + \tilde{1}) |\lambda|, - \lambda) } \left( (I+G)^{-N}f \right) (z) \Phi_{\mu}^\lambda(z^\prime) e^{i \lambda \cdot z^{\prime \prime}} \, dz \right) \\ 
& \quad \quad \quad \quad \quad \Phi_{\mu}^\lambda(x^\prime) e^{-i \lambda \cdot x^{\prime \prime}} \, d\lambda.
\end{align*}

Therefore, 
\begin{align*} 
& \left( (I+G)^{N} \left( T^S_{J} \psi_{l^{\prime}} (G) \right)^* (I+G)^{-N}f \right) (x) \\ 
&= (2\pi)^{-n_2} \int_{\mathbb{R}^{n_2}} \sum_{\mu} \left(\int_{\mathbb{R}^{n_1 + n_2}} \overline{ \left( \chi_l m^S_{J} \right)(z, (2 \mu + \tilde{1}) |\lambda|, - \lambda) } \left( (I+G)^{-N}f \right) (z) \Phi_{\mu}^\lambda(z^\prime) e^{i \lambda \cdot z^{\prime \prime}} \, dz \right) \\
&\quad \quad \quad \quad \quad (I+G)^{N} \left(\Phi_{\mu}^\lambda(x^\prime) e^{-i \lambda \cdot x^{\prime \prime}}\right) \, d\lambda \\
&= (2\pi)^{-n_2} \int_{\mathbb{R}^{n_2}} \sum_{\mu} \left(\int_{\mathbb{R}^{n_1 + n_2}} \overline{ \left( \chi_l m^S_{J} \right)(z, (2 \mu + \tilde{1}) |\lambda|, - \lambda) }  \left( (I+G)^{-N}f \right) (z) \Phi_{\mu}^\lambda(z^\prime) e^{i \lambda \cdot z^{\prime \prime}} \, dz \right) \\
&\quad \quad \quad \quad \quad \left(1+(2|\mu|+n_1)|\lambda|\right)^N \left(\Phi_{\mu}^\lambda(x^\prime) e^{-i \lambda \cdot x^{\prime \prime}}\right) \, d\lambda \\
&= (2\pi)^{-n_2} \int_{\mathbb{R}^{n_2}} \sum_{\mu} \left( \int_{\mathbb{R}^{n_1 + n_2}} \overline{ \left( \chi_l m^S_{J} \right)(z, (2 \mu + \tilde{1}) |\lambda|, - \lambda) } \left( (I+G)^{-N}f \right) (z) \right. \\
&\quad \quad \quad \quad \quad \left. \left\{ (I+G)^N \left(\Phi_{\mu}^\lambda(z^\prime) e^{i \lambda \cdot z^{\prime \prime}} \right) \right\} \, dz \right) \Phi_{\mu}^\lambda(x^\prime) e^{-i \lambda \cdot x^{\prime \prime}} \, d\lambda \\
&= (2\pi)^{-n_2} \int_{\mathbb{R}^{n_2}} \sum_{\mu} \left( \int_{\mathbb{R}^{n_1 + n_2}} (I+G)^N \left\{\overline{ \left( \chi_l m^S_{J} \right)(z, (2 \mu + \tilde{1}) |\lambda|, - \lambda) } (I+G)^{-N}f\right\} (z) \right. \\
&\quad \quad \quad \quad \quad \left. \Phi_{\mu}^\lambda(z^\prime) e^{i \lambda \cdot z^{\prime \prime}} \, dz \right) \Phi_{\mu}^\lambda(x^\prime) e^{-i \lambda \cdot x^{\prime \prime}} \, d\lambda, 
\end{align*}
where the last equality follows from the integration by parts in $z$-variable. 

For fixed $N \in \mathbb{N}$ the above expression is a finite linear combination of terms of the following form: 
\begin{align*}
& \int_{\mathbb{R}^{n_2}} \sum_{\mu} \left( \int_{\mathbb{R}^{n_1 + n_2}} \overline{ \left( \chi_l \left(X_z^\Gamma m^S_{J} \right)\right)(z, (2 \mu + \tilde{1}) |\lambda|, - \lambda) } \left(X_z^{\tilde{\Gamma}} (I+G)^{-N} f\right) (z) \Phi_{\mu}^\lambda(z^\prime) e^{i \lambda \cdot z^{\prime \prime}} \, dz \right) \\
&\quad \quad \quad \quad \Phi_{\mu}^\lambda(x^\prime) e^{-i \lambda \cdot x^{\prime \prime}} \, d\lambda,
\end{align*}
with $|\Gamma| + |\tilde{\Gamma}| \leq 2N.$

Now, we can apply Lemma \ref{freeze-lem} for the operator corresponding to the symbol $\chi_l X_z^\Gamma m^S_{J}$ to conclude that $L^2$-norm of the above functions is bounded by \begin{align*}
\sup_{|\Gamma^{\prime}| \leq 2 \left( 1 + \lfloor\frac{Q}{4}\rfloor \right)} \sup_{x, \tau, \kappa} \left|X^{\Gamma^{\prime} + \Gamma}_x m(x, \tau, \kappa) \right| \left\|X_z^{\tilde{\Gamma}} (I+G)^{-N} f\right\|_2.
\end{align*}
But then, since $|\tilde{\Gamma}| \leq 2N$, we can invoke the Plancherel theorem for the Grushin operator to get that 
$\left\|X_z^{\tilde{\Gamma}} (I+G)^{-N} f\right\|_2 \leq C_N \left\|f\right\|_2$. Hence, the claimed estimate (\ref{uniform-1}) follows.

The proof of the claimed estimate (\ref{uniform-2}) is very similar to the above process used to prove (\ref{uniform-1}). For this note first that $\left( (I+G)^{-N} f \right)^{\lambda} = (I+H(\lambda))^{-N}f^{\lambda}.$ Now, 
\begin{align*} 
& \left( \chi_{l^\prime} (G) T^S_{J^\prime} (I+G)^{-N}f \right) (x) \\ 
&= (2\pi)^{-n_2} \int_{\mathbb{R}^{n_2}} \sum_{\mu} \left( \chi_{l^\prime} m^S_{J^\prime} \right) (x, (2 \mu + \tilde{1}) |\lambda|, -\lambda) \left((I+H(\lambda))^{-N}f^{\lambda}, \Phi_{\mu}^\lambda\right) \Phi_{\mu}^\lambda(x^\prime) e^{-i \lambda \cdot x^{\prime \prime}} \, d\lambda.
\end{align*}

Therefore, 
\begin{align*} 
& (2\pi)^{n_2} \left( (I+G)^N \left( \chi_{l^\prime} (G) T^S_{J^\prime} \right) (I+G)^{-N} f \right) (x) \\ 
&= \int_{\mathbb{R}^{n_2}} \sum_{\mu} (I+G)^N \left\{ \left( \chi_{l^\prime} m^S_{J^\prime} \right) (x, (2 \mu + \tilde{1}) |\lambda|, -\lambda) \Phi_{\mu}^\lambda(x^\prime) e^{-i  \lambda \cdot x^{\prime \prime}}\right\} \left((I+H(\lambda))^{-N}f^{\lambda}, \Phi_{\mu}^\lambda\right) \, d\lambda.
\end{align*}

Once again, one can make use of the Leibniz formula for multiplication while applying $(I+G)^N$ on the product of $\left( \chi_{l^\prime} m^S_{J^\prime} \right) (x, (2 \mu + \tilde{1}) |\lambda|, -\lambda)$ and $\Phi_{\mu}^\lambda(x^\prime) e^{-i \lambda \cdot x^{\prime \prime}}$, and thus one can write the above expression as a finite linear combination of terms of the following form: 
\begin{align*}
& \int_{\mathbb{R}^{n_2}} \sum_{\mu} \left( \chi_{l^\prime} X_x^{\Gamma} m^S_{J^\prime} \right) (x, (2 \mu + \tilde{1}) |\lambda|, -\lambda) X_x^{\tilde{\Gamma}}\left(\Phi_{\mu}^\lambda(x^\prime) e^{-i \lambda \cdot x^{\prime \prime}}\right) \left((I+H(\lambda))^{-N} f^{\lambda}, \Phi_{\mu}^\lambda\right) \, d\lambda, 
\end{align*}
where $|\Gamma| + |\tilde{\Gamma}| \leq 2N.$ 

The above term can be written as a finite linear combination of 
\begin{align*}
& \int_{\mathbb{R}^{n_2}} \sum_{\mu} \left( \chi_{l^\prime} X_x^{\Gamma} m^S_{J^\prime} \right)(x, (2 \mu + \tilde{1}) |\lambda|, -\lambda) \left\{ A(\lambda)^{\tilde{\alpha}} A^*(\lambda)^{\tilde{\alpha}^{\prime}} \Phi_{\mu}^\lambda \right\} (x^\prime) e^{-i \lambda \cdot x^{\prime \prime}} \\ 
& \quad \quad \quad \quad \left((I+H(\lambda))^{-N} f^{\lambda}, \Phi_{\mu}^\lambda \right) \, d\lambda \\
&= \int_{\mathbb{R}^{n_2}} \sum_{\mu} \left( \chi_{l^\prime} X_x^{\Gamma} m^S_{J^\prime} \right)(x, (2 (\mu + \tilde{\alpha} - \tilde{\alpha}^{\prime}) + \tilde{1}) |\lambda|, -\lambda) \, \Phi_{\mu}^\lambda(x^\prime) e^{-i  \lambda \cdot x^{\prime \prime}} \\
& \quad \quad \quad \quad \left((I+H(\lambda))^{-N}f^{\lambda}, A(\lambda)^{ \tilde{\alpha}^{\prime}} A^*(\lambda)^{\tilde{\alpha}} \Phi_{\mu}^\lambda\right) \, d\lambda \\
&= \int_{\mathbb{R}^{n_2}} \sum_{\mu} \left( \chi_{l^\prime} X_x^{\Gamma} m^S_{J^\prime} \right)(x, (2 (\mu + \tilde{\alpha} - \tilde{\alpha}^{\prime}) + \tilde{1}) |\lambda|, -\lambda) \, \Phi_{\mu}^\lambda(x^\prime) e^{-i  \lambda \cdot x^{\prime \prime}} \\
&\quad \quad \quad \quad \left(A(\lambda)^{\tilde{\alpha}} A^* (\lambda)^{\tilde{\alpha}^{\prime}} (I+H(\lambda))^{-N}f^{\lambda}, \Phi_{\mu}^\lambda\right) \, d\lambda, \\ 
& = \left( \widetilde{T} \left( R^N_{\tilde{\alpha},\tilde{\alpha}^{\prime}} f \right) \right) (x), 
\end{align*}
where $|\tilde{\alpha}| + |\tilde{\alpha}^{\prime}| \leq |\tilde{\Gamma}|.$ In the above expression, the operator $\widetilde{T}$ is corresponding to the symbol $ \left( \chi_{l^\prime} X_x^{\Gamma} m^S_{J^\prime} \right) (x, (2 (\mu + \tilde{\alpha} - \tilde{\alpha}^{\prime}) + \tilde{1}) |\lambda|, -\lambda) $, and the operator $R^N_{\tilde{\alpha},\tilde{\alpha}^{\prime}}$ is defined by 
$$ \left( R^N_{\tilde{\alpha},\tilde{\alpha}^{\prime}} f \right) (x) := (2\pi)^{n_2} \int_{\mathbb{R}^{n_2}} e^{-i  \lambda \cdot x^{\prime \prime}} \left( A(\lambda)^{\tilde{\alpha}} A^* (\lambda)^{\tilde{\alpha}^{\prime}} (I+H(\lambda))^{-N} \right) f^{\lambda} (x^\prime) \, d\lambda.$$ 

Now, the boundedness of $R^N_{\tilde{\alpha}, \tilde{\alpha}^{\prime}}$ on $L^2(\mathbb{R}^{n_1 + n_2})$ follows from the Plancherel theorem for the Grushin operator (since $|\tilde{\alpha}| + |\tilde{\alpha}^{\prime}| \leq 2N$), with bound depending on $N$. On the other hand, it follows from Lemma \ref{freeze-lem} that for each fixed $\tilde{\alpha}$ and $\tilde{\alpha}^{\prime}$, the operator  $\widetilde{T}$ is bounded on $L^2(\mathbb{R}^{n_1 + n_2})$. This completes the proof of the estimate (\ref{uniform-2}). 

Putting all four of the estimates (\ref{decay-1}), (\ref{uniform-1}), (\ref{uniform-2}), and (\ref{growth-1}) together, we see that 
\begin{equation} \label{TT*-final-1}
\left\| \left( T^S_{J,l} \right)^* T^S_{J^\prime,l^\prime} \right\|_{op} \lesssim_N \|m\|^2_{S^0_{0, 0}} 2^{- N l} 2^{N l^{\prime}}.
\end{equation}

Since $\left\| \left( T^S_{J,l} \right)^* T^S_{J^\prime,l^\prime}\right\|_{op} = \left\| \left( T^S_{J^\prime,l^\prime} \right)^* T^S_{J,l} \right\|_{op},$ one can interchange the role of $(J,l)$ and $(J^\prime,l^\prime)$ in the estimate (\ref{TT*-final-1}) to conclude that 
\begin{equation} \label{TT*-final-2}
\left\| \left( T^S_{J,l} \right)^* T^S_{J^\prime,l^\prime} \right\|_{op} \lesssim_N \|m\|^2_{S^0_{0, 0}} 2^{- N |l - l^{\prime}|}.
\end{equation}

Finally, observe from the integral kernel representation of $\left( T^S_{J,l} \right)^* T^S_{J^\prime,l^\prime}$ that this operator vanishes unless $B(x_J, 2) \cap B(x_{J^\prime}, 2) \neq \emptyset$, that is, unless $d(x_J, x_{J^\prime}) \leq 4 C_0$, with $C_0$ same as in \eqref{quasi-metric-constant}. This completes the proof of the claimed operator norm of $\left( T^S_{J,l} \right)^* T^S_{J^\prime,l^\prime}$ in Proposition \ref{TT*-and-T*T-main-prop}.
\end{proof}


\section{Operator norm of \texorpdfstring{$T^S_{J,l} \left( T^S_{J^\prime,l^\prime} \right)^*$}.} \label{Sec-S00-proof-2} 
Recall the metric $d$ defined in (\ref{def:distance-1}): 
\begin{align*} 
d(x,y) := \left(\left|x^\prime - y^\prime \right|^4 + \frac{\left|x^{\prime \prime} - y^{\prime \prime}\right|^4}{\left|x^{\prime \prime} - y^{\prime \prime}\right|^2 + \left(\left|x^\prime \right|^2 + \left|y^\prime \right|^2\right)^2}\right)^{1/4}. 
\end{align*}

In view of the nature of the metric $d$, the kernel analysis needs to be performed carefully for points $x$ with $x^{\prime}$ near $0$ versus those with $x^{\prime}$ away from $0$. 

In what follows, we make use of the following functions: 

\begin{alignat}{2} \label{case-by-distances} 
d_1(x,y) &= \left(\left|x^\prime - y^\prime \right|^4 + \frac{\left|x^{\prime \prime} - y^{\prime \prime}\right|^4}{\left|x^\prime \right|^4}\right)^{1/4}, & \quad \textup{when } \left|x^\prime\right| > 3, \\
\nonumber d_2(x,y) &= \left(\left|x^\prime - y^\prime \right|^4 + \left|x^{\prime \prime} - y^{\prime \prime}\right|^2\right)^{1/4}, & \quad \textup{when } \left|x^\prime\right| < 5.
\end{alignat}

Let us explain the purpose of introducing the functions $d_1$ and $d_2$.
Note first that $d$ is bounded from above by any of the $d_i$'s in their respective domains. Now, when we apply powers of the distance function $(1 + d(x,y))^N$ on the kernel of $T^S_{J,l} \left( T^S_{J^\prime,l^\prime} \right)^*$, we get to estimate the action of $(x^{\prime} - y^{\prime})$, $(x^{\prime \prime} - y^{\prime \prime})$, and their various powers with the help of Lemma \ref{weighted-kernel-estimate-3}. Now, what we get is a finite linear combination of terms which are products of kernels of Grushin multipliers and monomials in $x^{\prime}$. When $x^{\prime}$ is near zero, being bounded it does not create any issue while we estimate $L^2$-norms of the associated operators. But, we have no control on powers of $x^{\prime}$ that are in the denominator in the definition on $d$. So, when $x^{\prime}$ is near zero, it is important not to take it in the denominator of the distance function. On the other hand, when $x^{\prime}$ is away from $0$, it is important to keep the denominator in the definition of $d$ which in fact dominates the above mentioned monomials. In addition, note that we have also discarded $(x^{\prime \prime} - y^{\prime \prime})$ from the denominator in $d_1$. This along with the advantage of taking $\left|x^\prime \right|^4 $ in place of $\left( \left|x^\prime \right|^2 + \left|y^\prime \right|^2 \right)^2$ in $d_1$ can best be understood in the proof of Proposition \ref{prop:weighted-kernel-estimate-1}. 

Let us fix a $C^\infty$ smooth function $\mathfrak{R}$ on $\R^{n_1}$ which is supported on the set $|x^\prime| < 5$ and $\mathfrak{R}\equiv 1$ on $|x^\prime|\leq 3$. Following the idea of the proof of Lemma \ref{lem-grushin-partition-part-fix}, one can verify that for every $\Gamma \in \mathbb{N}^{n_1 + n_1 n_2}$ and $N \in \mathbb{N}$ there exists a constant $C_{\Gamma, N} > 0$ such that for any $J, J^{\prime} \in \mathbb{N}$, we have 
\begin{align} \label{four-dist-der}
\sup_{\substack{x \in B(x_J, 2) \\ y \in B(x_{J^{\prime}}, 2)}} \left| X^{{\Gamma}}_{x} \left\{ \left( 1 - \mathfrak{R}(x^{\prime}) \right) \left( 1 + \left( d_1(x,  y) \right)^4 \right)^{-N} \right\} \right| 
\leq C_{\Gamma, N} \left( 1 + \left( d(x_J, x_{J^{\prime}}) \right)^4 \right)^{-N}; \\ 
\nonumber \sup_{\substack{x \in B(x_J, 2) \\ y \in B(x_{J^{\prime}}, 2)}} \left| X^{{\Gamma}}_{x} \left\{ \mathfrak{R}(x^{\prime}) \left( 1 + \left( d_2(x,  y) \right)^4 \right)^{-N} \right\} \right| 
\leq C_{\Gamma, N} \left( 1 + \left( d(x_J, x_{J^{\prime}}) \right)^4 \right)^{-N}. 
\end{align}

\begin{lem} \label{ker-TT*-pieces}
We have 
\begin{align*}
T^S_{J,l} \left( T^S_{J^\prime,l^\prime} \right)^* f(x) = \int_{\mathbb{R}^{n_1 + n_2}} f(y) k_{J,l,J^{\prime}, l^{\prime}}^{x, y, S} (x,y) \, dy, 
\end{align*} 
with the integral kernel $\displaystyle k_{J,l,J^{\prime}, l^{\prime}}^{u,v,S} (x,y)$ given by 
\begin{align*}
(2\pi)^{-n_2} \int_{\mathbb{R}} \sum_{\mu} m^S_{J,l} \left( u, (2 \mu + \tilde{1}) |\lambda|, - \lambda \right) \overline{m^S_{J^{\prime}, l^{\prime}} \left(v, (2 \mu + \tilde{1}) |\lambda|, - \lambda \right)} E_{\mu, \lambda} (x, y) \, d\lambda. 
\end{align*}
\end{lem}

Proof of Lemma \ref{ker-TT*-pieces} follows easily as an application of the Euclidean Parseval's theorem on $\mathbb{R}^{n_2}$ and the orthonormality relations for $\Phi_{\mu}^\lambda$'s on $\mathbb{R}^{n_1}$, when applied to the integral kernel in terms of those of operators $T^S_{J,l}$ and $\left( T^S_{J^\prime,l^\prime} \right)^*$. 

\begin{proof}[Proof of the second part of Proposition \ref{TT*-and-T*T-main-prop} (operator norm of $T^S_{J,l} \left( T^S_{J^\prime,l^\prime} \right)^*$)] \

Note first that if $|l-l^{\prime}| > 2$ then $T^S_{J,l} \left( T^S_{J^\prime,l^\prime} \right)^* = 0$. So, let us assume that $|l-l^{\prime}| \leq 2.$ Take and fix $f \in \mathcal{S} (\mathbb{R}^{n_1 + n_2})$. Let us denote the kernel of the operator $T_{J,l} T^*_{J^\prime,l^\prime}$ by $k_{J,l,J^{\prime}, l^{\prime}}^{x, y, S}(x,y)$ (see Lemma \ref{ker-TT*-pieces}). We will also make use of the fact that $k_{J,l,J^{\prime}, l^{\prime}}^{x, y, S}(x,y) = 0$ for any $x \notin B(x_J, 2)$ or $y \notin B(x_{J^{\prime}}, 2)$. Then, we can write 
\begin{align} \label{TT*-ineq1}
\left\| T^S_{J,l} \left( T^S_{J^\prime,l^\prime} \right)^* f\right\|_{L^2} &  \leq Term(1) + Term(2), 
\end{align}
where 
\begin{align} \label{TT*-ineq1-decompose}
Term(1) &= \left\|(1-\mathfrak{R}(x^\prime)) \int_{\mathbb{R}^{n_1 + n_2}} k_{J,l,J^{\prime}, l^{\prime}}^{x, y, S}(x,y) f(y) \, dy \,  \right\|_{L^2(dx)}, \\ 
\nonumber Term(2) &= \left\|\mathfrak{R}(x^\prime) \int_{\mathbb{R}^{n_1 + n_2}} k_{J,l,J^{\prime}, l^{\prime}}^{x, y, S}(x,y) f(y) \, dy \,  \right\|_{L^2(dx)}.
\end{align}

Let us also write 
\begin{align} \label{TT*-kernel-1}
k_{J,l,J^{\prime}, l^{\prime}}^{x_0, v, S, \Gamma}(x,y) := X^{\Gamma}_{x_0} k_{J,l,J^{\prime}, l^{\prime}}^{x_0, v, S}(x,y). 
\end{align}

\medskip \noindent \underline{\textbf{Analysis of $\bf Term(1)$}} : 

For any $N \in \mathbb{N}$ and $x \in B(x_J, 2)$ with $|x^{\prime}| > 3$, we have 
\begin{align} \label{integrand-mult-div} 
& |B(x_J,1)| \left|\int_{\mathbb{R}^{n_1 + n_2}} (1-\mathfrak{R}(x^\prime))k_{J,l,J^{\prime}, l^{\prime}}^{x, y, S}(x,y) f(y) \, dy \right| \\
\nonumber = & |B(x_J,1)| \left|\int_{\mathbb{R}^{n_1 + n_2}}(1-\mathfrak{R}(x^\prime)) \frac{\left( 1 + \left( d_1(x, y) \right)^4 \right)^N}{\left( 1 + \left( d_1(x, y) \right)^4 \right)^N} k_{J,l,J^{\prime}, l^{\prime}}^{x, y, S} (x,y) f(y) \, dy \right| \\
\nonumber \leq & |B(x_J,1)| \sup_{x_0 \in B(x_J, 2)} \left|\int_{\mathbb{R}^{n_1 + n_2}} \left( 1 + \left( d_1(x, y) \right)^4 \right)^N \frac{(1-\mathfrak{R}(x_0^\prime)) k_{J,l,J^{\prime}, l^{\prime}}^{x_0, y, S} (x,y)}{\left( 1 + \left( d_1(x_0, y) \right)^4 \right)^N} f(y) \, dy \right| \\ 
\nonumber \lesssim & \sum_{|\Gamma| \leq  2 \left( 1 + \lfloor\frac{Q}{4}\rfloor \right)} \left\| \int_{\mathbb{R}^{n_1 + n_2}} \left( 1 + \left( d_1(x, y) \right)^4 \right)^N X^{\Gamma}_{x_0} \left\{ \frac{(1-\mathfrak{R}(x_0^\prime)) k_{J,l,J^{\prime}, l^{\prime}}^{x_0, y, S}(x,y)}{\left( 1 + \left( d_1(x_0, y) \right)^4 \right)^N} \right\} f (y) \, dy \right\|_{L^2(dx_0)} \\ 
\nonumber \lesssim & \sum_{|\Gamma + \tilde{\Gamma}| \leq 2 \left( 1 + \lfloor\frac{Q}{4}\rfloor \right)} \left\|\int_{\mathbb{R}^{n_1 + n_2}} \left( 1 + \left( d_1(x, y) \right)^4 \right)^N k_{J,l,J^{\prime}, l^{\prime}}^{x_0, y, S, \Gamma } (x,y) f_{\tilde{\Gamma}}(x_0, y) \, dy \right\|_{L^2(dx_0)}, 
\end{align}
where the second last inequality follows from Lemma \ref{grushin-Sobolev-embed}. Here we have used the notation $f_{\tilde{\Gamma}}(x_0, y) := f(y) X^{\tilde{\Gamma}}_{x_0} \left\{ (1-\mathfrak{R}(x_0^\prime)) \left( 1 + \left( d_1(x_0, y) \right)^4 \right)^{-N} \right\}.$

We now claim that for $N_0 = \lfloor\frac{Q}{4}\rfloor+1$ and $\phi\in L^2(\R^{n_1+n_2})$,
\begin{align} \label{TT*-kernel-claim1}
& \left\| \mathbbm{1}_{\{|x^{\prime}| \geq 3\}}(x) \int_{\mathbb{R}^{n_1 + n_2}} \left( 1 + \left( d_1(x, y) \right)^4 \right)^{N_0} k_{J,l,J^{\prime}, l^{\prime}}^{x_0, y, S, \Gamma}(x,y) \phi(y) \, dy \right\|_{L^2(dx)} \\ 
\nonumber & \quad \lesssim_{N_0} \mathbbm{1}_{|l-l^{\prime}| \leq 2} \|m\|^2_{S^0_{0, 0}} \| \phi \|_{L^2}. 
\end{align}

Assuming the claimed estimate (\ref{TT*-kernel-claim1}) for now, we get from (\ref{integrand-mult-div}) that 
\begin{align*}
Term(1) \lesssim_{N_0} \mathbbm{1}_{|l-l^{\prime}| \leq 2} \|m\|^2_{S^0_{0, 0}} \left\| f_{\tilde{\Gamma}} (x_0, y) \right\|_{L^2(dx_0 \, dy)}.
\end{align*}

But, we know from \eqref{four-dist-der} that  
\begin{align*}
\sup_{\substack{x_0 \in B(x_J, 2) \\ y \in B(x_{J^{\prime}}, 2)}} \left| X^{\tilde{\Gamma}}_{x_0} \left\{ \left( 1 - \mathfrak{R}(x^{\prime}) \right) \left( 1 + \left( d_1(x_0, y) \right)^4 \right)^{-N_0} \right\} \right| \lesssim_{\tilde{\Gamma}, N_0} \left( 1 + \left( d( x_J, x_{J^{\prime}} ) \right)^4 \right)^{-N_0},
\end{align*}
and therefore 
\begin{align*}
Term(1) &\lesssim_{N_0} \mathbbm{1}_{|l-l^{\prime}| \leq 2} \|m\|^2_{S^0_{0, 0}} \left( 1 + \left( d( x_J, x_{J^{\prime}} ) \right)^4 \right)^{-N_0} \left\|f\right\|_{L^2}.
\end{align*}

$Term (2)$ can be analysed in a similar manner. We shall prove (\ref{TT*-kernel-claim1}) in Proposition \ref{prop:weighted-kernel-estimate-1}. This completes the proof of the claimed operator norm of $T^S_{J,l} \left( T^S_{J^\prime,l^\prime} \right)^*$ of Proposition \ref{TT*-and-T*T-main-prop}.
\end{proof} 


\begin{prop} \label{prop:weighted-kernel-estimate-1}
Estimate (\ref{TT*-kernel-claim1}) holds true uniformly in $S \in \mathbb{N}$. 
\end{prop}
\begin{proof} It is clear from the definition of $d_1$ that it suffices to show that 
\begin{align} \label{est-1-for-TT*-kernel-claim1}
& \left\| \mathbbm{1}_{\{|x^{\prime}| \geq 3\}}(x) \int_{\mathbb{R}^{n_1 + n_2}} \left|x^\prime - y^\prime \right|^{4 N_1} \frac{\left|x^{\prime \prime} - y^{\prime \prime}\right|^{4 N_2}}{\left|x^\prime \right|^{4 N_2}} k^{x_0, y, S, \Gamma}_{J,l,J^{\prime},l^{\prime}}(x,y) \phi(y) \, dy \right\|_{L^2(dx)} \\ 
\nonumber & \quad \lesssim_N \mathbbm{1}_{|l-l^{\prime}| \leq 2} \|m\|^2_{S^0_{0, 0}} \|\phi\|_{L^2}, 
\end{align}
for any choice of $N_1, N_2 \in \mathbb{N}$ such that $N_1 + N_2 \leq N_0.$ 

With $k^{x_0, y, S, \Gamma}_{J, l, J^{\prime}, l^{\prime}}$ as in \eqref{TT*-kernel-1} and $m^{x_0, y, S, \Gamma}_{J,l,J^{\prime}, l^{\prime}}$ denoting the corresponding product of symbols, we can express $\left|x^\prime - y^\prime \right|^{4 N_1} \frac{\left|x^{\prime \prime} - y^{\prime \prime}\right|^{4 N_2}}{\left|x^\prime \right|^{4 N_2}} k^{x_0, y, S, \Gamma}_{J, l, J^{\prime}, l^{\prime}}(x,y)$ (with the help of Remark \ref{rem2-indices-clarity}, using $4 N_2$ in place of $2 N_2$) as a finite linear combination of terms of the form 
\begin{align*} 
& \frac{{x^{\prime}}^{\alpha_{1}}}{\left|x^\prime \right|^{4 N_2}} \int_{[0,1]^{\nu_1} \times \Omega^{\nu_2} \times [0,1]^{4 N_1}} \int_{\mathbb{R}^{n_2}} \Theta_l(\lambda) \sum_{\mu} C_{\mu, \tilde{a}, \vec{a}} \left(\tau^{\frac{1}{2} \theta_1} \partial_\tau^{\theta_2}\partial_{\kappa}^{\beta} m^{x_0, y, S, \Gamma}_{J,l,J^{\prime}, l^{\prime}} \right) ((2\mu + \tilde{a} + \vec{a}(\omega)) |\lambda|, - \lambda) \\ 
\nonumber & \quad \quad \quad \quad \quad \quad \quad \quad \quad \quad \quad \quad \quad \quad \quad \quad \quad \Phi_{\mu + \tilde{\mu}}^{\lambda}(x^{\prime}) \Phi_{\mu}^{\lambda}(y^{\prime}) e^{- i \lambda \cdot (x^{\prime \prime} - y^{\prime \prime})} \, g(\omega) \, d \lambda \, d \omega,  
\end{align*}
where $\nu_1, \nu_2 \leq 4 N_2 - (|\beta| + l)$, $| \alpha_1 | \leq 4 N_2 - (|\beta|+l)$, $|\mathcal{\theta}_1| \leq |\mathcal{\theta}_2| \leq 4 N_1+ 8 N_2 - 2  (|\beta|+l)$, $|\mathcal{\theta}_2| - \frac{1}{2} |\mathcal{\theta}_1| = 2 N_1 + 4 N_2 - (|\beta|+l)- \frac{|\alpha_1|}{2} \geq 2 N_1 + 2 N_2 - \frac{1}{2}  (|\beta|+l)$, $C_{\mu, \tilde{a}, \vec{a}}$ is a bounded function of $\mu$ and $\vec{a}$, and $\Theta_l$ is a continuous function on $\R^{n_2} \setminus \{0\}$ which is homogeneous of degree $-l$. 

In view of the Leibniz rule we have 
\begin{align*}
& \left( \tau^{\frac{1}{2} \theta_1} \partial_\tau^{\theta_2} \partial_{\kappa}^{\beta}m^{x_0, y, S, \Gamma}_{J,l,J^{\prime}, l^{\prime}}\right) ((2\mu + \tilde{a} + \vec{a}(\omega)) |\lambda|, - \lambda)  \\ 
& = \sum_{\tilde{\theta}_2 \leq \theta_2} \sum_{\beta_1+\beta_2=\beta} C_{\tilde{\theta}_2, \beta_1, \beta_2} \left(\tau^{\frac{1}{2} \theta_1} \partial_\tau^{{\theta}_2 - \tilde{\theta}_2}  \partial_{\kappa}^{\beta_1} X^{\Gamma}_{x_0} m^S_{J,l}\right) (x_0, (2\mu + \tilde{a} + \vec{a}(\omega)) |\lambda|, - \lambda) \\ 
& \quad \quad \quad \quad \quad \quad \quad \overline{\left(\partial_\tau^{\tilde{\theta}_2} \partial_{\kappa}^{\beta_2}m^S_{J^{\prime},l^{\prime}} \right)(y, (2\mu + \tilde{a} + \vec{a}(\omega)) |\lambda|, - \lambda)}. 
\end{align*}
Now, as $m^S_{J,l}(u, \tau,\kappa) = \chi_J(u) m^S(u,\tau,\kappa)\psi_l(\tau)$, the above will vanish if $\left| l-l^{\prime} \right| > 2$. Also, we know that if $m\in S^0_{0, 0}$, then $m_{J,l}$ will also be in $S^0_{0, 0}$, for all $J$ and $l$, with bound uniform in $J$ and $l$ (see Lemma \ref{freeze-lem}).

Thus we can write each piece in the above finite sum as the kernel of the composition of two operators $A^S_1(x_0, \alpha_1, \beta_1, \theta_1, \theta_2, \tilde{\theta}_2)$ and $A^S_2(x, \beta_2, \theta_2, \tilde{\theta}_2)^*$, where $A^S_1(x_0, \alpha_1, \beta_1, \theta_1, \theta_2, \tilde{\theta}_2)$ corresponds to the kernel 
\begin{align*}
& \frac{{x^{\prime}}^{\alpha_{1}}}{\left|x^\prime\right|^{4 N_2}} \int_{[0,1]^{\nu_1} \times \Omega^{\nu_2} \times [0,1]^{4 N_1}} \int_{\mathbb{R}} \sum_{\mu} C_{\mu, \tilde{a}, \vec{a}} \left(\tau^{\frac{1}{2} (\theta_1 - \tilde{\theta}_2)} \partial_{\tau}^{\theta_2 - \tilde{\theta}_2} \partial_{\kappa}^{\beta_1} X^{\Gamma}_{x_0} m^S_{J,l}\right) (x_0,(2\mu + \tilde{a} + \vec{a}(\omega)) |\lambda|, - \lambda) \\ 
\nonumber & \quad \quad \quad \quad \quad \quad \quad \quad \quad \quad \quad \quad \quad \Phi_{\mu + \tilde{\mu}}^{\lambda}(x^{\prime}) \Phi_{\mu}^{\lambda}(y^{\prime}) g(\omega) e^{- i \lambda \cdot (x^{\prime \prime} - y^{\prime \prime})} \, d \lambda \, d\omega,
\end{align*}
and $A^S_2(x, \beta_2, \theta_2, \tilde{\theta}_2)$ corresponds to the kernel
\begin{align*}
& \int_{[0,1]^{\nu_1} \times \Omega^{\nu_2} \times [0,1]^{4 N_1}} \int_{\mathbb{R}} \sum_{\mu} C_{\mu, \tilde{a}, \vec{a}} \left(\tau^{\frac{1}{2} \tilde{\theta}_2}  \partial_{\tau}^{\tilde{\theta}_2} \partial_{\kappa}^{\beta_2} m^S_{J^\prime,l^\prime}\right) (x_0,(2\mu + \tilde{a} + \vec{a}(\omega)) |\lambda|, - \lambda) \\ 
\nonumber & \quad \quad \quad \quad \quad \quad \quad \quad \quad \quad \quad \Phi_{\mu}^{\lambda}(x^{\prime}) \Phi_{\mu}^{\lambda}(y^{\prime}) g(\omega) e^{- i \lambda \cdot (x^{\prime \prime} - y^{\prime \prime})} \, d \lambda \, d\omega. 
\end{align*}

\medskip \noindent \underline{\textbf{Case $\boldsymbol{1}$ (when $\boldsymbol{S = 0}$)}} : For fixed $x_0\in \R^{n_1+n_2}$, $A^0_1(x_0, \alpha_1, \beta_1, \theta_1, \theta_2, \tilde{\theta}_2)$ can be viewed as a composition of the multiplication operator associated to the function $\frac{{x^{\prime}}^{\alpha_{1}}}{\left|x^\prime\right|^{4 N_2}}$ and a Grushin multiplier where the multiplier function is $\tau^{\frac{1}{2} \left( \theta_1 - \tilde{\theta}_2 \right)}  \partial_{\tau}^{\theta_2 - \tilde{\theta}_2} \partial_{\kappa}^{\beta_1}X^{\Gamma}_{x_0} m^0_{J,l}$. 

Since $|\alpha_1| \leq 4 N_2 - (|\beta|+l)$, we have  $\frac{\left| {x^{\prime}}^{\alpha_{1}} \right|}{\left|x^\prime\right|^{4 N_2}} \leq 1$ whenever $|x^\prime|>3$. Therefore, by Plancherel's theorem for the Grushin operator
\begin{align*} 
\| A^0_1(x_0, \alpha_1, \beta_1, \theta_1, \theta_2, \tilde{\theta}_2) \|_{op} & \lesssim \sup_{\tau,\kappa} \left| \left( \tau^{\frac{1}{2} \left( \theta_1 - \tilde{\theta}_2 \right)} \partial_{\tau}^{\theta_2 - \tilde{\theta}_2} X^{\Gamma}_{x_0} m^0_{J,l} \right) (x_0,\tau,\kappa) \right| \\ 
& \lesssim 2^{l | \theta_1 - \tilde{\theta}_2| /2} 2^{- l |{\theta}_2 - \tilde{\theta}_2| / 2} \|m\|^2_{S^0_{0, 0}} \\ 
& \leq \|m\|^2_{S^0_{0, 0}}, 
\end{align*} 
where we have used the condition $|\mathcal{\theta}_1| \leq |\mathcal{\theta}_2|$. 

On the other hand, note that $ \| \left( A^0_2(x, \beta_2, {\theta}_2, \tilde{\theta}_2) \right)^* \|_{op} = \| A^0_2(x, \beta_2, {\theta}_2, \tilde{\theta}_2) \|_{op},$ and the operator norm $ \| A^0_2(x, \beta_2, \tilde{\theta}, \tilde{\tilde{\theta}}) \|_{op}$ could be computed same as $ \| A^0_1(x_0, \alpha_1, \beta_1, \theta_1, \tilde{\theta},\tilde{\tilde{\theta}}) \|_{op}$ but this time using Lemma \ref{freeze-lem} instead of the Plancherel theorem. In this case we would get 
\begin{align*}
\| A^0_2(x, \beta_2, \theta_2, \tilde{\theta}_2) \|_{op} \lesssim \sup_{x, \tau,\kappa} \left| \tau^{\frac{1}{2} \tilde{\theta}_2} \left(\partial_{\tau}^{\tilde{\theta}_2} \partial_{\kappa}^{\beta_2} m^0_{J^\prime, l^\prime}\right)(x,\tau,\kappa)\right| \lesssim \|m\|^2_{S^0_{0, 0}}. 
\end{align*}
This completes the proof of the claimed estimate (\ref{TT*-kernel-claim1}) for $S=0$.

\medskip \noindent \underline{\textbf{Case $\boldsymbol{2}$ (when $\boldsymbol{S \geq 1}$)}} : 
Here we make use of the cancellation condition \eqref{def:grushin-symb-vanishing-0-condition} to obtain the claimed operator norm boundedness. Note that similar to how we did in the proof of Theorem \ref{thm:weighted-Plancherel-L2}, we can write 
\begin{align*} 
& \partial_{\tau}^{\theta_2 - \tilde{\theta}_2} \partial_{\kappa}^{\beta_1} X^{\Gamma}_{x_0} m^S_{J,l} (x_0,\tau,\kappa) \\ 
& = \sum_{\beta^\prime_1 \leq \beta_1} C_{\beta^\prime_1} \left\{ \sum_{j=0}^S \partial_{\kappa}^{\beta^\prime_1} \left( \psi_1 (2^j|\cdot|) \right) (\kappa) \right\} \partial_{\tau}^{\theta_2 - \tilde{\theta}_2} \partial_{\kappa}^{\beta_1 - \beta^\prime_1} X^{\Gamma}_{x_0} m_{J,l} (x_0,\tau,\kappa) \\ 
& = \sum_{\beta^\prime_1 \leq \beta_1} C_{\beta^\prime_1} \sum_{|\beta| = |\beta^\prime_1|} \frac{1}{\beta!} \left\{ \kappa^{\beta^\prime} \sum_{j=0}^S \partial_{\kappa}^{\beta^\prime_1} \left( \psi_1 (2^j|\cdot|) \right) (\kappa) \right\} \int_0^1 \partial_{\tau}^{\theta_2 - \tilde{\theta}_2} \partial_{\kappa}^{\beta_1 + \beta - \beta^\prime_1 } X^{\Gamma}_{x_0} m_{J,l} (x_0, \tau, t \kappa) \, dt, 
\end{align*}
where the last equality follows from the application of Taylor's theorem in $\kappa$-variable and then using the condition \eqref{def:grushin-symb-vanishing-0-condition}. 

It follows from the definition of $\psi_1$ that each one of $\kappa^{\beta^\prime_1} \sum_{j=0}^S \partial_{\kappa}^{\beta^\prime_1} \left( \psi_1 (2^j|\cdot|) \right) (\kappa)$ is a bounded function of $\kappa$-variable with bounds uniform in $S$. Now, using the above identity, we can essentially repeat the process of Case $1$ to estimate $L^2$-norms of $A^S_2(x, \beta_2, \theta_2, \tilde{\theta}_2)^*$ and $A^S_1(x_0, \alpha_1, \beta_1, \theta_1, \theta_2, \tilde{\theta}_2)$, and this completes the proof of Proposition \ref{prop:weighted-kernel-estimate-1}. 
\end{proof}

Before moving further, let us see how the proof of Corollary \ref{cor:Grushin-CV-compact-support} follows from the analysis developed so far. 

\begin{rem} \label{rem:proof-rem-Grushin-CV-compact-support}
Let $\mathcal{U} \subseteq \mathbb{R}^{n_2}$ be a closed set such that $0 \notin \mathcal{U}$, and recall that $L^2_{\mathcal{U}} (\mathbb{R}^{n_1 + n_2})$ denotes the collection of all functions $f \in L^2(\mathbb{R}^{n_1 + n_2})$ such that $f^\lambda(x')$ is supported in $\mathbb{R}^{n_1} \times \mathcal{U}$ as a function of $(x', \lambda)$. In the just concluded proof of Proposition \ref{prop:weighted-kernel-estimate-1}, while working out Case 2, if we are restricting functions $f$ to only come from the class $L^2_{\mathcal{U}} (\mathbb{R}^{n_1 + n_2})$, then the growth of the derivatives of various dilates $\psi_1(2^j \kappa)$ near $\kappa = 0$ can be controlled with the help of the assumption that $0 \notin \mathcal{U}$. That way we do need to invoke Taylor's theorem, and therefore no cancellation condition is required.
\end{rem}

\section{On \texorpdfstring{$L^2$}--boundedness in the case of \texorpdfstring{$S^0_{\rho, \delta}(\boldsymbol{L}, \boldsymbol{U}), \, 0 \leq \delta \leq \rho \leq 1, \, \delta \neq 1$}.} 
\label{sec:rho-delta} 

Following theorem is the main result of this section. As mentioned in the introduction, this theorem along with the results proved in the previous sections would complete the proof of Theorem \ref{thm:old-joint-calc}. 

\begin{thm} \label{thm:rho-delta} 
Let $0 \leq \tilde{\rho} \leq 1$ be such that $\left\| \tilde{m}(x, \boldsymbol{L}, \boldsymbol{U}) \right\|_{op} \lesssim_{\tilde{\rho}} \left\| \tilde{m} \right\|_{S^0_{\tilde{\rho}, 0}}$ for all $\tilde{m} \in S^0_{\tilde{\rho}, 0}(\boldsymbol{L}, \boldsymbol{U})$. Then for any pair $(\rho, \delta)$ such that $0 \leq \delta \leq \rho \leq 1$, $\delta \neq 1$, and $\tilde{\rho} = (\rho - \delta) / (1 - \delta)$, we have that $\left\| m(x, \boldsymbol{L}, \boldsymbol{U}) \right\|_{op} \lesssim_{\rho, \delta} \left\| m \right\|_{S^0_{\rho, \delta}}$ for all $m \in S^0_{\rho, \delta}(\boldsymbol{L}, \boldsymbol{U})$. 
\end{thm}

We decompose the operator $T = m(x, \boldsymbol{L}, \boldsymbol{U})$ into pieces similar to how we did in \eqref{0-symbol-decompose}. But, there are two notable differences. One is that we do not decompose the pieces in the space variable (that is, no index in $J \in \mathbb{N}$). Secondly, we make the decomposing in the joint spectrum, that is, jointly in $(\tau, \kappa)$ (see remark \ref{rem:cut-off-spectral-variable-comparison}). 

With $\psi_0$ as earlier (just above \eqref{0-symbol-decompose}), we define for all $S \in \mathbb{N}$,  
\begin{align} \label{0-symbol-decompose-joint-final}
m^S_{l}(x, \tau, \kappa) := m(x, \tau, \kappa) \psi_l(|(\tau, \kappa)|_1) \sum_{j = - \infty}^S \psi_1 (2^j |\kappa|).
\end{align}

With $T^S = \sum_{l \in \mathbb{N}} m_l^S(x, \boldsymbol{L}, \boldsymbol{U})$, it suffices to show that 
$$ \sup_{S \in \mathbb{N}} \|T^S\|_{op} < \infty. $$ 
In order to do so, we need to first check the boundedness of each of the operators $T^S_l$. 

\begin{lem}\label{rho-rho-case-unif-bound}
Let $\rho, \delta$ and $\tilde{\rho}$ satisfy the conditions stated in Theorem \ref{thm:rho-delta}. Then corresponding to every $m \in S^0_{\rho, \delta}(G)$, the operators $T^S_l = m^S_l (x,G)$ are bounded on $L^2(\mathbb{R}^{n_1 + n_2})$ with their operator norms being uniformly bounded in $l$ and $S$. 
\end{lem}
\begin{proof}
For $t>0$, consider the dilates of functions $\delta_t f(x) := f(\delta_t x)$ with respect to the non-isotropic dilations $\delta_t (x^{\prime}, x^{\prime \prime}) = (t x^{\prime}, t^2 x^{\prime \prime})$, and define the family of operators $T_{l, t}$ and $T^S_{l, t}$ respectively by 
$$T_{l, t} := \delta_{t^{-1}} T_l \delta_t \quad \text{and} \quad T^S_{l, t} := \delta_{t^{-1}} T^S_l \delta_t.$$ 

Note that 
\begin{align*} 
T^S_l \delta_t f(x) &= \int_{\mathbb{R}^{n_2}} \sum_{\mu} m^S_l(x, (2 \mu + \tilde{1}) |\lambda|, -\lambda) \left((\delta_t f)^\lambda, \Phi_{\mu}^\lambda \right) \Phi_{\mu}^\lambda(x^\prime) e^{-i \lambda \cdot x^{\prime \prime}} \, d\lambda \\
&= (2\pi)^{n_2} t^{-2 n_2} \int_{\mathbb{R}^{n_2}} \sum_{\mu} m^S_l(x, (2 \mu + \tilde{1}) |\lambda|, -\lambda) \left(f^{t^{-2} \lambda}(s \cdot), \Phi_{\mu}^\lambda \right) \Phi_{\mu}^\lambda(x^\prime) e^{-i \lambda \cdot x^{\prime \prime}} \, d\lambda, 
\end{align*}
and therefore 
\begin{align*} 
& (2\pi)^{n_2} T^S_{l, t} = (2\pi)^{n_2} (2\pi)^{n_2} \delta_{t^{-1}} T^S_l \delta_t f(x) = (2\pi)^{n_2} T^S_l \delta_t f (t^{-1} x^{\prime}, t^{-2} x^{\prime \prime}) \\
&= t^{-2 n_2} \int_{\mathbb{R}^{n_2}} \sum_{\mu} m^S_l(t^{-1} x^{\prime}, t^{-2} x^{\prime \prime}, (2 \mu + \tilde{1}) |\lambda|, -\lambda) \left(f^{t^{-2} \lambda}(s \cdot), \Phi_{\mu}^\lambda \right) \Phi_{\mu}^\lambda (t^{-1} x^\prime) e^{-i t^{-2} \lambda \cdot x^{\prime \prime}} \, d\lambda \\
&= \int_{\mathbb{R}^{n_2}} \sum_{\mu} m^S_l(\delta_{t^{-1}} x, t^2 (2 \mu + \tilde{1}) |\lambda|, -t^2\lambda)  \left(f^{\lambda}(s \cdot), \Phi_{\mu}^{t^2 \lambda} \right) \Phi_{\mu}^{t^2 \lambda} (t^{-1} x^\prime) e^{-i \lambda \cdot x^{\prime \prime}} \, d\lambda \\
&= \int_{\mathbb{R}^{n_2}} \sum_{\mu} m^S_l(\delta_{t^{-1}} x, t^2 (2 \mu + \tilde{1}) |\lambda|, -t^2\lambda)  \left(f^{\lambda}, \Phi_{\mu}^{\lambda} \right) \Phi_{\mu}^{\lambda}(x^\prime) e^{-i \lambda \cdot x^{\prime \prime}} \, d\lambda \\ 
&= \int_{\mathbb{R}^{n_2}} \sum_{\mu} m_l(\delta_{t^{-1}} x, t^2 (2 \mu + \tilde{1}) |\lambda|, -t^2\lambda) \, \zeta^S (t^2 \lambda) \, \left(f^{\lambda}, \Phi_{\mu}^{\lambda} \right) \Phi_{\mu}^{\lambda}(x^\prime) e^{-i \lambda \cdot x^{\prime \prime}} \, d\lambda.
\end{align*}

From the above expression we get that $T^S_{l, t} = m_{l, t} (x, \boldsymbol{L}, \boldsymbol{U}) \circ \zeta^S (t^2 \boldsymbol{U})$, where 
$$m_{l, t} (x, \tau, \kappa) := m_l(\delta_{t^{-1}} x, t^2 \tau, t^2 \kappa) = m_l(t^{-1} x^{\prime}, t^{-2} x^{\prime \prime}, t^2\tau, t^2\kappa).$$

Since $\|\zeta^S (\cdot)\|_{\infty}$ is uniform in $S$, it follows from the Plancherel's theorem that the operator norm $\|\zeta^S (t^2 \boldsymbol{U})\|_{op}$ is uniform in $S$ and $t$. We are therefore left with analysing the operator norm  $\|m_{l, t} (x, \boldsymbol{L}, \boldsymbol{U})\|_{op}$.

Choosing $t=2^{l \delta/2}$, one can verify that 
$$\|m_{l, 2^{l \delta/2}} \|_{\mathcal{S}^0_{ \frac{\rho - \delta}{1 - \delta}, 0}} \lesssim \|m\|_{\mathcal{S}^0_{\rho, \delta}}$$
uniformly in $l  \in \mathbb{N}$. To see this, let us just verify one derivative each in space and frequency variable as the computations for the higher order derivatives are analogous. 

Let $l \geq 1$. In view of the support condition of $m_l$ we have that $m_{l, 2^{l \delta/2}}$ (and hence it's derivatives) will be non-zero only if $t^2 |(\tau, \kappa)| \sim t^2 |(\tau, \kappa)|_1 = 2^{l \delta} |(\tau, \kappa)|_1 \sim 2^l$, which is equivalent to $|(\tau, \kappa)| \sim 2^{l (1 - \delta)}$ or $2^l \sim |(\tau, \kappa)|^{1/(1 - \delta)}$. 

Now, 
\begin{align*}
\left| \partial_{x^\prime_k} m_{l, 2^{l \delta/2}} (x, \tau, \kappa) \right| 
& = 2^{-l \delta/2} \left| \left( \partial_{x^\prime_k} m_{l} \right) (2^{-l \delta/2} x^{\prime}, 2^{-l \delta} x^{\prime \prime}, 2^{l \delta} \tau, 2^{l \delta} \kappa) \right| \\  
& \lesssim 2^{-l \delta/2} \left( 2^{l \delta} |(\tau, \kappa)| \right)^{ \delta / 2} \sim 2^{-l \delta/2} \left( 2^l \right)^{ \delta / 2} = 1. 
\end{align*} 

Next, 
\begin{align*}
\left| \partial_{\tau_j} m_{l, 2^{l \delta/2}} (x, \tau, \kappa) \right| &= 2^{l \delta} \left| \left( \partial_{\tau_j} m_{l} \right) (2^{-l \delta/2} x^{\prime}, 2^{-l \delta} x^{\prime \prime}, 2^{l \delta} \tau, 2^{l \delta} \kappa) \right| \\ 
& \lesssim 2^{l \delta} \left( 2^{l \delta} |(\tau, \kappa)| \right)^{- (1 + \rho) / 2} \\ 
& = \left( 2^l \right)^{ \delta (1 - \rho) / 2} |(\tau, \kappa)|^{- (1 + \rho) / 2} \\ 
& \sim \left( |(\tau, \kappa)|^{\frac{1}{(1 - \delta)}} \right)^{ \delta (1 - \rho) / 2} |(\tau, \kappa)|^{- (1 + \rho) / 2} \\ 
& = |(\tau, \kappa)|^{- \left( 1 + \frac{\rho - \delta}{1 - \delta} \right) / 2} \sim (1 + |(\tau, \kappa)|)^{- \left( 1 + \frac{\rho - \delta}{1 - \delta} \right) / 2}. 
\end{align*}
Similarly, one can show that $\left| \partial_{\tau_j} m_{l, 2^{l \delta/2}} (x, \tau, \kappa) \right| \lesssim (1 + |(\tau, \kappa)|)^{- \left( 1 + \frac{\rho - \delta}{1 - \delta} \right) / 2}$. 

This completes the proof of the claim that $\|m_{l, 2^{l \delta/2}} \|_{\mathcal{S}^0_{ \frac{\rho - \delta}{1 - \delta}, 0}} \lesssim \|m\|_{\mathcal{S}^0_{\rho, \delta}}$ uniformly in $l \in \mathbb{N}$, $l \geq 1$. The proof for $l=0$ can be proved analogously, utilizing the fact that the support condition for $m_0$ implies that $1 \leq 1 + |(\tau, \kappa)| \leq 3$. In this case, $s=1$ is sufficient.

Hence the lemma follows from the assumption that the operators corresponding to symbols from the class $\mathcal{S}^0_{ \frac{\rho - \delta}{1 - \delta}, 0}$ are $L^2$-bounded. 
\end{proof}

We now return to the proof of Theorem \ref{thm:rho-delta}. For a large enough natural number $\mathfrak{z}$ that will be chosen at the very end of this section, we write $T^S = \sum\limits_{I = 1}^{\mathfrak{z}} T^{S, I},$ where 
$$T^{S, I} = \sum\limits_{l \equiv I (mod \, \mathfrak{z})} T^S_l.$$ 
It suffices to show that for each $1 \leq I \leq \mathfrak{z}$, 
$$ \sup_{S \in \mathbb{N}} \| T^{S, I} \|_{op} < \infty. $$ 

Clearly, for each fixed $I$, for any distinct $l, l^{\prime} \equiv I (mod \, \mathfrak{z})$ we have 
$$T^S_l \left( T^S_{l^{\prime}} \right)^* = T \circ \psi_l (G) \circ \psi_{l^{\prime}} (G) \circ \left( T^S \right)^* = 0.$$  

For the sake of convenience, we drop the notation $I$ from now on wards. We prove the result using the following crude version of the Cotlar-Stein Lemma (see Section 2.3 of Chapter VII in \cite{SteinHarmonicBook93})): 

\begin{thm}[Cotlar-Stein lemma-2] \label{Cotlar-Stein-lemma-2}
Let $\mathcal{H}$ be a Hilbert space. Let $\{T_j\}_{j \in \mathbb{Z}} \subset \mathcal{B}(\mathcal{H})$ be such that 
$T_j T^*_k = 0$ for all $j \neq k$, and $\sup_{j} \left\| T_j \right\|_{\mathcal{B}(\mathcal{H})} = B < \infty$. If there exists a sequence of positive constants $\{c(j)\}_{j \in \mathbb{Z}}$ with 
$$A = \sum_{j \in \mathbb{Z}} c(j) < \infty,$$ 
such that the operators $T_j$ satisfy the estimate 
$$ \left\| T^*_j T_k \right\|_{\mathcal{B}(\mathcal{H})} \leq c(j) c(k), \textup{ for all } j \neq k, $$
then $\sum_{j \in \mathbb{Z}} T_j$ converges in the strong operator topology to some $T \in \mathcal{B}(\mathcal{H})$ and 
$$\| T \|_{\mathcal{B}(\mathcal{H})} \leq 2^{1/2} \max \left\{A, B \right\}.$$
\end{thm} 

In view of Theorem \ref{Cotlar-Stein-lemma-2}, we shall be done if we could show that 
$$ \| \left( T^S_l \right)^* T^S_{l^{\prime}} \|_{op} \lesssim c(l) c(l^{\prime}), \textup{ for all } l \neq l^{\prime},$$
for some $c(l) \geq 0$ satisfying $\sum_l c(l) < \infty.$

Since $\| \left( T^S_l \right)^* T^S_{l^{\prime}} \|_{op} = \| \left( T^S_{l^{\prime}} \right)^* T^S_l \|_{op}$, we can assume without loss of generality that $l > l^{\prime}$. We shall analyse the kernel $K^S_{l, l^{\prime}}$ of $T_l^* T_{l^{\prime}}$. In order to do so, we need the following estimate. 

\begin{lem} \label{lem:rho-delta-frequency-decay}
Let $l > l^{\prime}$. For each $N \in \mathbb{N}$ such that $N \leq 2^{\mathfrak{z}-5}$, we can write the kernel $K^S_{l, l^\prime}(x,y)$ of the operator $\left( T^S_l \right)^* T^S_{l^{\prime}}$ as a finite linear combination of terms which are of the form 
\begin{align} \label{eq:1-p1}
& \int_{\mathbb{R}^{n_1 + n_2} \times \mathbb{R}^{n_2} \times \mathbb{R}^{n_2}} \Theta_0(\lambda^{\prime}) \sum_{\mu, {\mu^{\prime}}} \mathcal{A}(\mu^{\prime}) m^{S, \Gamma_1, \Gamma_2, L}_{l, l^{\prime}}(z, \mu, {\mu^{\prime}} + \vec{c}, \lambda, \lambda^{\prime}) \Phi_\mu^{\lambda}(z^{\prime}) \Phi_{{\mu^{\prime}}}^{\lambda^{\prime}}(z^{\prime}) \\ 
\nonumber & \quad \quad \quad \quad \quad \quad \quad \quad \quad \quad \quad \Phi_\mu^{\lambda}(x^{\prime}) \Phi_{{\mu^{\prime}} + \vec{c}}^{\lambda^{\prime}}(y^{\prime}) e^{i \lambda \cdot (z^{\prime \prime} - x^{\prime \prime})} e^{- i \lambda^{\prime} \cdot (z^{\prime \prime} - y^{\prime \prime})} \, d\lambda \, d\lambda^{\prime} \, dz,
\end{align} 
where 
$$ m^{S, \Gamma_1, \Gamma_2, L}_{l, l^{\prime}}(z, \mu, {\mu^{\prime}}, \lambda, \lambda^{\prime}) = \frac{\overline{X^{\Gamma_1} m^S_l(z, (2 \mu + \tilde{1})|\lambda|, - \lambda)} \, X^{\Gamma_2} m^S_{l^\prime}(z, (2 \mu^{\prime} + \tilde{1}) |\lambda^{\prime}|, - \lambda^\prime) P_{L}({\mu^{\prime}}, \lambda^{\prime})} {\prod_{q=1}^N \left\{(2|\mu|+n_1)|\lambda| - (2|{\mu^{\prime}}| + C_q + n_1)|\lambda^{\prime}| \right\} } $$
with $|\Gamma_1| + |\Gamma_2| + L = 2N$ and $L \leq N$. Here, $P_{L}({\mu^{\prime}}, \lambda^{\prime}) = ((2|{\mu^{\prime}}| + n_1)|\lambda^{\prime}|)^{L/2}$, with $C_q$'s being integers bounded by $2q$, the vectors $\vec{c} \in \mathbb{Z}^{n_1}$ may depend only on $\Gamma_1, \Gamma_2$, and $N$. Finally, $\Theta_0$ is a continuous function on $\mathbb{R}^{n_2} \setminus \{0\}$ which is homogeneous of degree zero, and $\mathcal{A}(\mu^{\prime})$ is bounded function of $\mu^{\prime}$, with both functions $\Theta_0$ and $\mathcal{A}$ depending only on $N$. 
\end{lem}

\begin{proof} Note that $l - l^\prime \geq \mathfrak{z}$. Now, if $ (2|\mu^{\prime}| + n_1) |\lambda^{\prime}| + |\lambda^\prime|_1 \in [2^{l^{\prime} - 1}, 2^{l^{\prime} + 1}]$, then $|\lambda^{\prime}| \leq 2^{l^{\prime} + 1}$, which implies that $ (2 | {\mu^{\prime}}| + 2 N + n_1) |\lambda^{\prime}| \leq (2N + 1) 2^{l^{\prime} + 1}$. Also, observe that $ (2|\mu|+n_1)|\lambda| + |\lambda|_1 \in [2^{l-1}, 2^{l+1}]$ implies that $(2|\mu|+n_1)|\lambda| \geq C_{n_2} 2^{l-1}$. 

 Therefore, 
\begin{align*}
\left| (2|\mu|+n_1)|\lambda| - (2 | {\mu^{\prime}}| \pm 2 N + n)|\lambda^{\prime}| \right| & \geq C_{n_2} 2^{l-1} - (2N + 1) 2^{l^{\prime} + 1} \\ 
& \geq 2^{l-1} \left[ C_{n_2} - (2N + 1) 2^{2 - \mathfrak{z}} \right] \gtrsim 2^l
\end{align*}
provided that $N \leq C_{n_2} 2^{\mathfrak{z}-5}$. 

Writing $m^S_l(z, \mu, \lambda) = m^S_l(z, (2 \mu + \tilde{1}) |\lambda|, - \lambda)$ we have
\begin{align*}
& (2\pi)^{2 n_2} K^S_{l, l^\prime}(x,y) \\ 
& = \int_{\mathbb{R}^{n_1 + n_2} \times \mathbb{R}^{n_2} \times \mathbb{R}^{n_2}} \sum_{\mu, {\mu^{\prime}}} \overline{m^S_l(z, \mu, \lambda)} \, m^S_{l^\prime}(z, {\mu^{\prime}}, \lambda^{\prime}) \overline{E_{\mu, \lambda} (z, x)} E_{\mu^{\prime}, \lambda^{\prime}} (z, y) \, d\lambda \, d\lambda^{\prime} \, dz \\ 
& = \int_{\mathbb{R}^{n_1 + n_2} \times \mathbb{R}^{n_2} \times \mathbb{R}^{n_2}} \sum_{\mu, {\mu^{\prime}}} \frac{\left( (2|\mu|+n_1)|\lambda| - (2|{\mu^{\prime}}| + n_1)|\lambda^{\prime}| \right)}{\left( (2|\mu|+n_1)|\lambda| - (2|{\mu^{\prime}}| + n_1)|\lambda^{\prime}| \right)} \overline{m^S_l(z, \mu, \lambda)} \, m^S_{l^\prime}(z, {\mu^{\prime}}, \lambda^{\prime}) \\ 
& \quad \quad \quad \quad \quad \quad \quad \quad \quad \quad \overline{E_{\mu, \lambda} (z, x)} E_{\mu^{\prime}, \lambda^{\prime}} (z, y) \, d\lambda \, d\lambda^{\prime} \, dz\\
&= \int_{\mathbb{R}^{n_1 + n_2} \times \mathbb{R}^{n_2} \times \mathbb{R}^{n_2}} \sum_{\mu, {\mu^{\prime}}} \frac{\overline{m^S_l(z, \mu, \lambda)} \, m^S_{l^\prime}(z, {\mu^{\prime}}, \lambda^{\prime})}{\left( (2|\mu|+n_1)|\lambda| - (2|{\mu^{\prime}}| + n_1)|\lambda^{\prime}| \right)}  \left(H(\lambda) \Phi_\mu^{\lambda}(z^{\prime})\right) \Phi_{\mu^{\prime}}^{\lambda^{\prime}}(z^{\prime}) \\
& \quad \quad \quad \quad \quad \quad \quad \quad \quad \quad \Phi_\mu^{\lambda}(x^{\prime}) \Phi_{\mu^{\prime}}^{\lambda^{\prime}}(y^{\prime}) e^{i \lambda \cdot (z^{\prime \prime} - x^{\prime \prime})} e^{- i \lambda^{\prime} \cdot (z^{\prime \prime} - y^{\prime \prime})} \, d\lambda \, d\lambda^{\prime} \, dz \\
& \quad- \int_{\mathbb{R}^{n_1 + n_2} \times \mathbb{R}^{n_2} \times \mathbb{R}^{n_2}} \sum_{\mu, {\mu^{\prime}}} \frac{\overline{m^S_l(z, \mu, \lambda)} \, m^S_{l^\prime}(z, {\mu^{\prime}}, \lambda^{\prime})}{\left( (2|\mu|+n_1)|\lambda| - (2|{\mu^{\prime}}| + n_1)|\lambda^{\prime}| \right)} \Phi_\mu^{\lambda}(z^{\prime}) \left(H(\lambda^{\prime}) \Phi_{\mu^{\prime}}^{\lambda^{\prime}}(z^{\prime})\right) \\
& \quad \quad \quad \quad \quad \quad \quad \quad \quad \quad \quad \Phi_\mu^{\lambda}(x^{\prime}) \Phi_{\mu^{\prime}}^{\lambda^{\prime}}(y^{\prime}) e^{i \lambda \cdot (z^{\prime \prime} - x^{\prime \prime})} e^{- i \lambda^{\prime} \cdot (z^{\prime \prime} - y^{\prime \prime})} \, d\lambda \, d\lambda^{\prime} \, dz \\
&= - \int_{\mathbb{R}^{n_1 + n_2} \times \mathbb{R}^{n_2} \times \mathbb{R}^{n_2}} \sum_{\mu, {\mu^{\prime}}} \frac{\overline{m^S_l(z, \mu, \lambda)} \, m^S_{l^\prime}(z, {\mu^{\prime}}, \lambda^{\prime})}{\left( (2|\mu|+n_1)|\lambda| - (2|{\mu^{\prime}}| + n_1)|\lambda^{\prime}| \right)} \left( \Delta_{z^{\prime}} \Phi_\mu^{\lambda}(z^{\prime})\right) \Phi_{\mu^{\prime}}^{\lambda^{\prime}}(z^{\prime}) \\
& \quad \quad \quad \quad \quad \quad \quad \quad \quad \quad \Phi_\mu^{\lambda}(x^{\prime}) \Phi_{\mu^{\prime}}^{\lambda^{\prime}}(y^{\prime}) e^{i \lambda \cdot (z^{\prime \prime} - x^{\prime \prime})} e^{- i \lambda^{\prime} \cdot (z^{\prime \prime} - y^{\prime \prime})} \, d\lambda \, d\lambda^{\prime} \, dz \\
& \quad + \int_{\mathbb{R}^{n_1 + n_2} \times \mathbb{R}^{n_2} \times \mathbb{R}^{n_2}} \sum_{\mu, {\mu^{\prime}}} \frac{\overline{m^S_l(z, \mu, \lambda)} \, m^S_{l^\prime}(z, {\mu^{\prime}}, \lambda^{\prime})}{\left( (2|\mu|+n_1)|\lambda| - (2|{\mu^{\prime}}| + n_1)|\lambda^{\prime}| \right)} |\lambda|^2 |z^{\prime}|^2 \Phi_\mu^{\lambda}(z^{\prime}) \Phi_{\mu^{\prime}}^{\lambda^{\prime}}(z^{\prime}) \\
& \quad \quad \quad \quad \quad \quad \quad \quad \quad \quad \Phi_\mu^{\lambda}(x^{\prime}) \Phi_{\mu^{\prime}}^{\lambda^{\prime}}(y^{\prime}) e^{i \lambda \cdot (z^{\prime \prime} - x^{\prime \prime})} e^{- i \lambda^{\prime} \cdot (z^{\prime \prime} - y^{\prime \prime})} \, d\lambda \, d\lambda^{\prime} \, dz \\
& \quad- \int_{\mathbb{R}^{n_1 + n_2} \times \mathbb{R}^{n_2} \times \mathbb{R}^{n_2}} \sum_{\mu, {\mu^{\prime}}} \frac{\overline{m^S_l(z, \mu, \lambda)} \, m^S_{l^\prime}(z, {\mu^{\prime}}, \lambda^{\prime})}{\left( (2|\mu|+n_1)|\lambda| - (2|{\mu^{\prime}}| + n_1)|\lambda^{\prime}| \right)} \Phi_\mu^{\lambda}(z^{\prime}) \\
& \quad \quad \quad \quad \quad \left\{ \left( - \Delta_{z^{\prime}} + |\lambda^{\prime}|^2 |z^{\prime}|^2 \right) \Phi_{\mu^{\prime}}^{\lambda^{\prime}}(z^{\prime}) \right\} \Phi_\mu^{\lambda}(x^{\prime}) \Phi_{\mu^{\prime}}^{\lambda^{\prime}}(y^{\prime}) e^{i \lambda \cdot (z^{\prime \prime} - x^{\prime \prime})} e^{- i \lambda^{\prime} \cdot (z^{\prime \prime} - y^{\prime \prime})} \, d\lambda \, d\lambda^{\prime} \, dz \\
&=: E_1 + E_2 + E_3.
\end{align*}

We shall perform integration by parts in $z^{\prime}$-variable in $E_1$ and in $z^{\prime \prime}$-variable in $E_2$ and that would lead to the cancellation of $E_3$. We start with the term $E_1$. 
\begin{align} \label{eq:E1E2E3-calc1}
E_1 & = - \int_{\mathbb{R}^{n_1 + n_2} \times \mathbb{R}^{n_2} \times \mathbb{R}^{n_2}} \sum_{\mu, {\mu^{\prime}}} \frac{\Delta_{z^{\prime}} \left(\overline{m^S_l(z, \mu, \lambda)} \, m^S_{l^\prime}(z, {\mu^{\prime}}, \lambda^{\prime})\right)}{\left( (2|\mu|+n_1)|\lambda| - (2|{\mu^{\prime}}| + n_1)|\lambda^{\prime}| \right)} \Phi_\mu^{\lambda}(z^{\prime}) \Phi_{\mu^{\prime}}^{\lambda^{\prime}}(z^{\prime})\\
\nonumber & \quad \quad \quad \quad \quad \quad \quad \quad \quad \quad \quad \Phi_\mu^{\lambda}(x^{\prime}) \Phi_{\mu^{\prime}}^{\lambda^{\prime}}(y^{\prime}) e^{i \lambda \cdot (z^{\prime \prime} - x^{\prime \prime})} e^{- i \lambda^{\prime} \cdot (z^{\prime \prime} - y^{\prime \prime})} \, d\lambda \, d\lambda^{\prime} \, dz \\
\nonumber & \quad - 2 \int_{\mathbb{R}^{n_1 + n_2} \times \mathbb{R}^{n_2} \times \mathbb{R}^{n_2}} \sum_{\mu, {\mu^{\prime}}} \frac{\nabla_{z^{\prime}} \left( \overline{m^S_l(z, \mu, \lambda)} \, m^S_{l^\prime}(z, {\mu^{\prime}}, \lambda^{\prime})\right) \cdot  \nabla_{z^{\prime}} \left( \Phi_{\mu^{\prime}}^{\lambda^{\prime}}(z^{\prime}) \right)} {\left( (2|\mu|+n_1)|\lambda| - (2|{\mu^{\prime}}| + n_1)|\lambda^{\prime}| \right)} \Phi_\mu^{\lambda}(z^{\prime}) \\
\nonumber &  \quad \quad \quad \quad \quad \quad \quad \quad \quad \quad \quad \quad \Phi_\mu^{\lambda}(x^{\prime}) \Phi_{\mu^{\prime}}^{\lambda^{\prime}}(y^{\prime}) e^{i \lambda \cdot (z^{\prime \prime} - x^{\prime \prime})} e^{- i \lambda^{\prime} \cdot (z^{\prime \prime} - y^{\prime \prime})} \, d\lambda \, d\lambda^{\prime} \, dz \\
\nonumber &  \quad - \int_{\mathbb{R}^{n_1 + n_2} \times \mathbb{R}^{n_2} \times \mathbb{R}^{n_2}} \sum_{\mu, {\mu^{\prime}}} \frac{\overline{m^S_l(z, \mu, \lambda)} \, m^S_{l^\prime}(z, {\mu^{\prime}}, \lambda^{\prime})}{\left( (2|\mu|+n_1)|\lambda| - (2|{\mu^{\prime}}| + n_1)|\lambda^{\prime}| \right)} \Phi_\mu^{\lambda}(z^{\prime}) \left(\Delta_{z^{\prime}} \Phi_{\mu^{\prime}}^{\lambda^{\prime}}(z^{\prime})\right) \\
\nonumber &  \quad \quad \quad \quad \quad \quad \quad \quad \quad \quad \quad \Phi_\mu^{\lambda}(x^{\prime}) \Phi_{\mu^{\prime}}^{\lambda^{\prime}}(y^{\prime}) e^{i \lambda \cdot (z^{\prime \prime} - x^{\prime \prime})} e^{- i \lambda^{\prime} \cdot (z^{\prime \prime} - y^{\prime \prime})} \, d\lambda \, d\lambda^{\prime} \, dz.
\end{align}
 
Next, 
\begin{align} \label{eq:E1E2E3-calc2}
E_2 & = - \int_{\mathbb{R}^{n_1 + n_2} \times \mathbb{R}^{n_2} \times \mathbb{R}^{n_2}} \sum_{\mu, {\mu^{\prime}}} \frac{\overline{m^S_l(z, \mu, \lambda)} \, m^S_{l^\prime}(z, {\mu^{\prime}}, \lambda^{\prime})}{\left( (2|\mu|+n_1)|\lambda| - (2|{\mu^{\prime}}| + n_1)|\lambda^{\prime}| \right)} |z^{\prime}|^2 \Phi_\mu^{\lambda}(z^{\prime}) \Phi_{\mu^{\prime}}^{\lambda^{\prime}}(z^{\prime}) \\
\nonumber &  \quad \quad \quad \quad \quad \quad \quad \quad \quad \quad \Phi_\mu^{\lambda}(x^{\prime}) \Phi_{\mu^{\prime}}^{\lambda^{\prime}}(y^{\prime}) \Delta_{z^{\prime \prime}} \left( e^{i \lambda \cdot (z^{\prime \prime} - x^{\prime \prime})} \right) e^{- i \lambda^{\prime} \cdot (z^{\prime \prime} - y^{\prime \prime})} \, d\lambda \, d\lambda^{\prime} \, dz \\ 
\nonumber &  = - \int_{\mathbb{R}^{n_1 + n_2} \times \mathbb{R}^{n_2} \times \mathbb{R}^{n_2}} \sum_{\mu, {\mu^{\prime}}} \frac{\Delta_{z^{\prime \prime}} \left(\overline{m^S_l(z, \mu, \lambda)} \, m^S_{l^\prime}(z, {\mu^{\prime}}, \lambda^{\prime})\right)}{\left( (2|\mu|+n_1)|\lambda| - (2|{\mu^{\prime}}| + n_1)|\lambda^{\prime}| \right)} |z^{\prime}|^2 \Phi_\mu^{\lambda}(z^{\prime}) \Phi_{\mu^{\prime}}^{\lambda^{\prime}}(z^{\prime}) \\
\nonumber &  \quad \quad \quad \quad \quad \quad \quad \quad \quad \quad \Phi_\mu^{\lambda}(x^{\prime}) \Phi_{\mu^{\prime}}^{\lambda^{\prime}}(y^{\prime}) e^{i \lambda \cdot (z^{\prime \prime} - x^{\prime \prime})} e^{- i \lambda^{\prime} \cdot (z^{\prime \prime} - y^{\prime \prime})} \, d\lambda \, d\lambda^{\prime} \, dz \\ 
\nonumber &  \quad - 2 \int_{\mathbb{R}^{n_1 + n_2} \times \mathbb{R}^{n_2} \times \mathbb{R}^{n_2}} \sum_{\mu, {\mu^{\prime}}} \frac{\nabla_{z^{\prime \prime}} \left( \overline{m^S_l(z, \mu, \lambda)} \, m^S_{l^\prime}(z, {\mu^{\prime}}, \lambda^{\prime})\right) \cdot \nabla_{z^{\prime \prime}}  \left( e^{i z^{\prime \prime} \cdot \lambda^{\prime}} \right)} {\left( (2|\mu|+n_1)|\lambda| - (2|{\mu^{\prime}}| + n_1)|\lambda^{\prime}| \right)} \\
\nonumber &  \quad \quad \quad \quad \quad \quad \quad \quad \quad \quad \quad |z^{\prime}|^2 \Phi_\mu^{\lambda}(z^{\prime}) \Phi_{\mu^{\prime}}^{\lambda^{\prime}}(z^{\prime})  \Phi_\mu^{\lambda}(x^{\prime}) \Phi_{\mu^{\prime}}^{\lambda^{\prime}}(y^{\prime}) e^{i \lambda \cdot (z^{\prime \prime} - x^{\prime \prime})} e^{- i \lambda^{\prime} \cdot (z^{\prime \prime} - y^{\prime \prime})} \, d\lambda \, d\lambda^{\prime} \, dz \\
\nonumber &  \quad - \int_{\mathbb{R}^{n_1 + n_2} \times \mathbb{R}^{n_2} \times \mathbb{R}^{n_2}} \sum_{\mu, {\mu^{\prime}}} \frac{\overline{m^S_l(z, \mu, \lambda)} \, m^S_{l^\prime}(z, {\mu^{\prime}}, \lambda^{\prime})}{\left( (2|\mu|+n_1)|\lambda| - (2|{\mu^{\prime}}| + n_1)|\lambda^{\prime}| \right)} |z^{\prime}|^2 \Phi_\mu^{\lambda}(z^{\prime}) \Phi_{\mu^{\prime}}^{\lambda^{\prime}}(z^{\prime}) \\
\nonumber &  \quad \quad \quad \quad \quad \quad \quad \quad \quad \quad \Phi_\mu^{\lambda}(x^{\prime}) \Phi_{\mu^{\prime}}^{\lambda^{\prime}}(y^{\prime}) e^{i \lambda \cdot (z^{\prime \prime} - x^{\prime \prime})} \Delta_{z^{\prime \prime}} \left( e^{- i \lambda^{\prime} \cdot (z^{\prime \prime} - y^{\prime \prime})} \right) d\lambda \, d\lambda^{\prime} \, d\lambda \, d\lambda^{\prime} \, dz. 
\end{align}

Note that the sum of the final terms from the right hand side of \eqref{eq:E1E2E3-calc1} and \eqref{eq:E1E2E3-calc2} is equal to $- E_3$. Therefore,  
\begin{align*} 
& (2\pi)^{2 n_2} K_{l, l^\prime}(x,y) = E_1 + E_2 + E_3 \\ 
& = - \int_{\mathbb{R}^{n_1 + n_2} \times \mathbb{R}^{n_2} \times \mathbb{R}^{n_2}} \sum_{\mu, {\mu^{\prime}}} \frac{\Delta_{z^{\prime}} \left(\overline{m^S_l(z, \mu, \lambda)} \, m^S_{l^\prime}(z, {\mu^{\prime}}, \lambda^{\prime})\right)}{\left( (2|\mu|+n_1)|\lambda| - (2|{\mu^{\prime}}| + n_1)|\lambda^{\prime}| \right)} \Phi_\mu^{\lambda}(z^{\prime}) \Phi_{\mu^{\prime}}^{\lambda^{\prime}}(z^{\prime})\\
& \quad \quad \quad \quad \quad \quad \quad \quad \quad \quad \quad \Phi_\mu^{\lambda}(x^{\prime}) \Phi_{\mu^{\prime}}^{\lambda^{\prime}}(y^{\prime}) e^{i \lambda \cdot (z^{\prime \prime} - x^{\prime \prime})} e^{- i \lambda^{\prime} \cdot (z^{\prime \prime} - y^{\prime \prime})} \, d\lambda \, d\lambda^{\prime} \, dz \\
&\quad - 2 \int_{\mathbb{R}^{n_1 + n_2} \times \mathbb{R}^{n_2} \times \mathbb{R}^{n_2}} \sum_{\mu, {\mu^{\prime}}} \frac{\nabla_{z^{\prime}} \left(\overline{m^S_l(z, \mu, \lambda)} \, m^S_{l^\prime}(z, {\mu^{\prime}}, \lambda^{\prime})\right) \cdot \nabla_{z^{\prime}} \left( \Phi_{\mu^{\prime}}^{\lambda^{\prime}}(z^{\prime}) \right)} {\left( (2|\mu|+n_1)|\lambda| - (2|{\mu^{\prime}}| + n_1)|\lambda^{\prime}| \right)} \Phi_\mu^{\lambda}(z^{\prime}) \\
& \quad \quad \quad \quad \quad \quad \quad \quad \quad \quad \quad \Phi_\mu^{\lambda}(x^{\prime}) \Phi_{\mu^{\prime}}^{\lambda^{\prime}}(y^{\prime}) e^{i \lambda \cdot (z^{\prime \prime} - x^{\prime \prime})} e^{- i \lambda^{\prime} \cdot (z^{\prime \prime} - y^{\prime \prime})} \, d\lambda \, d\lambda^{\prime} \, dz \\ 
& \quad - \int_{\mathbb{R}^{n_1 + n_2} \times \mathbb{R}^{n_2} \times \mathbb{R}^{n_2}} \sum_{\mu, {\mu^{\prime}}} \frac{|z^{\prime}|^2 \Delta_{z^{\prime \prime}} \left(\overline{m^S_l(z, \mu, \lambda)} \, m^S_{l^\prime}(z, {\mu^{\prime}}, \lambda^{\prime})\right)}{\left( (2|\mu|+n_1)|\lambda| - (2|{\mu^{\prime}}| + n_1)|\lambda^{\prime}| \right)} \Phi_\mu^{\lambda}(z^{\prime}) \Phi_{\mu^{\prime}}^{\lambda^{\prime}}(z^{\prime}) \\
& \quad \quad \quad \quad \quad \quad \quad \quad \quad \quad \quad \Phi_\mu^{\lambda}(x^{\prime}) \Phi_{\mu^{\prime}}^{\lambda^{\prime}}(y^{\prime}) e^{i \lambda \cdot (z^{\prime \prime} - x^{\prime \prime})} e^{- i \lambda^{\prime} \cdot (z^{\prime \prime} - y^{\prime \prime})} \, d\lambda \, d\lambda^{\prime} \, dz \\ 
& \quad + 2 i \sum_{j_1=1}^{n_1} \sum_{j_2=1}^{n_2} \int_{\mathbb{R}^{n_1 + n_2} \times \mathbb{R}^{n_2} \times \mathbb{R}^{n_2}} \frac{\lambda_{j_2}^{\prime}}{|\lambda^{\prime}|} \sum_{\mu, {\mu^{\prime}}} \frac{z_{j_1}^{\prime} \partial_{z_{j_2}^{\prime \prime}} \left( \overline{m^S_l(z, \mu, \lambda)} \, m^S_{l^\prime}(z, {\mu^{\prime}}, \lambda^{\prime}) \right)} {\left( (2|\mu|+n_1)|\lambda| - (2|{\mu^{\prime}}| + n_1)|\lambda^{\prime}| \right)} \Phi_\mu^{\lambda}(z^{\prime}) \\
& \quad \quad \quad \quad \quad \quad \quad \quad \quad \quad \left\{|\lambda^{\prime}| z_{j_1}^{\prime} \Phi_{\mu^{\prime}}^{\lambda^{\prime}}(z^{\prime})\right\} \Phi_\mu^{\lambda}(x^{\prime}) \Phi_{\mu^{\prime}}^{\lambda^{\prime}}(y^{\prime}) e^{i \lambda \cdot (z^{\prime \prime} - x^{\prime \prime})} e^{- i \lambda^{\prime} \cdot (z^{\prime \prime} - y^{\prime \prime})} \, d\lambda \, d\lambda^{\prime} \, dz.
\end{align*} 

Clearly, all the the terms in the right hand side of the above expression are of the form 
\begin{align*} 
& \int_{\mathbb{R}^{n_1 + n_2} \times \mathbb{R}^{n_2} \times \mathbb{R}^{n_2}} \Theta_0(\lambda^{\prime}) \sum_{\mu, {\mu^{\prime}}} m^{S, \Gamma_1, \Gamma_2, L}_{l, l^{\prime}}(z, \mu, {\mu^{\prime}}, \lambda, \lambda^{\prime}) \Phi_\mu^{\lambda}(z^{\prime}) \Phi_{{\mu}^{\prime} + \vec{c}}^{\lambda^{\prime}}(z^{\prime}) \Phi_\mu^{\lambda}(x^{\prime}) \Phi_{\mu^{\prime}}^{\lambda^{\prime}}(y^{\prime}) \\ 
& \quad \quad \quad \quad \quad \quad \quad \quad \quad \quad \quad e^{i \lambda \cdot (z^{\prime \prime} - x^{\prime \prime})} e^{- i \lambda^{\prime} \cdot (z^{\prime \prime} - y^{\prime \prime})} \, d\lambda \, d\lambda^{\prime} \, dz,
\end{align*} 
where 
$$m^{S, \Gamma_1, \Gamma_2, L}_{l, l^{\prime}}(z, \mu, {\mu^{\prime}}, \lambda, \lambda^{\prime}) = \frac{\overline{X^{\Gamma_1} m^S_l(z, \mu, \lambda)} X^{\Gamma_2} \, m^S_{l^{\prime}}(z, {\mu^{\prime}}, \lambda^{\prime}) P_{L}({\mu^{\prime}}, \lambda^{\prime})} {\left( (2|\mu|+n_1)|\lambda| - (2|{\mu^{\prime}}| + n_1)|\lambda^{\prime}| \right)}$$
with $|\Gamma_1| + |\Gamma_2| + L = 2$ and $L \leq 1$. Here $P_0({\mu^{\prime}}, \lambda^{\prime})$ is a constant and $P_1({\mu^{\prime}}, \lambda^{\prime})$ is either $((2\mu^{\prime}_{j_1}) |\lambda^{\prime}|)^{1/2}$ or $((2\mu^{\prime}_{j_1} + 2) |\lambda^{\prime}|)^{1/2}$. Also, $\vec{c} = \vec{0}$ if $L = 0$, whereas $\vec{c} = \pm e_{j_1} $ for $L = 1$ for some $1 \leq j_1 \leq n_1$, and $\Theta_0 (\lambda^{\prime})$ is either a constant function or $ \lambda_{j_2}^{\prime} / |\lambda^{\prime}|$. 

One can repeat the above process on each term by multiplying and dividing in the integrand by $(2|\mu| + n_1) |\lambda| - (2|{\mu}^{\prime} + \vec{c}| + n_1) |\lambda^{\prime}|$, and then it is obvious that $P_{L}({\mu^{\prime}}, \lambda^{\prime})$ can be expressed as $\mathcal{A}(\mu^{\prime}) \left( (2|\mu^{\prime}| + n_1) |\lambda^{\prime}| \right)^{L/2}$. Continuing this way, one could verify that $K^S_{l, l^\prime}(x,y)$ can be written as a finite linear sum of terms as claimed in the statement of Lemma \ref{lem:rho-delta-frequency-decay}. 
\end{proof}


Now, it suffices to estimate an operator corresponding to one such term appearing in Lemma \ref{lem:rho-delta-frequency-decay}. Let us denote its kernel by $\widetilde{K}^{S, \Gamma_1, \Gamma_2, L, \Theta_0, \mathcal{A}}_{l, l^{\prime}}$. 

Define a family of operators $S_{\vec{c}}^{\lambda^{\prime}}$ by $S_{\vec{c}}^{\lambda^{\prime}} \Phi_{{\mu^{\prime}}}^{\lambda^{\prime}} = \Phi_{{\mu^{\prime}} + \vec{c}}^{\lambda^{\prime}}$, and  consider the operator $f \to \widetilde{f}$ defined by 
\begin{align*}
\left(\widetilde{f}_{\Theta_0}^{\lambda^{\prime}}, \Phi_{{\mu^{\prime}}}^{\lambda^{\prime}}\right) & = \Theta_0(\lambda^{\prime}) \mathcal{A}(\mu^{\prime}) \left(\left(S_{\vec{c}}^{\lambda^{\prime}}\right)^* f^{\lambda^{\prime}}, \Phi_{{\mu^{\prime}}}^{\lambda^{\prime}}\right). 
\end{align*} 
It follows from the Plancherel theorem that the operator $f \to \widetilde{f}$ is $L^2$-bounded with bound depending only on $\Theta_0$ and $\mathcal{A}$. Notice also that  
$$\int_{\mathbb{R}^{n_1 + n_2}} \widetilde{K}^{S, \Gamma_1, \Gamma_2, L, \Theta_0, \mathcal{A}}_{l, l^{\prime}} (x,y) f(y) \, dy = \int_{\mathbb{R}^{n_1 + n_2}} K^{S, \Gamma_1, \Gamma_2, L}_{l, l^{\prime}} (x,y) \widetilde{f} (y) \, dy,$$
where 
\begin{align*} 
& K^{S, \Gamma_1, \Gamma_2, L}_{l, l^{\prime}}(x,y) \\ 
& = (2 \pi)^{-n_2} \int_{\mathbb{R}^{n_1 + n_2} \times \mathbb{R}^{n_2} \times \mathbb{R}^{n_2}} \sum_{\mu, {\mu^{\prime}}} m^{S, \Gamma_1, \Gamma_2, L}_{l, l^{\prime}}(z, \mu, {\mu^{\prime}} + \vec{c}, \lambda, \lambda^{\prime}) \overline{E_{\mu, \lambda} (z, x)} E_{\mu^{\prime}, \lambda^{\prime}} (z, y) \, d\lambda \, d\lambda^{\prime} \, dz, 
\end{align*} 
and therefore in view of Schur's lemma it is sufficient to estimate 
\begin{align} \label{final-schur} 
& \int_{\mathbb{R}^{n_1 + n_2} \times \mathbb{R}^{n_1 + n_2}} \left| \int_{\mathbb{R}^{n_2} \times \mathbb{R}^{n_2}} \sum_{\mu, {\mu^{\prime}}} m^{S, \Gamma_1, \Gamma_2, L}_{l, l^{\prime}}(z, \mu, {\mu^{\prime}} + \vec{c}, \lambda, \lambda^{\prime}) \overline{E_{\mu, \lambda} (z, x)} E_{\mu^{\prime}, \lambda^{\prime}} (z, y) \, d\lambda \, d\lambda^{\prime} \right| dz \, dx, \\ 
\nonumber & \textup{and} \quad \\ 
\nonumber & \int_{\mathbb{R}^{n_1 + n_2} \times \mathbb{R}^{n_1 + n_2}} \left| \int_{\mathbb{R}^{n_2} \times \mathbb{R}^{n_2}} \sum_{\mu, {\mu^{\prime}}} m^{S, \Gamma_1, \Gamma_2, L}_{l, l^{\prime}}(z, \mu, {\mu^{\prime}} + \vec{c}, \lambda, \lambda^{\prime}) \overline{E_{\mu, \lambda} (z, x)} E_{\mu^{\prime}, \lambda^{\prime}} (z, y) \, d\lambda \, d\lambda^{\prime} \right| dz \, dy. 
\end{align} 

A usual way to achieve this is via estimating 
\begin{align} \label{eq:pointwise-weighted-1}
& (1 + d(x,z))^{Q+1} (1 + d(y,z))^{Q+1} \left|\int_{\mathbb{R}^{n_2} \times \mathbb{R}^{n_2}} \sum_{\mu, {\mu^{\prime}}} m^{S, \Gamma_1, \Gamma_2, L}_{l, l^{\prime}}(z, \mu, {\mu^{\prime}}+\vec{c}, \lambda, \lambda^{\prime}) \right. \\ 
\nonumber & \quad \quad \quad \quad \quad \quad \quad \quad \quad \quad \quad \quad \quad \quad \quad \quad \quad \quad \quad \quad \left. \overline{E_{\mu, \lambda} (z, x)} E_{\mu^{\prime}, \lambda^{\prime}} (z, y) \, d\lambda \, d\lambda^{\prime}\right|. 
\end{align}

Now, to do so, one would like to use Lemma \ref{lem:weighted-plancherel-joint-functional}. While this is applicable in $\tau = (2 \mu + \tilde{1}) |\lambda|$ variable, we encounter the obstacle of shifts occurring in $\mu^\prime$ variable and we don't know if we can still apply Lemma \ref{lem:weighted-plancherel-joint-functional} in the presence of shifts (remember that Lemma \ref{lem:weighted-plancherel-joint-functional} is applicable on spectral multipliers). On the other hand if we try to use Proposition \ref{prop:lambda-diff-theorem} directly then we get singularities of $\lambda^\prime$-variable in the form of high negative inverse powers of $|\lambda^\prime|$, and this is something we do not know how to avoid. However, we find that we can make use of the mix of both of the ideas by taking a Taylor series expansion of the multipliers and that way we can get our desired result. We explain it below. 

For computational convenience, we will assume $C_q=0$ for all $q$ in \eqref{eq:1-p1} and prove the estimates on $L^2$-norm of the operator associated to that kernel. Shortly we shall explain how one can handle the case of $C_q \neq 0$. For fixed $\mu$ and $\lambda$, by abuse of notation, let us define 

$$ F(\mu^\prime, \lambda^\prime) =  X^{\Gamma_2} \, m^S_{l^{\prime}}(z, (2 \mu^{\prime} + \tilde{1})| \lambda^{\prime}|, - \lambda^{\prime}).$$

Then using the Taylor series expansion we have
\begin{align} \label{eq:Taylor-expansion-discrete-variable}
F(\mu^\prime + \vec{c}, \lambda^\prime) = \sum_{|\alpha| \leq k} D_{\mu}^\alpha F(\mu^\prime, \lambda^\prime) \frac{\vec{c}^{\alpha}}{\alpha!} + \sum_{|\alpha| = k+1} \frac{k+1}{\alpha!} \int_0^1 (1-t)^{k} D_{\mu}^\alpha F(\mu^\prime + t \vec{c}, \lambda^\prime) \, dt, 
\end{align} 
with an appropriate $k \in \mathbb{N}$ that we shall choose later. 

Note that for any admissible $\vec{c}\in \Z^{n}$ (so that $\mu^\prime + \vec{c} \in \mathbb{N}^{n_1}$),
\begin{align*}
D^{\alpha}_{\mu} F(\mu^\prime + \vec{c}, \lambda^\prime) =& |\lambda^{\prime}|^{|\alpha|} \left( X^{\Gamma_2} \, \partial^{\alpha}_{\tau} m^S_{l^{\prime}}\right) (z, (2({\mu^{\prime}}+\vec{c}) + \tilde{1})| \lambda^{\prime}|, - \lambda^\prime). 
\end{align*}

Therefore, applying the above expansion of $F$ on $m^{S, \Gamma_1, \Gamma_2, L}_{l, l^{\prime}}(z, \mu, {\mu^{\prime}}, \lambda, \lambda^{\prime})$ we have the two terms which are
\begin{align*}
\widetilde{m^{S, \Gamma_1, \Gamma_2, L}_{l, l^{\prime}}}(z, \mu, {\mu^{\prime}}, \lambda, \lambda^{\prime}) 
& = \frac{\overline{X^{\Gamma_1} m^S_l(z, (2 \mu + \tilde{1}) |\lambda|, - \lambda)} \left( (2|{\mu^{\prime}}| + n_1)|\lambda^{\prime}| \right)^{L/2} } {\left( (2|\mu|+n_1)|\lambda| - (2|{\mu^{\prime}}| + n_1)|\lambda^{\prime}| \right)^N} \\ 
& \quad \sum_{|\alpha| \leq k} \frac{\vec{c}^{\alpha}}{\alpha!} |\lambda^{\prime}|^{|\alpha|} \left( X^{\Gamma_2} \, \partial_{\tau}^{\alpha} m^S_{l^{\prime}}\right) (z, (2 \mu^{\prime} + \tilde{1}) |\lambda^{\prime}|, - \lambda^{\prime}), 
\end{align*}
and 
\begin{align*}
Err^{S, \Gamma_1, \Gamma_2, L}_{l, l^{\prime}}(z, \mu, {\mu^{\prime}}, \lambda, \lambda^{\prime}) & =  \frac{\overline{X^{\Gamma_1} m^S_l(z, (2 \mu^{\prime} + \tilde{1}) |\lambda^{\prime}|, - \lambda^{\prime})} \left( (2|{\mu^{\prime}}| + n_1)|\lambda^{\prime}| \right)^{L/2}} {\left( (2|\mu|+n_1)|\lambda| - (2|{\mu^{\prime}}| + n_1)|\lambda^{\prime}| \right)^N} \\ 
& \quad \quad \sum_{|\alpha| = k+1} \frac{k+1}{\alpha!} \int_0^1 (1-t)^{|\alpha|-1}|\lambda^{\prime}|^{|\alpha|}\left( X^{\Gamma_2} \, \partial_{\tau}^{\alpha} m^S_{l^{\prime}}\right) (\mu^\prime + t \vec{c}, \lambda^\prime)dt.
\end{align*}

Let us remark that when $C_q \neq 0$, we can again apply Taylor's expansion on $C_q$ variable of each such term $\widetilde{m^{S, \Gamma_1, \Gamma_2, L}_{l, l^{\prime}}}(z, \mu, {\mu^{\prime}}, \lambda, \lambda^{\prime})$ so that the power of $|\lambda^{\prime}|$ in the $Err$ terms coming form each of these expansions become $k+1$. 

\medskip \noindent \underline{\textbf{Step 1 (Analysis of the operator associated to $\widetilde{m^{S, \Gamma_1, \Gamma_2, L}_{l, l^{\prime}}}(z, \mu, {\mu^{\prime}}, \lambda, \lambda^{\prime})$)}} : 

Note that the associated operator in this case can be written as a finite linear combination (on $|\alpha| \leq k$) of the composition of three operators $\zeta_S (\boldsymbol{U}) \circ Op(\widetilde{m^{\Gamma_1, \Gamma_2, L}_{l, l^{\prime}}}(z, \mu, {\mu^{\prime}}, \lambda, \lambda^{\prime})) \circ \zeta_S (\boldsymbol{U})$, with the kernel of $Op(\widetilde{m^{\Gamma_1, \Gamma_2, L}_{l, l^{\prime}}}(z, \mu, {\mu^{\prime}}, \lambda, \lambda^{\prime}))$ being 
\begin{align*}
& \int_{\mathbb{R}^{n_1 + n_2} \times \mathbb{R}^{n_2} \times \mathbb{R}^{n_2}} \sum_{\mu, {\mu^{\prime}}} \frac{\overline{X^{\Gamma_1} m_l(z, (2 \mu + \tilde{1}) |\lambda|, - \lambda)} \left( (2|{\mu^{\prime}}| + n_1)|\lambda^{\prime}| \right)^{L/2} } {\left( (2|\mu|+n_1)|\lambda| - (2|{\mu^{\prime}}| + n_1)|\lambda^{\prime}| \right)^N} \\ 
& \quad |\lambda^{\prime}|^{|\alpha|} \left( X^{\Gamma_2} \, \partial_{\tau}^{\alpha} m_{l^{\prime}}\right) (z, (2 \mu^{\prime} + \tilde{1}) |\lambda^{\prime}|, - \lambda^\prime) \overline{E_{\mu, \lambda} (z, x)} E_{\mu^{\prime}, \lambda^{\prime}} (z, y) \, d\lambda \, d\lambda^{\prime} \, dz. 
\end{align*}

Since $\zeta_S(\boldsymbol{U}) $ is $L^2$-bounded by Plancherel's theorem with bound uniform in $S$, we are left with analysing the kernel associated to $\widetilde{m^{\Gamma_1, \Gamma_2, L}_{l, l^{\prime}}}(z, \mu, {\mu^{\prime}}, \lambda, \lambda^{\prime})$. In this case we will use the weighted Plancherel estimate. First we rewrite the associated kernel 
$$ \int_{\mathbb{R}^{n_1 + n_2} \times \mathbb{R}^{n_2} \times \mathbb{R}^{n_2}}\sum_{\mu, {\mu^{\prime}}} \widetilde{m^{\Gamma_1, \Gamma_2, L}_{l, l^{\prime}}}(z, \mu, {\mu^{\prime}}, \lambda, \lambda^{\prime}) \overline{E_{\mu, \lambda} (z, x)} E_{\mu^{\prime}, \lambda^{\prime}} (z, y) \, d\lambda \, d\lambda^{\prime} \, dz $$
as a finite linear sum over $\sum_{|\alpha| \leq k}$, of the following terms: 
\begin{align*}
& \int_{\mathbb{R}^{n_1 + n_2} \times \mathbb{R}^{n_2} \times \mathbb{R}^{n_2}} \sum_{\mu, {\mu^{\prime}}} \mathcal{A}_{\alpha}(\mu^{\prime}) \overline{X^{\Gamma_1} m_l(z, (2 \mu + \tilde{1}) |\lambda|, - \lambda)} \left( X^{\Gamma_2} \, \partial_{\tau}^{\alpha} m_{l^{\prime}}\right) (z, (2 \mu^{\prime} + \tilde{1}) |\lambda^{\prime}|, - \lambda^{\prime}) \\ 
& \quad \left( (2|{\mu^{\prime}}| + n_1)|\lambda^{\prime}| \right)^{|\alpha| + \frac{L}{2} } \left\{ (2|\mu|+n_1)|\lambda| - (2|{\mu^{\prime}}| + n_1)|\lambda^{\prime}| \right\}^{-N} \overline{E_{\mu, \lambda} (z, x)} E_{\mu^{\prime}, \lambda^{\prime}} (z, y) \, d\lambda \, d\lambda^{\prime} \, dz, 
\end{align*}
where $\mathcal{A}_{\alpha}(\mu^{\prime}) = (2|\mu^{\prime}|+n_1)^{-|\alpha|}$.

We estimate each such term separately. As discussed earlier, we can get rid of $\mathcal{A}_{\alpha}(\mu^{\prime})$ by realising the corresponding operator as a composition of two operators. Therefore, it suffices to analyse the kernel
\begin{align} \label{eq:pointwise-weighted-2}
& \left( (1 + d(x,z)) (1 + d(y,z)) \right)^{Q+1} \\ 
& \nonumber \quad \left| \int_{\mathbb{R}^{n_2}} \sum_{\mu} \left\{ \overline{X^{\Gamma_1} m_l(z, (2 \mu + \tilde{1}) |\lambda|, -\lambda)} M^{\Gamma_2, L}_{l^{\prime}} (z, y, \mu, \mu^{\prime}, \lambda, \lambda^{\prime}) \right\} \overline{E_{\mu, \lambda} (z, x)} \, d\lambda \right|, 
\end{align}
where 
\begin{align*}
M^{\Gamma_2, L}_{l^{\prime}}(z, y, \mu, \mu^{\prime}, \lambda, \lambda^{\prime}) & = \int_{\mathbb{R}^{n_2}} \sum_{\mu^{\prime}} \left\{ \left( X^{\Gamma_2} \, \partial_{\tau}^{\alpha} m_{l^{\prime}}\right) (z, (2 \mu^{\prime} + \tilde{1}) |\lambda^{\prime}|, - \lambda^{\prime}) \right\} \left( (2|{\mu^{\prime}}| + n_1)|\lambda^{\prime}| \right)^{|\alpha| + \frac{L}{2}} \\ 
& \quad \quad \quad \quad \quad \quad \left\{ (2|\mu|+n_1)|\lambda| - (2|{\mu^{\prime}}| + n_1)|\lambda^{\prime}| \right\}^{-N} E_{\mu^{\prime}, \lambda^{\prime}} (z, y) \, d\lambda^{\prime}.  
\end{align*}

We can now make use of Lemma \ref{lem:weighted-plancherel-joint-functional} (with $R = 2^{l/2}, r = Q+1,$ and $s = Q+2$) to estimate \eqref{eq:pointwise-weighted-2}. In doing so, note once again that 
\begin{align*} 
& \left( (1 + d(x,z)) (1 + d(y,z)) \right)^{Q+1} \\ 
& \quad \left| \int_{\mathbb{R}^{n_2}} \sum_{\mu} \left\{ \overline{X^{\Gamma_1} m_l(z, (2 \mu + \tilde{1}) |\lambda|, -\lambda)} M^{\Gamma_2, L}_{l^{\prime}}(z, y, \mu, \mu^{\prime}, \lambda, \lambda^{\prime}) \right\} \overline{E_{\mu, \lambda} (z, x)} \, d\lambda \right| \\ 
& \leq \sup_{z_0 \in \mathbb{R}^{n_1 + n_2}} (1 +2^{l/2} d(x,z))^{Q+1} (1 + 2^{l^{\prime}/2}d(y,z))^{Q+1} \\ 
& \quad \quad \quad \quad \left| \int_{\mathbb{R}^{n_2}} \sum_{\mu} \left\{ \overline{X^{\Gamma_1} m_l(z_0, (2 \mu + \tilde{1})|\lambda|, -\lambda)} M^{\Gamma_2, L}_{l^{\prime}}(z_0, y, \mu, \mu^{\prime}, \lambda, \lambda^{\prime}) \right\} \overline{E_{\mu, \lambda} (z, x)} \, d\lambda \right|, 
\end{align*}
which justifies the applicability of Lemma \ref{lem:weighted-plancherel-joint-functional}. Now, we remark that since $\rho \leq 1$, the situation where $\tau$-derivatives falls on $\left\{ (2|\mu| + n_1)|\lambda| - (2|{\mu^{\prime}}| + n_1)|\lambda^{\prime}| \right\}^{-N} $ is better compared to the one where derivatives fall on the symbol functions $X^{\Gamma_1} m_l$. So, in the follow up calculations, we just take the case when the $\tau$-derivatives fall on the symbol functions $X^{\Gamma_1} m_l$. 

Now, on $(1 + 2^{l^{\prime}/2}d(y,z))^{Q+1} \left| M^{\Gamma_2, L}_{l^{\prime}}(z_0, y, \mu, \mu^{\prime}, \lambda, \lambda^{\prime}) \right| $, we apply Lemma \ref{lem:weighted-plancherel-joint-functional} with $R = 2^{l^{\prime}/2}, r = Q+1, s = Q+2$, and once again we ignore the case of $\tau$-derivatives falling on $P_L ({\mu^{\prime}}, \lambda^{\prime})$ or $\left\{ (2|\mu|+n_1)|\lambda| - (2|{\mu^{\prime}}| + n_1)|\lambda^{\prime}| \right\}^{-N}$. Furthermore, as the worst possible case, take $L = N$ in the order of $P_L ({\mu^{\prime}}, \lambda^{\prime})$. Overall, we get that \eqref{eq:pointwise-weighted-1} is dominated by 
\begin{align*}
& |B(x, 2^{-l/2})|^{-1/2} |B(z, 2^{-l/2})|^{-1/2} |B(y, 2^{-l^{\prime}/2})|^{-1/2} |B(z, 2^{-l^{\prime}/2})|^{-1/2} \\ 
& \quad \times 2^{l^{\prime} \left( |\alpha| + \frac{N}{2} \right)} 2^{- N \max \{l, l^{\prime}\}} \left\| \left( X^{\Gamma_1} m \right) (z, 2^l \cdot) \right\|_{W^\infty_{Q+2}} \left\| \left( X^{\Gamma_2} \, \partial_{\tau}^{\alpha} m \right) (z, 2^{l^{\prime}} \cdot)\right\|_{W^\infty_{Q+2}} \\ 
\lesssim & |B(x, 2^{-l/2})|^{-1/2} |B(z, 2^{-l/2})|^{-1/2} |B(y, 2^{-l^{\prime}/2})|^{-1/2} |B(z, 2^{-l^{\prime}/2})|^{-1/2} \\ 
& \quad \times 2^{l^{\prime} \left( |\alpha| + \frac{N}{2} \right)} 2^{- N \max \{l, l^{\prime}\}} \left\{ \left( 2^l \right)^{\delta |\Gamma_1|/2}  2^{l(Q+2)} \right\} \left\{ \left( 2^{l^{\prime}} \right)^{- (1 + \rho) \frac{|\alpha|}{2} + \delta \frac{|\Gamma_2|}{2}}  2^{l^{\prime} (Q+2)} \right\} \\ 
\lesssim & |B(x, 1)|^{-1/2} |B(y, 1)|^{-1/2} |B(z, 1)|^{-1} 2^{-N(1 - \delta) \max \{l, l^{\prime}\} /2} \, 2^{(1 - \rho) k l^{\prime}/2} \, 2^{(3Q+4) \left(l + l^{\prime} \right)/2},  
\end{align*}
where we have used the conditions $|\alpha| \leq k$  and $|\Gamma_1| + |\Gamma_2| = N$ (since $|\Gamma_1| + |\Gamma_2| + L = 2N$ and we have chosen the worst possible scenario of $L=N$). 

Since $\delta < 1$ and $l > l^{\prime}$, once $k$ is fixed one can choose $N$ large enough such that above quantity is bounded by $|B(x, 1)|^{-1/2} |B(y, 1)|^{-1/2} |B(z, 1)|^{-1} 2^{- \max \{l, l^{\prime}\}} $. 

\medskip \noindent \underline{\textbf{Step 2 (Analysis of the operator associated to $Err^{S, \Gamma_1, \Gamma_2, L}_{l, l^{\prime}}(z, \mu, {\mu^{\prime}}, \lambda, \lambda^{\prime})$)}} : 

For technical convenience, we do not carry the integration in $t$-variable over the interval $[0,1]$ as we shall see that estimates that we get are uniform over $[0,1]$. So, we have to analyse 
\begin{align} \label{eq:pointwise-weighted-3}
& \left( (1 + d(x,z)) (1 + d(y,z)) \right)^{Q+1} \\ 
\nonumber & \quad \left| \int_{\mathbb{R}^{n_2}} \sum_{\mu} \left\{ \overline{X^{\Gamma_1} m^S_l(z, \mu, \lambda)} Err^{S, \Gamma_2, L}_{l^{\prime}}(z, y, \mu, \mu^{\prime}, \lambda, \lambda^{\prime}) \right\} \overline{E_{\mu, \lambda} (z, x)} \, d\lambda \right|, 
\end{align}
where 
\begin{align*} 
Err^{S, \Gamma_2, L}_{l^{\prime}}(z, y, \mu, \mu^{\prime}, \lambda, \lambda^{\prime}) & =  \int_{\mathbb{R}^{n_2}} |\lambda^{\prime}|^{k+1} \sum_{\mu^{\prime}} \left( X^{\Gamma_2} \, \partial_{\tau}^{\alpha} m^S_{l^{\prime}} \right) (\mu^\prime + t \vec{c}, - \lambda^\prime) \left\{ (2|{\mu^{\prime}}| + n_1)|\lambda^{\prime}| \right\}^{L/2} \\ 
\nonumber & \quad \quad \quad \quad \quad \quad \left\{ (2|\mu|+n_1)|\lambda| - (2|{\mu^{\prime}}| + n_1)|\lambda^{\prime}| \right\}^{-N} E_{\mu^{\prime}, \lambda^{\prime}} (z, y) \, d\lambda^{\prime}, 
\end{align*}
with $|\alpha| = k+1$. 

We can estimate \eqref{eq:pointwise-weighted-3} same as that done for \eqref{eq:pointwise-weighted-2}, with the only exception being that for estimating the kernel $Err^{S, \Gamma_2, L}_{l^{\prime}}$ (in $\tau^{\prime}$-variable), we make use of Lemma \ref{weighted-kernel-estimate-3} instead of directly applying Lemma \ref{lem:weighted-plancherel-joint-functional}. For this, we need $k \geq Q$ to be such that $k+1$ is a multiple of 4. With $N_0 = \lfloor\frac{Q}{4}\rfloor + 1$, take and fix $k = 4 N_0 - 1$ in \eqref{eq:Taylor-expansion-discrete-variable}. With this choice of $k$, we can essentially repeat the steps of the proof of Theorem \ref{thm:weighted-Plancherel-L2} (for $p = \infty$ and $R = 2^{l^{\prime}/2}$) with the help of estimates of Lemma \ref{weighted-kernel-estimate-3} to get that \eqref{eq:pointwise-weighted-3} is dominated by 
\begin{align*} 
& |B(x, 2^{-l/2})|^{-1/2} |B(z, 2^{-l/2})|^{-1/2} |B(y, 2^{-l^{\prime}/2})|^{-1/2} |B(z, 2^{-l^{\prime}/2})|^{-1/2} \\ 
& \quad \times 2^{l^{\prime} \left( 4 N_0 + \frac{N}{2} \right)} 2^{- N \max \{l, l^{\prime}\}} \left\| \left( X^{\Gamma_1} m \right) (z, 2^l \cdot) \right\|_{W^2_{4N_0+1}} \left\| \left( X^{\Gamma_2} \, \partial_{\tau}^{\beta} m \right) (z, 2^{l^{\prime}} \cdot)\right\|_{W^2_{4N_0+1}} \\ 
\lesssim & |B(x, 1)|^{-1/2} |B(y, 1)|^{-1/2} |B(z, 1)|^{-1} 2^{-N(1 - \delta) \max \{l, l^{\prime}\} /2} \, 2^{2(1 - \rho) N_0 l^{\prime}} \, 2^{(Q+8N_0+4) \left(l + l^{\prime} \right)/2}. 
\end{align*}

As earlier, since $\delta < 1$ and $l > l^{\prime}$, with $N_0$ already fixed, we can choose $N$ large enough such that above quantity is bounded by $|B(x, 1)|^{-1/2} |B(y, 1)|^{-1/2} |B(z, 1)|^{-1} 2^{- \max \{l, l^{\prime}\}} $. 

For these estimations to work, once we know the choice of $N$ that works for us in the above estimations, we fix such an $N$ and then we fix a natural number $\mathfrak{z}$ such that it satisfies $N \leq C_{n_2} 2^{\mathfrak{z}-5}$ so that Lemma \ref{lem:rho-delta-frequency-decay} is applicable. With the above estimates, using Lemma \ref{lem-integrability-ball-growth} the operator norm in question can be shown to be bounded by $2^{- \max \{l, l^{\prime}\}}$. Hence, Theorem \ref{thm:rho-delta} follows from Theorem \ref{Cotlar-Stein-lemma-2}. 



\section*{Acknowledgements}
This work was supported in part by the individual INSPIRE Faculty Awards of both of the authors from the Department of Science and Technology (DST), Government of India. 

We are thankful to Sundaram Thangavelu for several discussions and his suggestions on Sobolev embedding theorem, and Loukas Grafakos for sharing his insight on homogeneous symbols. We are indebted to Alessio Martini for guiding us on our query on the weighted Plancherel estimates for the joint functional calculus. In the later half of the development of this project, we had several wonderful discussions with Abhishek Ghosh and Riju Basak. These discussions provided us a lot of insight resulting in the write-up in its present form. We are very thankful to Abhishek and Riju for the same. 

\providecommand{\bysame}{\leavevmode\hbox to3em{\hrulefill}\thinspace}
\providecommand{\MR}{\relax\ifhmode\unskip\space\fi MR }
\providecommand{\MRhref}[2]{%
  \href{http://www.ams.org/mathscinet-getitem?mr=#1}{#2}
}
\providecommand{\href}[2]{#2}

\end{document}